\documentclass{ieeetj}
\usepackage{cite}
\usepackage{amsmath,amssymb,amsfonts}
\usepackage{algorithmic}
\usepackage{graphicx,color}
\usepackage{textcomp}
\usepackage{xcolor}
\usepackage{hyperref}
\hypersetup{hidelinks=true}
\usepackage{algorithm,algorithmic}
\usepackage{amsmath,amsfonts}
\usepackage{array}
\usepackage[caption=false,font=normalsize,labelfont=sf,textfont=sf]{subfig}
\usepackage{textcomp}
\usepackage{stfloats}
\usepackage{url}
\usepackage{verbatim}
\usepackage{graphicx}
\usepackage{cite}
\usepackage{graphicx}
\usepackage{amsmath}
\usepackage{amsfonts}
\usepackage{amsthm}
\usepackage{float}
\usepackage[dvipsnames]{xcolor}

\newtheorem{theorem}{Theorem}

\DeclareMathOperator*{\trace}{Tr}

\newcommand{\reals}{{\mathbb R}}

\newcommand{\inner}[2]{\langle #1, #2 \rangle}

\def\BibTeX{{\rm B\kern-.05em{\sc i\kern-.025em b}\kern-.08em
    T\kern-.1667em\lower.7ex\hbox{E}\kern-.125emX}}
\AtBeginDocument{\definecolor{tmlcncolor}{cmyk}{0.93,0.59,0.15,0.02}\definecolor{NavyBlue}{RGB}{0,86,125}}

\def\OJlogo{\vspace{-4pt}$<$Society logo(s) and publication title will appear here.$>$}
\def\seclogo{\vspace{10pt}$<$Society logo(s) and publication title will appear here.$>$}
\def\OJlogo{}
\def\seclogo{}

\def\authorrefmark#1{\ensuremath{^{\textbf{#1}}}}

\usepackage{ifthen}

\newcommand\revise[1]{\textcolor{black}{#1}}

\begin{document}
\bstctlcite{IEEEexample:BSTcontrol}

\newboolean{arxiv}   
\setboolean{arxiv}{True}
\ifthenelse{\NOT \boolean{arxiv}}{
}{}

\markboth{RJD-BASE: Multi-Modal Spectral Clustering via Randomized Joint Diagonalization}{He, Pados, and Kressner}

\title{RJD-BASE: Multi-Modal Spectral Clustering via Randomized Joint Diagonalization}

\author{Haoze He\authorrefmark{1}, Artemis A. Pados\authorrefmark{2},\\ and Daniel Kressner \authorrefmark{1}}
\affil{\'Ecole Polytechnique F\'ed\'erale de Lausanne (EPFL), Institute of Mathematics, 1015 Lausanne, Switzerland.}
\affil{Massachusetts Institute of Technology, Department of Electrical Engineering and Computer Science, 
Cambridge, Massachusetts, USA}
\corresp{Corresponding author: Haoze He (email: haoze.he@epfl.ch).}

\begin{abstract}
We revisit the problem of spectral clustering in multimodal settings, where each data modality is encoded as a graph Laplacian. While classical approaches - including joint diagonalization, spectral co-regularization, and multiview clustering - attempt to align embeddings across modalities, they often rely on costly iterative refinement and may fail to directly target the spectral subspace relevant for clustering. In this work, we introduce two key innovations. First, we bring the power of randomization to this setting by sampling random convex combinations of Laplacians as a simple and scalable alternative to explicit eigenspace alignment. Second, we propose a principled selection rule based on Bottom-$k$ Aggregated Spectral Energy (BASE) - a $k$-dimensional extension of the directional smoothness objective from recent minimax formulations - which we uniquely apply as a selection mechanism rather than an optimization target. The result is \textbf{Randomized Joint Diagonalization with BASE Selection (RJD-BASE)}, a method that is easily implementable, computationally efficient, aligned with the clustering objective, and grounded in decades of progress in standard eigensolvers. Through experiments on synthetic and real-world datasets, we show that RJD-BASE reliably selects high-quality embeddings, outperforming classical multimodal clustering methods at low computational cost.
\end{abstract}

\begin{IEEEkeywords}
Graph Laplacian, joint diagonalization, multimodal learning, randomized numerical linear algebra, spectral clustering 
\end{IEEEkeywords}

\maketitle

\section{INTRODUCTION}

\IEEEPARstart{S}{pectral} clustering is a widely used technique for discovering latent group structure in data. It constructs a graph from pairwise similarities between data points and uses the eigenvectors of the graph Laplacians to embed data for clustering. Graph representations are particularly effective because they naturally encode pairwise relationships, thereby transforming the clustering task into a graph partitioning problem. In multimodal settings---where each data modality provides a different perspective on the same set of samples---it is natural to represent each modality with its own graph and seek a shared low-dimensional embedding that captures the common latent structure \cite{von2007tutorial}.

A central challenge in this setting is how to aggregate the modality-specific Laplacians in a way that retains the most informative directions for clustering. Classical approaches---including joint diagonalization, spectral co-regularization, and multiview clustering---aim to align spectral representations across modalities, but often do so by solving non-convex optimization problems and iteratively updating. These procedures can be computationally expensive and may fail to directly target the spectral subspace that underlies clustering---namely, the subspace spanned by the bottom-$k$ eigenvectors associated with the $k$ smallest nonzero eigenvalues \cite{JADE,ablin:hal-01936887,perry2021mvlearn, kumarconf, kumar2011coreg, bickel2004icdm}. Recent work~\cite{coifman2023common} has proposed alternatives based on the single-directional smoothness of graph Laplacians, culminating in selecting a convex combination of Laplacians by maximizing the smallest nonzero eigenvalue. While this approach may align more directly with the objectives of spectral clustering, it captures only a single direction at a time.

In this work, we make two key contributions. First, we introduce randomization into the setting of multimodal spectral clustering by sampling random convex combinations of Laplacians. Second, we extend the single-directional smoothness framework to a $k$-dimensional formulation that evaluates the aggregated spectral energy of the eigenvectors associated with the $k$ smallest nonzero eigenvalues. We use this \textbf{Bottom-$k$ Aggregated Spectral Energy (BASE)} objective not as a target for optimization, but as a principled selection criterion among random samples.
The combination of these two techniques results in \textbf{Randomized Joint Diagonalization with BASE Selection (RJD-BASE)}.  RJD-BASE leverages decades of advances in efficient standard eigensolvers~\cite{laug,Bai2000}. It is parallelizable and scalable, requiring no optimization or initialization, and consistently delivers high-quality embeddings. We benchmark its performance against a wide range of classical multimodal clustering methods and show that RJD-BASE outperforms these techniques while simultaneously running at a low computational cost.

\textit{Notation:} $(\cdot)^\top$,  $\trace(\cdot)$ and $\|\cdot\|_F$ denote transpose,  trace and Frobenius norm of matrices, respectively. Bold upper-case letters (e.g., $\mathbf{A}$) denote matrices or matrix-valued functions, bold lower-case letters (e.g., $\mathbf{x}$) denote column vectors or vector-valued functions, and italic letters (e.g., $k$, $f$) denote scalars or scalar-valued functions. The $(i,j)$th entry of a matrix $\mathbf{A}$ is denoted by $a_{ij}$, while its $i$th column is denoted by $\mathbf{a}_i$. The $i$th entry of a vector $\mathbf{a}$ is denoted by $a_i$.  
 $\mathbf{I}_N$ is the $N \times N$ identity matrix, $\mathbf{0}$ and $\mathbf{1}$ denote  the vector/matrix of \revise{all zeros or} all ones, respectively. 
Finally, we will make use of the $m$-dimensional standard simplex
\begin{equation} \label{eq:simplex} 
 \Delta^{m-1} = \{\mathbf{x}\in\reals^m: \sum_i x_i =1, x_i\geq0\}.
\end{equation}

\section{BACKGROUND AND RELATED WORK}
\label{background}

We first recall the standard (single-modality) spectral clustering setup. Let \(\mathbf{W}\in\mathbb{R}^{N\times N}\) be a symmetric, nonnegative adjacency matrix encoding a \emph{connected}, \emph{undirected} weighted graph on \(N\) nodes (data points). We interpret every entry \(w_{pq} = w_{qp} \ge 0\) as pairwise similarity between nodes \(p\) and \(q\) in the  graph and we set \(w_{pp}=0\). In practice, \(\mathbf{W}\) may be formed from data features via a symmetric similarity function \(s(p,q)\) (e.g., an RBF kernel), or acquired from observed weighted edges. Throughout, we use the symmetric normalized Laplacian \cite{spectralgraphtheory}
\[
\mathbf{L}\;=\;\mathbf{I}_N-\mathbf{D}^{-1/2}\mathbf{W}\mathbf{D}^{-1/2}
\]
where the diagonal degree matrix $\mathbf{D}$ is defined by its diagonal entries $D_{pp} = \sum_q w_{pq}$ for $p = 1,\ldots,N$. 
Note that \(\mathbf{L}\) is symmetric and positive semidefinite by construction.

It is well known that the smallest eigenvalue \(\lambda_0\) of $\mathbf{L}$ is always zero and, because the graph is connected, its second smallest eigenvalue $\lambda_1$ is positive~\cite{spectralgraphtheory}. More generally, we sort the eigenvalues of $\mathbf{L}$ in increasing order,
\[
 0=\lambda_0<\lambda_1\le\lambda_2 \le \cdots\le\lambda_{N-1},
\]
$\mathbf{x}_0,\mathbf{x}_1,\ldots,\mathbf{x}_{N-1}$ denote the corresponding eigenvectors .
In particular, we refer to $\lambda_1,\ldots,\lambda_k$ as the \emph{bottom-$k$} eigenvalues and $\mathbf{x}_1,\ldots,\mathbf{x}_k$ as the bottom-$k$ eigenvectors.

For a target of \(k\) clusters, we form the embedding matrix
\[
\mathbf{X}
\;=\;
\big[\mathbf{x}_1,\ldots,\mathbf{x}_k\big]\in\mathbb{R}^{N\times k}.
\]
Each row of \(\mathbf{X}\) is the embedded representation of a data point. Typically, the \(k\)-means algorithm, which is guaranteed to converge to local optima, is applied to the rows of $\mathbf{X}$ to obtain cluster assignments \cite{von2007tutorial}.  More recently, deep clustering frameworks have emerged to enhance spectral methods by jointly optimizing feature representation and cluster assignment \cite{deeplearning1, deeplearning2} \revise{\cite{liu2024aevc, yacobi2025specrage}}.

In the fully coupled multimodal case, we consider a family of Laplacians \(\mathbf{L}_1,\ldots,\mathbf{L}_m\) on the same \(N\) data points (in matching order). Our goal is to obtain a shared embedding \(\mathbf{X}\in\mathbb{R}^{N\times k}\) whose columns align with the clustering-relevant subspaces across modalities. 

To contextualize our proposed approach, we will first review several classical multimodal clustering methods that aim at aligning eigenspaces across modalities. These methods will serve as the baselines in our experimental comparisons.

\subsection{Joint Diagonalization-Based Multimodal Spectral Clustering}
\label{jd_background}

A natural approach to multimodal spectral clustering is to align the eigenspaces of the modality-specific Laplacians $\mathbf{L}_1, \ldots, \mathbf{L}_m \in \mathbb{R}^{N \times N}$ through \emph{joint (approximate) diagonalization} (JD). This strategy, originally developed for solving blind source separation problems in signal processing~\cite{Belouchrani1997} and later applied in various multiview clustering settings \cite{7053905, 9599440, 6858515}, seeks an orthogonal matrix $\mathbf{Q} \in \mathbb{R}^{N \times N}$ that makes all transformed matrices $\mathbf{Q}^\top \mathbf{L}_i \mathbf{Q}$ as diagonal as possible, thereby producing a shared approximate basis of eigenvectors across modalities. 

Joint diagonalization is commonly formulated as an optimization problem aimed at minimizing a prescribed measure of off-diagonality; two widely used approaches are JADE~\cite{JADE} and QN-Diag~\cite{ablin:hal-01936887}. JADE minimizes the sum of squared off-diagonal entries using a generalization of the classical Jacobi method~\cite{MR1238912}, which can be viewed as a block coordinate descent optimization technique. The Quasi-Newton Diagonalization (QN-Diag) method~\cite{ablin:hal-01936887} uses a different off-diagonality loss function specifically designed for symmetric positive-definite matrices, which extends to graph Laplacians for connected graphs by ignoring the single zero eigenvalue.  QN-Diag updates $\mathbf{X}$ using quasi-Newton steps.

Compared to JADE, QN-Diag typically offers modest improvements in computational efficiency, particularly for large $N$, while pursuing a similar objective. This behavior has been observed in several experimental studies, including~\cite{ablin:hal-01936887, ablin2019quasi}. However, JD-based methods for multiview clustering face two key limitations:
\begin{enumerate}
    \item Computational inefficiency: Both JADE and QN-Diag recover all $N$ joint eigenvectors, whereas spectral clustering requires only the bottom-$k$ eigenvectors. This full-spectrum computation adds substantial overhead. 
    \item Spectral distortion: Because the optimization considers the entire spectrum, alignment errors in large eigenvector components can adversely affect the quality of the clustering-relevant subspace, as will be shown by numerical experiments in Section~\ref{exp}-\ref{rjdbase_qn_diag_jade}.
\end{enumerate}
These drawbacks motivate alternatives - such as our RJD-BASE framework - that avoid full joint diagonalization and directly target the clustering-relevant subspace.

\subsection{Multiview Spectral Clustering (MVSC)}
Widely applied in the literature \cite{lei2024anchor, 9652464, 7053905}, Multiview Spectral Clustering (MVSC) \cite{perry2021mvlearn,kumarconf} extends spectral clustering to multiple modalities by iteratively co-regularizing per-view embeddings. For each modality \(i=1,\ldots,m\), with Laplacian \(\mathbf{L}_i\) from affinity \(\mathbf{W}_i\), initialize
\[
\mathbf{X}_i^{(0)} \in \mathbb{R}^{N\times k}
\]
as the bottom-\(k\) eigenvectors of \(\mathbf{L}_i\). At iteration \(j\), set
\[
\mathbf{S}_i^{(j)} = \mathrm{sym}\Big(\sum_{r\neq i}\mathbf{X}_r^{(j-1)}\mathbf{X}_r^{(j-1)\top}\mathbf{W}_i\Big),
\]
with the symmetrizer $\mathrm{sym}(\mathbf{A})=\tfrac12(\mathbf{A}+\mathbf{A}^\top)$. We then build \(\mathbf{L}_i^{(j)}\) as the graph Laplacian for the weight matrix \(\mathbf{S}_i^{(j)}\) and update \(\mathbf{X}_i^{(j)}\) as the bottom-\(k\) eigenspace of \(\mathbf{L}_i^{(j)}\). After a fixed number of iterations \(J\), the final embedding concatenates views:
\begin{equation} \label{eq:concatenate}
\mathbf{X}=\big[\mathbf{X}_1^{(J)} \,\big|\, \mathbf{X}_2^{(J)} \,\big|\, \cdots \,\big|\, \mathbf{X}_m^{(J)}\big]\in\mathbb{R}^{N\times mk}.
\end{equation}

\subsection{Co-Regularized Multiview Spectral Clustering (CoReg-MVSC)}
CoReg-MVSC \cite{perry2021mvlearn,kumar2011coreg,10495145,lei2024anchor} couples the different modalities on the graph Laplacian operator level rather than at the affinity level as in MVSC.
Concretely, CoReg-MVSC also initializes \(\mathbf{X}_i^{(0)}\) as the bottom-\(k\) eigenvectors of \(\mathbf{L}_i\) but at iteration \(j\) it performs the update
\[
\tilde{\mathbf{L}}_i^{(j)} \;=\; \mathbf{L}_i \;+\; \lambda \sum_{r\neq i} \mathbf{X}_r^{(j-1)}\mathbf{X}_r^{(j-1)\top}
\]
and computes $\mathbf{X}_i^{(j)}$ as the bottom-$k$ eigenvectors of \(\tilde{\mathbf{L}}_i^{(j)}\).
After a fixed number of iterations \(J\), the final embedding concatenates views, as in~\eqref{eq:concatenate}.

\subsection{Multiview K-Means (MV-KMeans)}
Multiview K-Means (MV-KMeans) \cite{perry2021mvlearn, bickel2004icdm} extends classical $k$-means clustering to two-view settings ($m=2$) by leveraging a co-EM (co-Expectation Maximization) strategy.
The algorithm initializes centroids  $\mathbf{C}^{(1)}_{0},\mathbf{C}^{(2)}_{0} $ separately for each modality either randomly or using $k$-means++~\cite{kmean++}. At iteration $t$, using \(  \mathbf{C}^{(2)}_{t-1}  \) as the centroids for view 1,  it computes cluster assignments by maximum likelihood and then updates   \( \mathbf{C}^{(1)}_t \). At iteration $t+1$, using \(  \mathbf{C}^{(1)}_{t}  \) as the centroids for view 2, it computes cluster assignments and then updates \( \mathbf{C}^{(2)}_{t+1} \). This alternating process continues until a predefined stopping criterion is satisfied \cite{bickel2004icdm}. At convergence, the final label for each data point is determined by selecting the cluster with the minimal average posterior probability across both views.

The alternating procedure promotes the clustering structure in each modality to agree with the latent structure captured by the other.

\subsection{Multiview Spherical K-Means (MV-SphKMeans)}

Multiview Spherical K-Means (MV-SphKMeans) \cite{perry2021mvlearn}, \cite{bickel2004icdm} alters MV-KMeans by using cosine similarity instead of the Euclidean distance metric. This modification makes the method suitable for data where directional information is more meaningful than magnitude.

\ifthenelse{\boolean{arxiv}}
  {\section{OUR FRAMEWORK}
\label{framework}

Our framework for multimodal spectral clustering departs from traditional full-spectrum alignment techniques, focusing directly on recovering the informative low-frequency components efficiently and accurately. For this purpose, it utilizes randomized sampling and selection based on $k$-dimensional spectral smoothness within the clustering-relevant subspace.

\subsection{Randomized Joint Diagonalization (RJD)}

Randomized Joint Diagonalization (RJD) \cite{he2024randomized} is a randomized method for approximately diagonalizing  a family of symmetric matrices by constructing random linear combinations of input matrices and successively performing  eigendecompositions to recover (approximate) common eigenvectors. For commuting matrices, RJD jointly diagonalizes each matrix with probability one. For nearly commuting matrices, RJD remains robust in the sense that, with high probability, it approximately diagonalizes each matrix with an error on the level of the input error. While these theoretical guarantees assume (near) commutativity, the graph Laplacians for
multimodal data are not expected to satisfy this property;
rather, this framework serves as a theoretical motivation for
investigating linear combinations.

The random (convex) combinations $\mathbf{L}(\pmb{\mu}) = \sum_{i=1}^m \mu_i \mathbf{L}_i$ used by RJD reflect randomized aggregations of modalities. The bottom-$k$ eigenvectors of $\mathbf{L}(\pmb{\mu})$ serve as an embedding $\mathbf{X} \in \mathbb{R}^{N \times k}$ used for downstream $k$-means clustering on its rows.  This approach is computationally efficient, requiring only partial eigendecompositions (avoids computing or optimizing a full joint diagonalizer), and scales well with the number of modalities. 

\subsection{Single-Directional Smoothness}

In \cite{coifman2023common}, a variational principle for selecting optimal convex combinations of graph Laplacians based on smoothness of functions on graphs has been established. 

For a graph $G$ with $N$ nodes, adjacency matrix $\mathbf{W}$ and  graph Laplacian $\mathbf{L}$,
the Rayleigh quotient
\[s_{\mathbf{L}}(\mathbf{x}) := \mathbf{x}^\top\mathbf{L}\mathbf{x} = \sum_{p\neq q}w_{pq}(x_p-x_q)^2\]
can be viewed as measuring the ``smoothness'' of a function with samples $\mathbf{x} \in \mathbb{R}^N$ on the $N$ nodes  
of the graph~\cite{coifman2023common}. In particular, large jumps across adjacent nodes get penalized.

For a family of connected graphs $\mathcal{G}=\{G_1,\ldots,G_m\}$ on the same $N$ nodes and with graph Laplacians $\mathbf{L}_1,\ldots,\mathbf{L}_m$, it has been proposed in~\cite{coifman2023common} to measure the smoothness of $\mathbf{x} \in \mathbb{R}^N$ over the nodes by the  worst-case smoothness, that is, 
\[s_{\mathcal{G}}(\mathbf{x}):= \max_{i=1,\ldots,m}s_{\mathbf{L}_i}(\mathbf{x}) = \|[ s_{\mathbf{L}_1}(\mathbf{x}),\ldots,s_{\mathbf{L}_m}(\mathbf{x})]\|_\infty\]
where $\|\cdot\|_\infty$ denotes the maximum norm. A key insight from~\cite{coifman2023common} is that the optimal $\mathbf{x} \in \reals^N$ (minimizing the worst-case smoothness) can be found by considering the second smallest eigenvalue $\lambda_1(\mathbf{L}(\pmb{\mu}))$ of linear combinations taking the form
\[\mathbf{L}(\pmb{\mu}):=\sum_{i=1}^m\mu_i\mathbf{L}_i, \quad \pmb{\mu} \in \Delta^{m-1},\]
where we recall that $\Delta^{m-1}$ denotes the standard $m$-dimensional simplex~\eqref{eq:simplex}.

\begin{theorem}[Theorem 2 in \cite{coifman2023common}]\label{thm:coifman2023}
   Assuming that $\lambda_1(\mathbf{L}(\pmb{\mu}))$ is a simple eigenvalue for every $\pmb{\mu} \in \Delta^{m-1}$, it holds that
   \[\min_{\substack{\mathbf{x}\top\mathbf{1} =0\\ \|\mathbf{x}\|_2 =1}} \max_{i=1,\ldots,m} \mathbf{x}^\top\mathbf{L}_i\mathbf{x} = \max_{\pmb{\mu} \in \Delta^{m-1}}\min_{\substack{\mathbf{x}^\top\mathbf{1} =0\\ \|\mathbf{x}\|_2 =1}} \mathbf{x}^\top\mathbf{L}(\pmb{\mu})\mathbf{x}.\]
\end{theorem}

Noting that the vector $\mathbf{1}$ is always an eigenvector belonging to the smallest eigenvalue $\lambda_0(\mathbf{L}(\pmb{\mu}))$, it follows that 
\[\min_{\substack{\mathbf{x}^\top\mathbf{1} =0\\ \|\mathbf{x}\|_2 =1}} \mathbf{x}^\top\mathbf{L}(\pmb{\mu})\mathbf{x}= \lambda_1(\mathbf{L}(\pmb{\mu})).\]
Thus, optimizing single-directional smoothness reduces to the eigenvalue optimization problem
\begin{equation}
\label{vector_vers}
\pmb{\mu}^* = \arg\max_{\pmb{\mu} \in \Delta^{m-1}} \lambda_1(\mathbf{L}(\pmb{\mu})).
\end{equation}
By Theorem~\ref{thm:coifman2023}, the optimal $\mathbf{x}$ is obtained as an eigenvector belonging to $\lambda_1(\mathbf{L}(\pmb{\mu^*}))$, which can serve as a one-dimensional embedding of the $N$ nodes across the whole family of graphs~\cite{coifman2023common}.

In the following, we will refer to the objective function $\lambda_1(\mathbf{L}(\pmb{\mu}))$ from~\eqref{vector_vers} as the \textbf{single-directional smoothness objective}.

\subsection{Bottom-$k$ Aggregated Spectral Energy (BASE) Smoothness}

We now aim at extending the concept of single-directional smoothness to suit the needs of spectral clustering, which requires a $k$-dimensional embedding matrix $\mathbf{X} \in \mathbb{R}^{N \times k}$ rather than a single vector.

Given a graph Laplacian $\mathbf{L}$ and a matrix $\mathbf{X} \in \mathbb{R}^{N \times k}$ with orthonormal columns (i.e., $\mathbf{X}^\top \mathbf{X} = \mathbf{I}_k$), we define the total smoothness of $\mathbf{X}$ with respect to $\mathbf{L}$ as

\begin{equation}\label{eq:S_L}
    s_{\mathbf{L}}(\mathbf{X}) = \trace(\mathbf{X}^\top \mathbf{L} \mathbf{X}) = \sum_{i=1}^k \mathbf{x}_i^\top \mathbf{L} \mathbf{x}_i.
\end{equation}
Each column $\mathbf{x}_i$ of $\mathbf{X}$ corresponds to a different direction in the embedding. 
Under the additional constraint $\mathbf{X}^\top\mathbf{1} = \mathbf{0}$, the matrix $\mathbf{X}$ formed by the bottom-$k$ eigenvectors of $\mathbf{L}$ minimizes~\eqref{eq:S_L}; see, e.g.,~\cite{matrixanalyis}.

In analogy to the single-vector case,
for a family of connected graphs $\mathcal{G}=\{G_1,\ldots,G_m\}$ with graph Laplacians $\mathbf{L}_1,\ldots,\mathbf{L}_m$
we measure the worst-case smoothness of an embedding $\mathbf{X} \in \mathbb{R}^{N \times k}$:
\begin{equation}
\label{eq:defsg}
s_\mathcal{G}(\mathbf{X}):=\| [ s_{\mathbf{L}_1}(\mathbf{X}),\ldots,s_{\mathbf{L}_m}(\mathbf{X}) ]\|_\infty. 
\end{equation}
The following result generalizes Theorem~\ref{thm:coifman2023} from one- to $k$-dimensional embeddings.

\begin{theorem}\label{thm:main_thm}
   Assuming that $\lambda_{k}(\mathbf{L}(\pmb{\mu})) < \lambda_{k+1}(\mathbf{L}(\pmb{\mu}))$ holds for every $\pmb{\mu} \in \Delta^{m-1}$, we have that
   \begin{equation}\label{eq:minimax_k}
       \min_{\substack{\mathbf{X}^\top\mathbf{1} =\mathbf{0}\\\mathbf{X}^\top\mathbf{X}=\mathbf{I}_k}} s_\mathcal{G}(\mathbf{X}) = \max_{\pmb{\mu} \in \Delta^{m-1}}\sum_{i=1}^{k}\lambda_i(\mathbf{L}(\pmb{\mu})).
   \end{equation}
\end{theorem}
\begin{proof}
Set $f(\pmb{\mu}, \mathbf{X} ):= \trace(\mathbf{X}^\top\mathbf{L}(\pmb{\mu})\mathbf{X})$.
Using dual norms,
$\|\mathbf u\|_\infty = \max\limits_{\|\mathbf v\|_1 =1} \inner{\mathbf u}{\mathbf v}$,
allows us to rewrite $s_\mathcal{G}(\mathbf{X})$ as
\begin{align*}
        s_\mathcal{G}(\mathbf{X}) &= \max_{\|\pmb{\mu}\|_1 =1 } \sum_{i = 1}^m \mu_i \trace(\mathbf{X}^\top\mathbf{L}_i\mathbf{X}) \\
        &= \max_{\pmb{\mu} \in \Delta^{m-1}} \sum_{i = 1}^m \mu_i \trace(\mathbf{X}^\top\mathbf{L}_i\mathbf{X}) = \max_{\pmb{\mu} \in \Delta^{m-1}} f(\pmb{\mu}, \mathbf{X}),
\end{align*}
where the second equality follows from the fact that $\trace(\mathbf{X}^\top\mathbf{L}_i\mathbf{X})$ is non-negative. On the other hand, the classical Ky-Fan theorem~\cite{stewardsun,matrixanalyis} implies that
\begin{equation}
 \label{eq:kyfan}
 \min_{\substack{\mathbf{X}^\top\mathbf{1} =\mathbf{0}\\\mathbf{X}^\top\mathbf{X}=\mathbf{I}_k}} f(\pmb{\mu}, \mathbf{X} ) =\sum_{i=1}^{k}\lambda_i(\mathbf{L}(\pmb{\mu}))
 =: g(\pmb{\mu}),
\end{equation}

where the minimum is assumed by the matrix
$\mathbf{X}$ containing an orthonormal basis of eigenvectors for $\lambda_1(\mathbf{L}(\pmb{\mu})),\ldots,\lambda_k(\mathbf{L}(\pmb{\mu}))$. Since $f(\pmb{\mu},\mathbf{X})$ is affine in $\pmb{\mu}$ for fixed $\mathbf{X}$, this variational characterization also shows that $g(\pmb{\mu})$, as the pointwise minimum of affine functions, is concave in $\pmb{\mu}$.
In summary,~\eqref{eq:minimax_k} is equivalent to establishing
\begin{equation} \label{eq:maxminaux}
 \min_{\substack{\mathbf{X}^\top\mathbf{1} =\mathbf{0}\\\mathbf{X}^\top\mathbf{X}=\mathbf{I}_k}} \max_{\pmb{\mu} \in \Delta^{m-1}} f(\pmb{\mu}, \mathbf{X} ) = 
 \max_{\pmb{\mu} \in \Delta^{m-1}}\min_{\substack{\mathbf{X}^\top\mathbf{1} =\mathbf{0}\\\mathbf{X}^\top\mathbf{X}=\mathbf{I}_k}} f(\pmb{\mu}, \mathbf{X} ).
\end{equation}

To prove~\eqref{eq:maxminaux}, choose $\pmb{\mu}^* \in \Delta^{m-1}$ that maximizes the eigenvalue sum $g(\pmb{\mu})$ from~\eqref{eq:kyfan}. Letting
$\mathbf{X}^*$ denote the corresponding
orthonormal basis of eigenvectors, we clearly have that
\begin{equation} \label{eq:saddlepoint1}
 f(\pmb{\mu}^*, \mathbf{X}^*) \le f(\pmb{\mu}^*, \mathbf{X})
\end{equation}
for all feasible $\mathbf{X}$. On the other hand, 
the spectral gap assumption, existing results on spectral functions~\cite{Lewis1996}, and the chain rule imply that the eigenvalue sum $g(\pmb{\mu})$ is differentiable,
with the gradient at $\pmb{\mu}^*$ given by

\begin{equation} \label{eq:gradient}
\nabla g(\pmb{\mu^*}) = \begin{bmatrix} \trace(\mathbf{X^*}^\top \mathbf{L}_1\mathbf{X}*)\\ \vdots \\ \trace(\mathbf{X^*}^\top \mathbf{L}_m\mathbf{X}^*)\end{bmatrix}.
\end{equation}
Because $g$ is concave and $\Delta^{m-1}$ is convex, $\pmb{\mu}^* \in \Delta^{m-1}$ is a maximizer if and only if 
the gradient of $g$ is in the normal cone of $\Delta^{m-1}$ at $\pmb{\mu}^*$, that is,
\[\inner{\nabla g(\pmb{\mu^*})}{\pmb{\mu} -\pmb{\mu}^*} \leq 0, \quad \forall \pmb{\mu} \in \Delta^{m-1}.\]
By the linearity of $f$ with respect to $\mu$, this condition can be rewritten
as
\begin{equation} \label{eq:saddlepoint2}
 f(\pmb{\mu}, \mathbf{X}^*) \le f(\pmb{\mu}^*, \mathbf{X}^*).
\end{equation}

The two inequalities~\eqref{eq:saddlepoint1} and~\eqref{eq:saddlepoint2} show that $(\pmb{\mu}^*, \mathbf{X}^*)$ is a saddlepoint of $f$ and, in turn,~\eqref{eq:maxminaux} holds; see, e.g.,~\cite[Sec 4.3, Exercise 14]{Borwein00convexanalysis}.
   
\end{proof}

The uniform spectral-gap condition in Theorem~\ref{thm:main_thm} is a technical assumption used to ensure differentiability of the eigenvalue-sum objective in the above analysis, analogous to the eigenvalue simplicity assumption in Theorem~\ref{thm:coifman2023}. Isolated violations of the spectral gap are not expected to significantly affect the practical applicability of the RJD-BASE algorithm.

The preceding result extends naturally to the case of $s_\mathcal{G}(\mathbf{X})$ measured in the $\ell_p$ norm. Specifically, for any $p>1$, we define
\[ 
s_\mathcal{G}^{(p)}(\mathbf{X}):=\| [ s_{\mathbf{L}_1}(\mathbf{X}),\ldots,s_{\mathbf{L}_m}(\mathbf{X}) ]\|_p.\]
 The following theorem generalizes Theorem~\ref{thm:main_thm} to this setting: finding minimal $s_\mathcal{G}^{(p)}$ is equivalent to maximizing the sum of the bottom-$k$ eigenvalues of $\mathbf{L}(\pmb{\mu})$ over $\pmb{\mu}$ in the unit $\ell_q$-ball,
\[\mathbf{B}_q := \{\mathbf{x} \in \mathbb{R}^m:\|\mathbf{x}\|_q=1\},\quad 1/p + 1/q = 1\] 
where $\ell_q$ is the dual norm of $\ell_p$.

\begin{theorem} For any $p > 1$, let $1/p + 1 /q = 1$. Assuming that $\lambda_{k}(\mathbf{L}(\pmb{\mu})) < \lambda_{k+1}(\mathbf{L}(\pmb{\mu}))$ holds for every $\pmb{\mu} \in \mathbf{B}_q $, we have that
   \begin{equation}
       \min_{\substack{\mathbf{X}^\top\mathbf{1} =\mathbf{0}\\\mathbf{X}^\top\mathbf{X}=\mathbf{I}_k}} s^{(p)}_\mathcal{G}(\mathbf{X}) = \max_{\pmb{\mu} \in \mathbf{B}_q}\sum_{i=1}^{k}\lambda_i(\mathbf{L}(\pmb{\mu})).
   \end{equation}
\end{theorem}
The proof of this theorem follows the same lines as the proof of Theorem~\ref{thm:main_thm}.

By Theorem~\ref{thm:main_thm}, the optimal $\mathbf{X}^*$ is obtained from the eigenvectors of $\mathbf{L(\pmb{\mu}^*)}$
for the optimal weight vector~$\mathbf{\pmb{\mu}}^*$ that maximizes $g(\pmb{\mu})$. We will refer to this
objective function
\begin{equation}\label{eq:base_smoothness}
    g(\pmb{\mu}) = \sum_{j=1}^k \lambda_j(\mathbf{L(\pmb{\mu})})
\end{equation}
as the \textbf{BASE smoothness objective}.

\subsection{RJD with BASE Selection}

While direct optimization of the BASE smoothness objective is possible and explored in later experiments, each evaluation requires an eigendecomposition, which comes with significant cost and limiting scalability. Instead, we leverage this objective as a selection criterion.

Specifically, we propose a simple procedure that retains the efficiency and robustness of randomized methods like RJD while introducing a principled and task-aligned mechanism for embedding selection. The procedure, RJD-BASE, is detailed in Algorithm~\ref{rjdbase}.

\begin{algorithm}[H]
\caption{Randomized Joint Diagonalization with Bottom-$k$ Aggregated Spectral Energy Selection (RJD-BASE)}
\label{rjdbase}
\textbf{Input:} Family of graph Laplacians $\{\mathbf{L}_1, \ldots, \mathbf{L}_m\}$, number of trials $T$, embedding dimension $k$\\
\textbf{Output:} Spectral embedding \ensuremath{\mathbf{X} \in \mathbb{R}^{N \times k}}

{\textsc{RJD-BASE}}$(\{\mathbf{L}_i\}_{i=1}^m,\, T,\, k)$
\begin{algorithmic}[1]
\FOR{$t = 1$ to $T$ \textbf{(in parallel)}}
    \STATE Sample $\tilde{\mu}^{(t)}_i \sim \text{Uniform}(0,1)$ \textbf{for} $i = 1, \ldots, m$
    \STATE Normalize: $\mu^{(t)}_i \leftarrow \tilde{\mu}^{(t)}_i / \sum_j \tilde{\mu}^{(t)}_j$
    \STATE Form $\mathbf{L}^{(t)} \leftarrow \sum_{i=1}^m \mu^{(t)}_i \mathbf{L}_i$
    \STATE Compute $\mathbf{X}^{(t)}$ as bottom-$k$ eigenvectors of $\mathbf{L}^{(t)}$
    \STATE Compute objective \ensuremath{O^{(t)} \leftarrow \sum_{j=1}^{k} \lambda_j(\mathbf{L}^{(t)})}
\ENDFOR
\STATE Select $t^\ast \gets \arg\max_{t \in \{1,\ldots,T\}} O^{(t)}$ 
\RETURN $\mathbf{X} \leftarrow \mathbf{X}^{(t^*)}$
\end{algorithmic}
\end{algorithm}

For $N\times N$ Laplacians, each RJD-BASE trial requires forming a weighted combination of the $m$ input Laplacians and computing only its bottom-$k$ eigenspace using a partial eigensolver. For dense Laplacians, the total complexity over $T$ trials is
$$
O\left(T\left[mN^2+C_{\mathrm{eig}}^{\mathrm{dense}}(N,k)\right]\right),
$$
where $C_{\mathrm{eig}}^{\mathrm{dense}}(N,k)$ denotes the cost of computing the bottom-$k$ eigenpairs. For sparse Laplacians, the corresponding complexity is
$$
O\left(T\left[\sum_{i=1}^m\operatorname{nnz}(\mathbf{L}_i)+C_{\mathrm{eig}}^{\mathrm{sparse}}(N,k)\right]\right),
$$
where the partial eigensolver can exploit sparse matrix-vector operations. \revise{Since the trials can be run sequentially while retaining only the current best embedding, the memory requirement does not grow with $T$.  For dense Laplacians, it is
$
O\left(mN^2+M_{\mathrm{eig}}^{\mathrm{dense}}(N,k)\right),
$
and for sparse Laplacians,
$$
O\left(\sum_{i=1}^m\operatorname{nnz}(\mathbf{L}_i)+M_{\mathrm{eig}}^{\mathrm{sparse}}(N,k)\right),
$$ where $M_{\mathrm{eig}}(N,k)$ denotes the memory required by the eigensolver.} Moreover, the $T$ randomized trials are mutually independent and can therefore be executed in parallel, reducing wall-clock time when parallel computational resources are available.

\noindent\textit{Note.} Lines 2–3 of Algorithm~\ref{rjdbase} sample from a distribution supported on the standard simplex $\Delta^{m-1}$, with mass concentrated near its center \cite{willms2021uniform, smith2004sampling}. As demonstrated in Section V.E, the choice between this distribution and a uniform distribution over the simplex has a negligible impact on final accuracy.

\noindent\textit{Note.} We investigate the effect of $T$ on RJD-BASE in Section~V.\ref{subsec:T_sensitivity} and show that performance stabilizes for moderate values of $T$, indicating that large $T$ is not required for reliable accuracy.
}
  {\section{OUR FRAMEWORK}
\label{framework}

Our framework for multimodal spectral clustering departs from traditional full-spectrum alignment techniques, focusing directly on recovering the informative low-frequency components efficiently and accurately. For this purpose, it utilizes randomized sampling and selection based on $k$-dimensional spectral smoothness within the clustering-relevant subspace.


\subsection{Randomized Joint Diagonalization (RJD)}

Randomized Joint Diagonalization (RJD) \cite{he2024randomized} is a randomized method for approximately diagonalizing a family of symmetric matrices by constructing random linear combinations of input matrices and successively performing eigendecompositions to recover (approximate) common eigenvectors. For commuting matrices, RJD jointly diagonalizes each matrix with probability one. For nearly commuting matrices, RJD remains robust in the sense that, with high probability, it approximately diagonalizes each matrix with an error on the level of the input error. While these theoretical guarantees assume (near) commutativity, the graph Laplacians for multimodal data are not expected to satisfy this property; rather, this framework serves as a theoretical motivation for investigating linear combinations.

The random (convex) combinations $\mathbf{L}(\pmb{\mu}) = \sum_{i=1}^m \mu_i \mathbf{L}_i$ used by RJD reflect randomized aggregations of modalities. The bottom-$k$ eigenvectors of $\mathbf{L}(\pmb{\mu})$ serve as an embedding $\mathbf{X} \in \mathbb{R}^{N \times k}$ used for downstream $k$-means clustering on its rows.  This approach is computationally efficient, requiring only partial eigendecompositions (avoids computing or optimizing a full joint diagonalizer), and scales well with the number of modalities. 

\subsection{Single-Directional Smoothness}

In \cite{coifman2023common}, a variational principle for selecting optimal convex combinations of graph Laplacians based on smoothness of functions on graphs has been established. 

For a graph $G$ with $N$ nodes, adjacency matrix $\mathbf{W}$ and  graph Laplacian $\mathbf{L}$,
the Rayleigh quotient
\[s_{\mathbf{L}}(\mathbf{x}) := \mathbf{x}^\top\mathbf{L}\mathbf{x} = \sum_{p\neq q}w_{pq}(x_p-x_q)^2\]
can be viewed as measuring the ``smoothness'' of a function with samples $\mathbf{x} \in \mathbb{R}^N$ on the $N$ nodes  
of the graph~\cite{coifman2023common}. In particular, large jumps across adjacent nodes get penalized.


For a family of connected graphs $\mathcal{G}=\{G_1,\ldots,G_m\}$ on the same $N$ nodes and with graph Laplacians $\mathbf{L}_1,\ldots,\mathbf{L}_m$, it has been proposed in~\cite{coifman2023common} to measure the smoothness of $\mathbf{x} \in \mathbb{R}^N$ over the nodes by the  worst-case smoothness, that is, 
\[s_{\mathcal{G}}(\mathbf{x}):= \max_{i=1,\ldots,m}s_{\mathbf{L}_i}(\mathbf{x}) = \|[ s_{\mathbf{L}_1}(\mathbf{x}),\ldots,s_{\mathbf{L}_m}(\mathbf{x})]\|_\infty\]
where $\|\cdot\|_\infty$ denotes the maximum norm. A key insight from~\cite{coifman2023common} is that the optimal $\mathbf{x} \in \reals^N$ (minimizing the worst-case smoothness) can be found by considering the second smallest eigenvalue $\lambda_1(\mathbf{L}(\pmb{\mu}))$ of linear combinations taking the form
\[\mathbf{L}(\pmb{\mu}):=\sum_{i=1}^m\mu_i\mathbf{L}_i, \quad \pmb{\mu} \in \Delta^{m-1},\]
where we recall that $\Delta^{m-1}$ denotes the standard $m$-dimensional simplex~\eqref{eq:simplex}.

\begin{theorem}[Theorem 2 in \cite{coifman2023common}]\label{thm:coifman2023}
   Assuming that $\lambda_1(\mathbf{L}(\pmb{\mu}))$ is a simple eigenvalue for every $\pmb{\mu} \in \Delta^{m-1}$, it holds that
   \[\min_{\substack{\mathbf{x}\top\mathbf{1} =0\\ \|\mathbf{x}\|_2 =1}} \max_{i=1,\ldots,m} \mathbf{x}^\top\mathbf{L}_i\mathbf{x} = \max_{\pmb{\mu} \in \Delta^{m-1}}\min_{\substack{\mathbf{x}^\top\mathbf{1} =0\\ \|\mathbf{x}\|_2 =1}} \mathbf{x}^\top\mathbf{L}(\pmb{\mu})\mathbf{x}.\]
\end{theorem}

Noting that the vector $\mathbf{1}$ is always an eigenvector belonging to the smallest eigenvalue $\lambda_0(\mathbf{L}(\pmb{\mu}))$, it follows that 
\[\min_{\substack{\mathbf{x}^\top\mathbf{1} =0\\ \|\mathbf{x}\|_2 =1}} \mathbf{x}^\top\mathbf{L}(\pmb{\mu})\mathbf{x}= \lambda_1(\mathbf{L}(\pmb{\mu})).\]
Thus, optimizing single-directional smoothness reduces to the eigenvalue optimization problem
\begin{equation}
\label{vector_vers}
\pmb{\mu}^* = \arg\max_{\pmb{\mu} \in \Delta^{m-1}} \lambda_1(\mathbf{L}(\pmb{\mu})).
\end{equation}
By Theorem~\ref{thm:coifman2023}, the optimal $\mathbf{x}$ is obtained as an eigenvector belonging to $\lambda_1(\mathbf{L}(\pmb{\mu^*}))$, which can serve as a one-dimensional embedding of the $N$ nodes across the whole family of graphs~\cite{coifman2023common}.

In the following, we will refer to the objective function $\lambda_1(\mathbf{L}(\pmb{\mu}))$ from~\eqref{vector_vers} as the \textbf{single-directional smoothness objective}.

\subsection{Bottom-$k$ Aggregated Spectral Energy (BASE) Smoothness}

We now aim at extending the concept of single-directional smoothness to suit the needs of spectral clustering, which requires a $k$-dimensional embedding matrix $\mathbf{X} \in \mathbb{R}^{N \times k}$ rather than a single vector.

Given a graph Laplacian $\mathbf{L}$ and a matrix $\mathbf{X} \in \mathbb{R}^{N \times k}$ with orthonormal columns (i.e., $\mathbf{X}^\top \mathbf{X} = \mathbf{I}_k$), we define the total smoothness of $\mathbf{X}$ with respect to $\mathbf{L}$ as

\begin{equation}\label{eq:S_L}
    s_{\mathbf{L}}(\mathbf{X}) = \trace(\mathbf{X}^\top \mathbf{L} \mathbf{X}) = \sum_{i=1}^k \mathbf{x}_i^\top \mathbf{L} \mathbf{x}_i.
\end{equation}
Each column $\mathbf{x}_i$ of $\mathbf{X}$ corresponds to a different direction in the embedding. 
Under the additional constraint $\mathbf{X}^\top\mathbf{1} = \mathbf{0}$, the matrix $\mathbf{X}$ formed by the bottom-$k$ eigenvectors of $\mathbf{L}$ minimizes~\eqref{eq:S_L}; see, e.g.,~\cite{matrixanalyis}.

In analogy to the single-vector case,
for a family of connected graphs $\mathcal{G}=\{G_1,\ldots,G_m\}$ with graph Laplacians $\mathbf{L}_1,\ldots,\mathbf{L}_m$
we measure the worst-case smoothness of an embedding $\mathbf{X} \in \mathbb{R}^{N \times k}$:
\begin{equation}
\label{eq:defsg}
s_\mathcal{G}(\mathbf{X}):=\| [ s_{\mathbf{L}_1}(\mathbf{X}),\ldots,s_{\mathbf{L}_m}(\mathbf{X}) ]\|_\infty. 
\end{equation}
The following result generalizes Theorem~\ref{thm:coifman2023} from one- to $k$-dimensional embeddings.

\begin{theorem}\label{thm:main_thm}
   Assuming that $\lambda_{k}(\mathbf{L}(\pmb{\mu})) < \lambda_{k+1}(\mathbf{L}(\pmb{\mu}))$ holds for every $\pmb{\mu} \in \Delta^{m-1}$, we have that
   \begin{equation}\label{eq:minimax_k}
       \min_{\substack{\mathbf{X}^\top\mathbf{1} =\mathbf{0}\\\mathbf{X}^\top\mathbf{X}=\mathbf{I}_k}} s_\mathcal{G}(\mathbf{X}) = \max_{\pmb{\mu} \in \Delta^{m-1}}\sum_{i=1}^{k}\lambda_i(\mathbf{L}(\pmb{\mu})).
   \end{equation}
\end{theorem}
\begin{proof}
Set $f(\pmb{\mu}, \mathbf{X} ):= \trace(\mathbf{X}^\top\mathbf{L}(\pmb{\mu})\mathbf{X})$.
Using dual norms,
$\|\mathbf u\|_\infty = \max\limits_{\|\mathbf v\|_1 =1} \inner{\mathbf u}{\mathbf v}$,
allows us to rewrite $s_\mathcal{G}(\mathbf{X})$ as
\begin{align*}
        s_\mathcal{G}(\mathbf{X}) &= \max_{\|\pmb{\mu}\|_1 =1 } \sum_{i = 1}^m \mu_i \trace(\mathbf{X}^\top\mathbf{L}_i\mathbf{X}) \\
        &= \max_{\pmb{\mu} \in \Delta^{m-1}} \sum_{i = 1}^m \mu_i \trace(\mathbf{X}^\top\mathbf{L}_i\mathbf{X}) = \max_{\pmb{\mu} \in \Delta^{m-1}} f(\pmb{\mu}, \mathbf{X}),
\end{align*}
where the second equality follows from the fact that $\trace(\mathbf{X}^\top\mathbf{L}_i\mathbf{X})$ is non-negative. On the other hand, the classical Ky-Fan theorem~\cite{stewardsun,matrixanalyis} implies that
\begin{equation}
 \label{eq:kyfan}
 \min_{\substack{\mathbf{X}^\top\mathbf{1} =\mathbf{0}\\\mathbf{X}^\top\mathbf{X}=\mathbf{I}_k}} f(\pmb{\mu}, \mathbf{X} ) =\sum_{i=1}^{k}\lambda_i(\mathbf{L}(\pmb{\mu}))
 =: g(\pmb{\mu}),
\end{equation}
where the minimum is assumed by the matrix
$\mathbf{X}$ containing an orthonormal basis of eigenvectors for $\lambda_1(\mathbf{L}(\pmb{\mu})),\ldots,\lambda_k(\mathbf{L}(\pmb{\mu}))$. Since $f(\pmb{\mu},\mathbf{X})$ is affine in $\pmb{\mu}$ for fixed $\mathbf{X}$, this variational characterization also shows that $g(\pmb{\mu})$, as the pointwise minimum of affine functions, is concave in $\pmb{\mu}$.
In summary,~\eqref{eq:minimax_k} is equivalent to establishing
\begin{equation} \label{eq:maxminaux}
 \min_{\substack{\mathbf{X}^\top\mathbf{1} =\mathbf{0}\\\mathbf{X}^\top\mathbf{X}=\mathbf{I}_k}} \max_{\pmb{\mu} \in \Delta^{m-1}} f(\pmb{\mu}, \mathbf{X} ) = 
 \max_{\pmb{\mu} \in \Delta^{m-1}}\min_{\substack{\mathbf{X}^\top\mathbf{1} =\mathbf{0}\\\mathbf{X}^\top\mathbf{X}=\mathbf{I}_k}} f(\pmb{\mu}, \mathbf{X} ).
\end{equation}

To prove~\eqref{eq:maxminaux}, choose $\pmb{\mu}^* \in \Delta^{m-1}$ that maximizes the eigenvalue sum $g(\pmb{\mu})$ from~\eqref{eq:kyfan}. Letting
$\mathbf{X}^*$ denote the corresponding
orthonormal basis of eigenvectors, we clearly have that
\begin{equation} \label{eq:saddlepoint1}
 f(\pmb{\mu}^*, \mathbf{X}^*) \le f(\pmb{\mu}^*, \mathbf{X})
\end{equation}
for all feasible $\mathbf{X}$. On the other hand, 
the spectral gap assumption, existing results on spectral functions~\cite{Lewis1996}, and the chain rule imply that the eigenvalue sum $g(\pmb{\mu})$ is differentiable,
with the gradient at $\pmb{\mu}^*$ given by
\begin{equation} \label{eq:gradient}
\nabla g(\pmb{\mu^*}) = \begin{bmatrix} \trace(\mathbf{X^*}^\top \mathbf{L}_1\mathbf{X}*)\\ \vdots \\ \trace(\mathbf{X^*}^\top \mathbf{L}_m\mathbf{X}^*)\end{bmatrix}.
\end{equation}
Because $g$ is concave and $\Delta^{m-1}$ is convex, $\pmb{\mu}^* \in \Delta^{m-1}$ is a maximizer if and only if 
the gradient of $g$ is in the normal cone of $\Delta^{m-1}$ at $\pmb{\mu}^*$, that is,
\[\inner{\nabla g(\pmb{\mu^*})}{\pmb{\mu} -\pmb{\mu}^*} \leq 0, \quad \forall \pmb{\mu} \in \Delta^{m-1}.\]
By the linearity of $f$ with respect to $\mu$, this condition can be rewritten
as
\begin{equation} \label{eq:saddlepoint2}
 f(\pmb{\mu}, \mathbf{X}^*) \le f(\pmb{\mu}^*, \mathbf{X}^*).
\end{equation}

The two inequalities~\eqref{eq:saddlepoint1} and~\eqref{eq:saddlepoint2} show that $(\pmb{\mu}^*, \mathbf{X}^*)$ is a saddlepoint of $f$ and, in turn,~\eqref{eq:maxminaux} holds; see, e.g.,~\cite[Sec 4.3, Exercise 14]{Borwein00convexanalysis}.
\end{proof}
\noindent\textit{Note.} For a more general formulation of Theorem~\ref{thm:main_thm} that allows for a general
$\ell_p$ norm, $p>1$, instead of the maximum norm in~\eqref{eq:defsg}, we refer the reader to the extended version~\cite[Theorem 3]{arxiv_version}.

By Theorem~\ref{thm:main_thm}, the optimal $\mathbf{X}^*$ is obtained from the eigenvectors of $\mathbf{L(\pmb{\mu}^*)}$
for the optimal weight vector~$\mathbf{\pmb{\mu}}^*$ that maximizes $g(\pmb{\mu})$. We will refer to this
objective function
\begin{equation}\label{eq:base_smoothness}
    g(\pmb{\mu}) = \sum_{j=1}^k \lambda_j(\mathbf{L(\pmb{\mu})})
\end{equation}
as the \textbf{BASE smoothness objective}. 

\subsection{RJD with BASE Selection}

While direct optimization of the BASE smoothness objective is possible and explored in later experiments, each evaluation requires an eigendecomposition, which comes with significant cost and limiting scalability. Instead, we leverage this objective as a selection criterion.

Specifically, we propose a simple procedure that retains the efficiency and robustness of randomized methods like RJD while introducing a principled and task-aligned mechanism for embedding selection. The procedure, RJD-BASE, is detailed in Algorithm~\ref{rjdbase}.

\begin{algorithm}[H]
\caption{Randomized Joint Diagonalization with Bottom-$k$ Aggregated Spectral Energy Selection (RJD-BASE)}
\label{rjdbase}
\textbf{Input:} Family of graph Laplacians $\{\mathbf{L}_1, \ldots, \mathbf{L}_m\}$, number of trials $T$, embedding dimension $k$\\
\textbf{Output:} Spectral embedding \ensuremath{\mathbf{X} \in \mathbb{R}^{N \times k}}

{\textsc{RJD-BASE}}$(\{\mathbf{L}_i\}_{i=1}^m,\, T,\, k)$
\begin{algorithmic}[1]
\FOR{$t = 1$ to $T$ \textbf{(in parallel)}}
    \STATE Sample $\tilde{\mu}^{(t)}_i \sim \text{Uniform}(0,1)$ \textbf{for} $i = 1, \ldots, m$
    \STATE Normalize: $\mu^{(t)}_i \leftarrow \tilde{\mu}^{(t)}_i / \sum_j \tilde{\mu}^{(t)}_j$
    \STATE Form $\mathbf{L}^{(t)} \leftarrow \sum_{i=1}^m \mu^{(t)}_i \mathbf{L}_i$
    \STATE Compute $\mathbf{X}^{(t)}$ as bottom-$k$ eigenvectors of $\mathbf{L}^{(t)}$
    \STATE Compute objective \ensuremath{O^{(t)} \leftarrow \sum_{j=1}^{k} \lambda_j(\mathbf{L}^{(t)})}
\ENDFOR
\STATE Select $t^\ast \gets \arg\max_{t \in \{1,\ldots,T\}} O^{(t)}$ 
\RETURN $\mathbf{X} \leftarrow \mathbf{X}^{(t^*)}$
\end{algorithmic}
\end{algorithm}

\noindent\textit{Note.} Lines 2–3 of Algorithm~\ref{rjdbase} sample from a distribution supported on the standard simplex $\Delta^{m-1}$, with mass concentrated near its center \cite{willms2021uniform, smith2004sampling}. As demonstrated in Section~V.E of \cite{arxiv_version}, the choice between this distribution and a uniform distribution over the simplex has a negligible impact on final accuracy.

\noindent\textit{Note.} \revise{A time and memory complexity analysis of RJD-BASE is given in Section~III.D of \cite{arxiv_version}.} The effect of $T$ on RJD-BASE is investigated in Section~V.F of \cite{arxiv_version}, showing that large $T$ is not required for reliable accuracy.}
\ifthenelse{\boolean{arxiv}}
  {\section{DATASETS}
\label{datasets}

\revise{To verify Algorithm~\ref{rjdbase}, we perform experiments on six datasets: two synthetic datasets (weighted SBM and NGM) and four real-world datasets (Caltech-7, Digits, Nutrimouse, and MSRC).} Depending on the setting, we either: (i) construct synthetic graph modalities ourselves and then generate the corresponding graph Laplacians or (ii) begin with modality-specific feature matrices which we model as graphs and compute graph Laplacians from them.

\subsection{Synthetic Weighted SBM}

We construct a synthetic multimodal dataset based on a weighted Stochastic Block Model (SBM) \cite{ng2021weighted}. The graph consists of \( N \) nodes partitioned into \( k \) ground-truth clusters. Cluster sizes are imbalanced, drawn from a Dirichlet distribution with uniform concentration, and the resulting cluster labels \( \mathbf{Y} \in \{1, \ldots, k\}^N \) are randomly permuted.

Each modality \( i \in \{1, \ldots, m\} \) is defined by a unique combination of:
\begin{itemize}
    \item A node-level real-valued feature vector \( \mathbf{x}^{(i)} \in \mathbb{R}^N \), with entries sampled i.i.d. from \( \mathcal{N}(0, 1) \)
    \item A symmetric block probability matrix \( \mathbf{B}^{(i)} \in \mathbb{R}^{k \times k} \) specifying relative edge strength between clusters
\end{itemize}

To simulate \textit{complementary} and \textit{partially informative} views, we define the block matrices per modality with \(k=6\) and \(m=4\) as follows:

\begin{itemize}
    \item Modality 1: Strong intra-cluster structure for clusters 1-3 and weaker structure for clusters 4-6:
    \[
    \mathbf{B}^{(1)} = \text{diag}(\underbrace{\alpha,\alpha,\alpha}_{\text{clusters 1--3}}, \underbrace{\beta,\beta,\beta}_{\text{clusters 4--6}}) + \varepsilon.
    \]
    \item Modality 2: Strong intra-cluster structure for clusters 4-6 and weaker structure for clusters 1-3:
    \[
    \mathbf{B}^{(2)} = \text{diag}(\underbrace{\zeta,\zeta,\zeta}_{\text{clusters 1--3}}, \underbrace{\xi,\xi,\xi}_{\text{clusters 4--6}}) + \eta.
    \]
    \item Modality 3: Overall poor clustering signal:
    \[
    \mathbf{B}^{(3)} = \gamma \cdot \mathbf{1}_{k \times k} + \chi.
    \]
    \item Modality 4: Moderate intra- and inter-cluster structure across all clusters:
    \[
    \mathbf{B}^{(4)} = \theta \cdot \mathbf{I}_k + \delta \cdot (\mathbf{1}_{k \times k} - \mathbf{I}_k).
    \]
\end{itemize}

Given features \( \mathbf{x}^{(i)} \), we define a similarity matrix via the radial basis function (RBF) kernel \cite{izquierdo2015spectral}:
\[
\mathbf{S}^{(i)}_{pq} = \exp\left( -\frac{(\mathbf{x}^{(i)}_p - \mathbf{x}^{(i)}_q)^2}{2\sigma^2} \right),
\]
where \( \sigma > 0 \) is a fixed kernel width. To inject cluster structure, we define a weight mask matrix \( \mathbf{C}^{(i)} \in \mathbb{R}^{N \times N} \) using the modality-specific block matrix \( \mathbf{B}^{(i)} \) and the cluster labels:
\[
\mathbf{C}^{(i)}_{pq} = \mathbf{B}^{(i)}_{Y_p, Y_q}.
\]

That is, for each node pair $(p,q)$, we look up the cluster memberships of the nodes (denoted $\mathbf{Y}_p$ and $\mathbf{Y}_q$) and note the proper edge weight as defined by the modality.

The final weighted adjacency matrix is then:
\[
\mathbf{W}^{(i)} = \mathbf{S}^{(i)} \circ \mathbf{C}^{(i)},
\]
where \( \circ \) denotes the element-wise product. We set the diagonal elements of \( \mathbf{W}^{(m)} \) to $0$ to enforce the absence of self-loops and finally compute the symmetric normalized Laplacian.

\noindent
The specific parameter values used in our experiments are:
\begin{align*}
N &= 300, \quad k = 6, \quad m = 4, \\
\sigma &= 1 \quad (\text{except } \sigma = 10^6 \text{ for Modality 3}), \\
\alpha &= \xi = 0.9, \quad \beta = \zeta = 0.05, \quad \gamma = 0.06, \\
\theta &= 0.7, \quad \delta = 0.2, \quad \varepsilon = \eta = \chi = 0.005~.
\end{align*}

This construction promotes the idea that no single modality fully resolves the clustering structure. Instead, each emphasizes different portions of the cluster space such that jointly, the modalities offer a richer and more complete view of the latent structure. Thus, this constructions yields a representative testbed for evaluating multimodal joint diagonalization methods. 

Figure \ref{heatmaps} provides an intuitive illustration of the generated graphs. We plot heatmaps of the adjacency matrices $\mathbf{W}^{(i)}$ for each modality, with nodes ordered by their ground-truth cluster assignments. In these visualizations, bright diagonal blocks correspond to strong intra-cluster connectivity, while darker off-diagonal regions indicate weaker inter-cluster connections. The qualitative differences across modalities are immediately visible: some views display sharp, high-contrast blocks for a subset of clusters, while others exhibit more moderate or noisy structure. 

\begin{figure}[h]
    \centering

    \begin{minipage}[c]{0.49\linewidth}
        \centering
        \includegraphics[width=\linewidth]{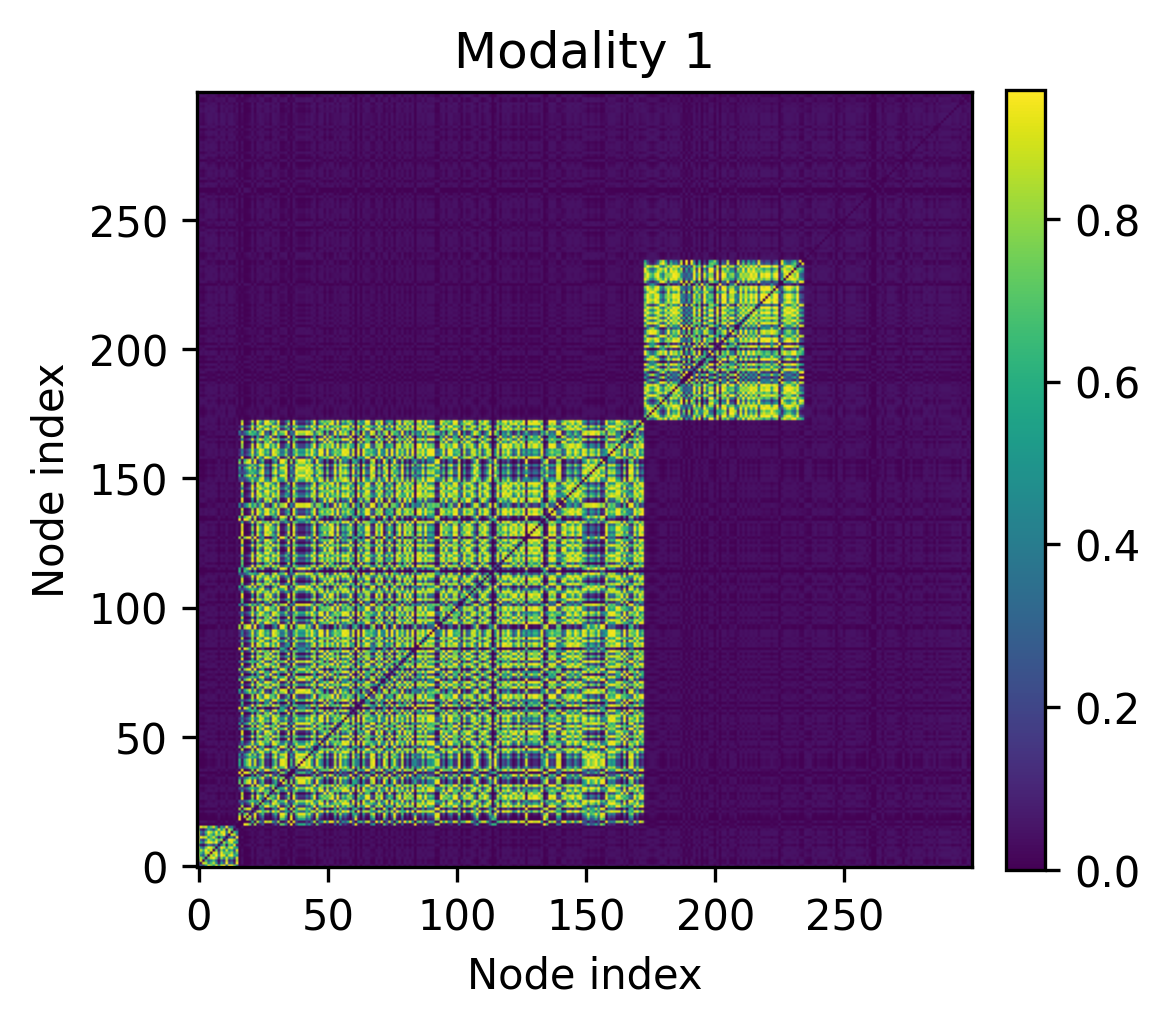}
        \vspace{2pt}
        {\footnotesize (a) Modality 1}
    \end{minipage}\hfill
    \begin{minipage}[c]{0.49\linewidth}
        \centering
        \includegraphics[width=\linewidth]{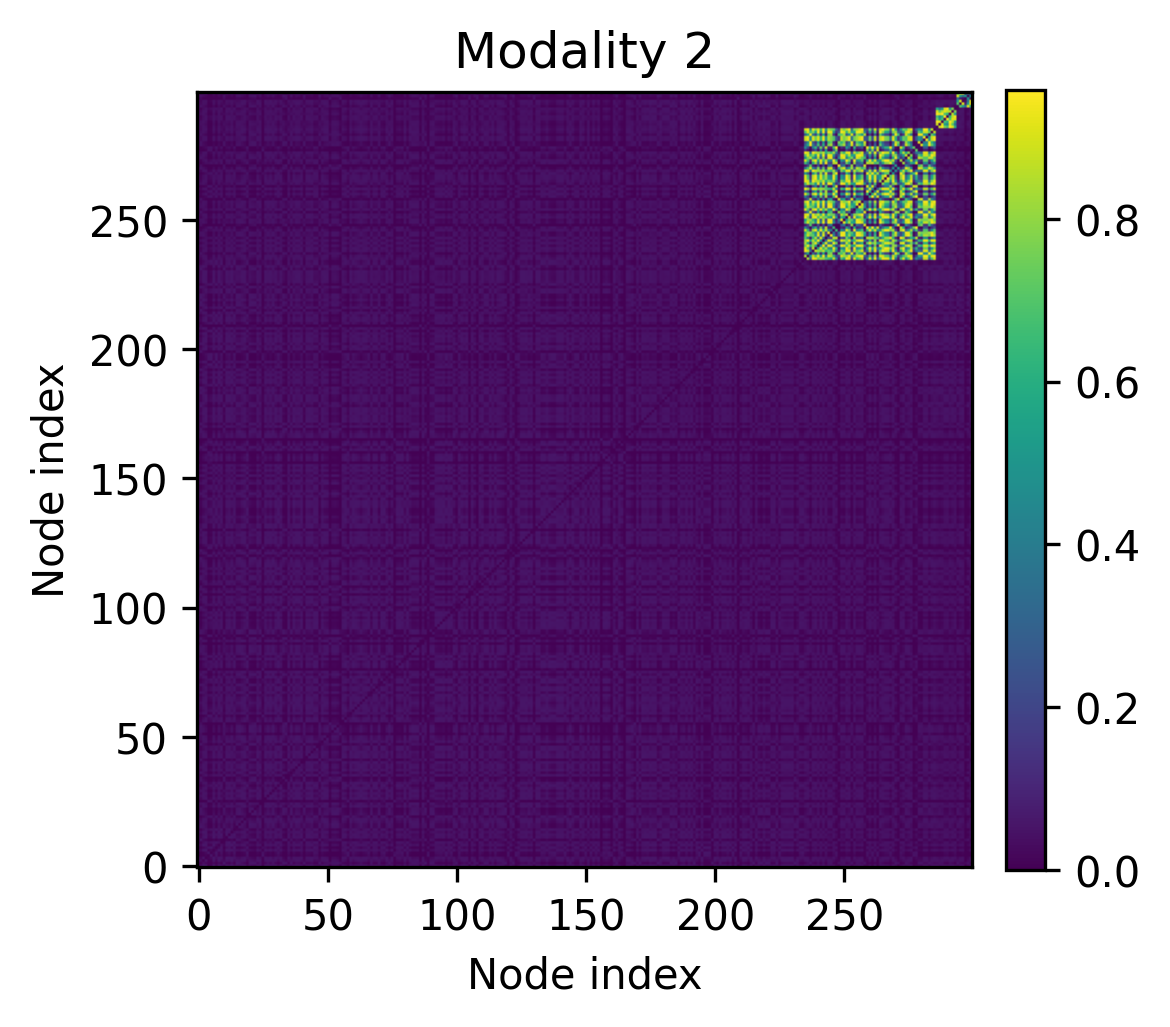}
        \vspace{2pt}
        {\footnotesize (b) Modality 2}
    \end{minipage}

    \vspace{0.5em}

    \begin{minipage}[c]{0.49\linewidth}
        \centering
        \includegraphics[width=\linewidth]{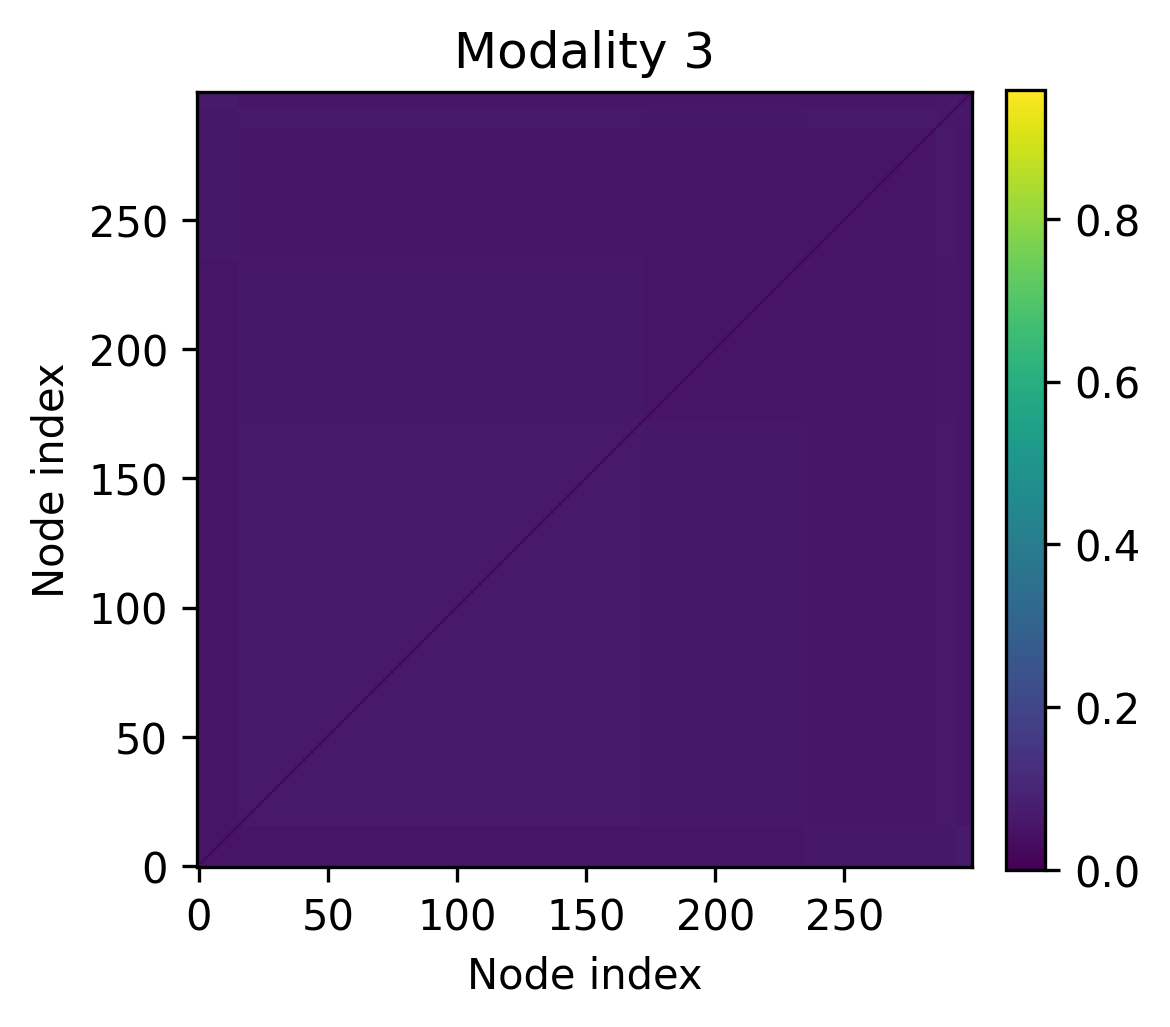}
        \vspace{2pt}
        {\footnotesize (c) Modality 3}
    \end{minipage}\hfill
    \begin{minipage}[c]{0.49\linewidth}
        \centering
        \includegraphics[width=\linewidth]{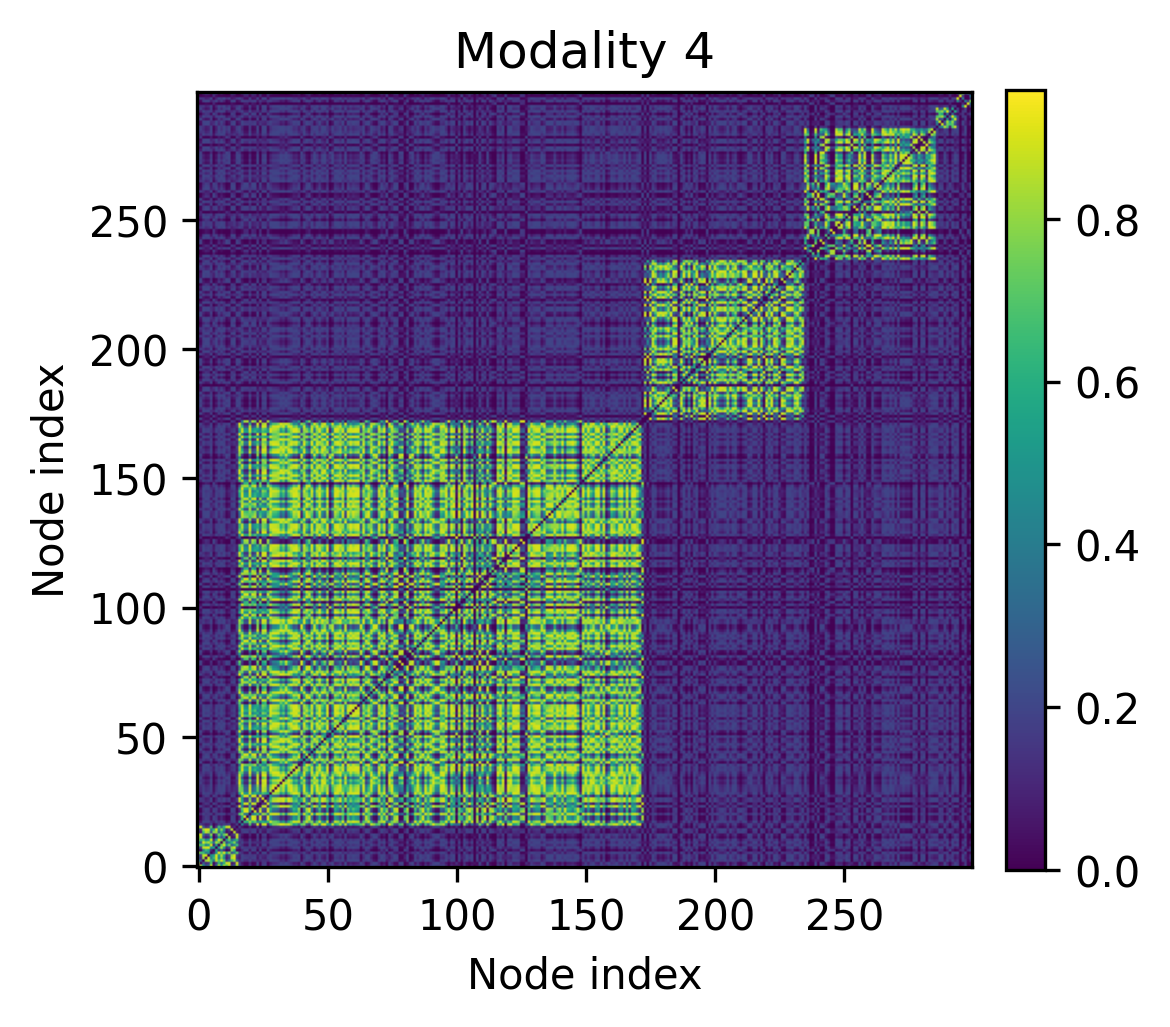}
        \vspace{2pt}
        {\footnotesize (d) Modality 4}
    \end{minipage}

    \caption{\footnotesize Adjacency matrix heatmaps for the $m=4$ modalities in the weighted SBM dataset, with nodes sorted by ground-truth cluster. Color scale shows edge weights (RBF similarity $\times$ block strength).}
    \label{heatmaps}
\end{figure}

\bigskip

The subsequent datasets utilize multimodal data in the form of feature matrices $\mathbf{Z}^{(i)} \in \mathbb{R}^{N \times d_i}$, where $N$ is the number of samples and $d_i$ is the number of features in modality $i$.
We define the affinity matrix $\mathbf{W}^{(i)} \in \mathbb{R}^{N \times N}$ using a self-tuning Gaussian kernel \cite{zelnik2004self} in terms of its entries as in \cite{7053905, nataliani}:
\[
w^{(i)}_{pq} = \begin{cases}
\exp\left(-\dfrac{\|\mathbf{z}^{(i)}_p - \mathbf{z}^{(i)}_q\|^2}{\sigma_p \sigma_q}\right)~, & p \ne q \\
0~, & p = q.
\end{cases}
\]
where $\mathbf{z}^{(i)}_p$ is the $p$th column of $\mathbf{Z}^{(i)}$ and $\sigma_p$ is a local bandwidth parameter for each sample defined as the distance to its $k$-th nearest neighbor. This results in a fully connected, symmetric graph with adaptive Gaussian weights and zero diagonal. Given $\mathbf{W}^{(i)}$, we construct the corresponding symmetric normalized Laplacian as before.

\subsection{Caltech-7}

We consider a multimodal subset of the Caltech-101 image dataset, commonly referred to as \textit{Caltech-7} \cite{5995740,li_andreeto_ranzato_perona_2022, li2015large}. This benchmark consists of 1,474 images across 7 categories:
\texttt{dollar\_bill}, \texttt{snoopy}, \texttt{windsor\_chair}, \texttt{stop\_sign}, \texttt{Motorbikes}, \texttt{garfield}, and \texttt{Faces}.
Each image is represented in six distinct feature modalities, yielding a total of six data views per sample:

\begin{itemize}
    \item Gabor (48 dimensions)
    \item Wavelet Moments (40 dimensions)
    \item CENTRIST (254 dimensions)
    \item Histogram of Oriented Gradients (HOG) (1984 dimensions)
    \item GIST (512 dimensions)
    \item Local Binary Patterns (LBP) (928 dimensions)
\end{itemize}

These feature vectors are treated independently as six modalities.

\subsection{Digits}
\label{subsec:digit}
We also consider a multimodal version of the UCI Optical Recognition of Handwritten Digits dataset \cite{liu2013multi}, commonly referred to as \textit{Digits}. This benchmark consists of 3,823 grayscale images of handwritten digits (0 through 9), each represented as an $8 \times 8$ pixel grid. The dataset comprises 10 classes corresponding to digit labels.

To enable multimodal analysis, we extract two distinct feature representations for each image sample:

\begin{itemize}
    \item DCT (76 dimensions): the top 76 coefficients from the 2D Discrete Cosine Transform of the image, capturing global frequency structure.
    \item Patch Averages (240 dimensions): average pixel intensities computed over a grid of $2 \times 3$ patches, capturing coarse spatial information.
\end{itemize}

These feature vectors are treated as complementary modalities.

\subsection{Nutrimouse}

We consider the \textit{Nutrimouse} dataset \cite{nutrimouse}, which arises from a nutrigenomics study of mice and is commonly used as a benchmark for multiview learning \cite{parea}. It contains $N=40$ mouse samples, each with two modalities: gene expression data and fatty acid concentration profiles. Each mouse is labeled according to one of $k=5$ diet categories.

\subsection{MSRC}

We construct a multimodal image dataset derived from the MSRC Object Recognition Image Database (MSRC-ORID) \cite{msrc}. We use $k=7$ object categories and sample an equal number of images per category, resulting in a total of $N=210$ images. We represent each image with four common visual feature modalities: local binary patterns (LBP) to capture texture information, downsampled grayscale pixel intensities to capture coarse global structure, Gabor filter energy features to encode oriented texture responses, and edge orientation histograms to represent shape and contour information.

\subsection{Nonlinear Gaussian Mixture (NGM)}

We generate a synthetic two-modality dataset using the \texttt{make\_gaussian\_mixture} generator from the \texttt{mvlearn} library \cite{perry2021mvlearn}. Following the construction described in \cite{perry2021mvlearn}, latent variables are sampled from a mixture of two multivariate Gaussian distributions with equal prior probability; a nonlinear transformation is applied to produce a second view; and independent noise dimensions are added to each view. We set $N=5000$ samples and $k=2$ clusters, with latent Gaussian centers at $(0,2)$ and $(0,-2)$ and identity covariances. We use \texttt{noise\_dims}$=2$, apply an element-wise $\tanh(\cdot)$ transformation to generate the second view, and standardize prior to Laplacian construction.
}
 {\section{DATASETS}
\label{datasets}

\revise{To verify Algorithm~\ref{rjdbase}, we perform experiments on six datasets: two synthetic datasets (weighted SBM and NGM) and four real-world datasets (Caltech-7, Digits, Nutrimouse, and MSRC).} Depending on the setting, we either: (i) construct synthetic graph modalities ourselves and then generate the corresponding graph Laplacians or (ii) begin with modality-specific feature matrices which we model as graphs and compute graph Laplacians from them. 

For more detailed construction details and specifications of each dataset; see~\cite{arxiv_version}.

\subsection{Synthetic Weighted SBM}
We construct a controlled, synthetic multimodal dataset  based on a weighted Stochastic Block Model (SBM) \cite{ng2021weighted}. We use \(N=300\) nodes partitioned into \(k=6\) clusters (imbalanced, drawn from a Dirichlet distribution with uniform concentration).

For each of the \(m=4\) modalities, we (i) generate 1D features and turn them into a soft similarity with a Gaussian radial basis function (RBF) kernel \cite{izquierdo2015spectral}, (ii) specify a simple \(k\times k\) block pattern \(\mathbf{B}^{(i)}\) that states intra-cluster connections, and (iii) combine the two by multiplying the similarity entry for a node pair with the corresponding weight. We then zero the self-edges and use the symmetric normalized Laplacian downstream.

To simulate \textit{complementary} and \textit{partially informative} views, modality 1 makes clusters \(1\text{–}3\) strong and \(4\text{–}6\) weak, modality 2 does the reverse, modality 3 intentionally provides almost no cluster signal (near-constant similarity via a very large kernel bandwidth), and modality 4 is moderate across all clusters. 
Figure~\ref{heatmaps} provides an intuitive illustration of the generated graphs.

\begin{figure}[h]
    \centering

    \begin{minipage}[c]{0.49\linewidth}
        \centering
        \includegraphics[width=0.8\linewidth]{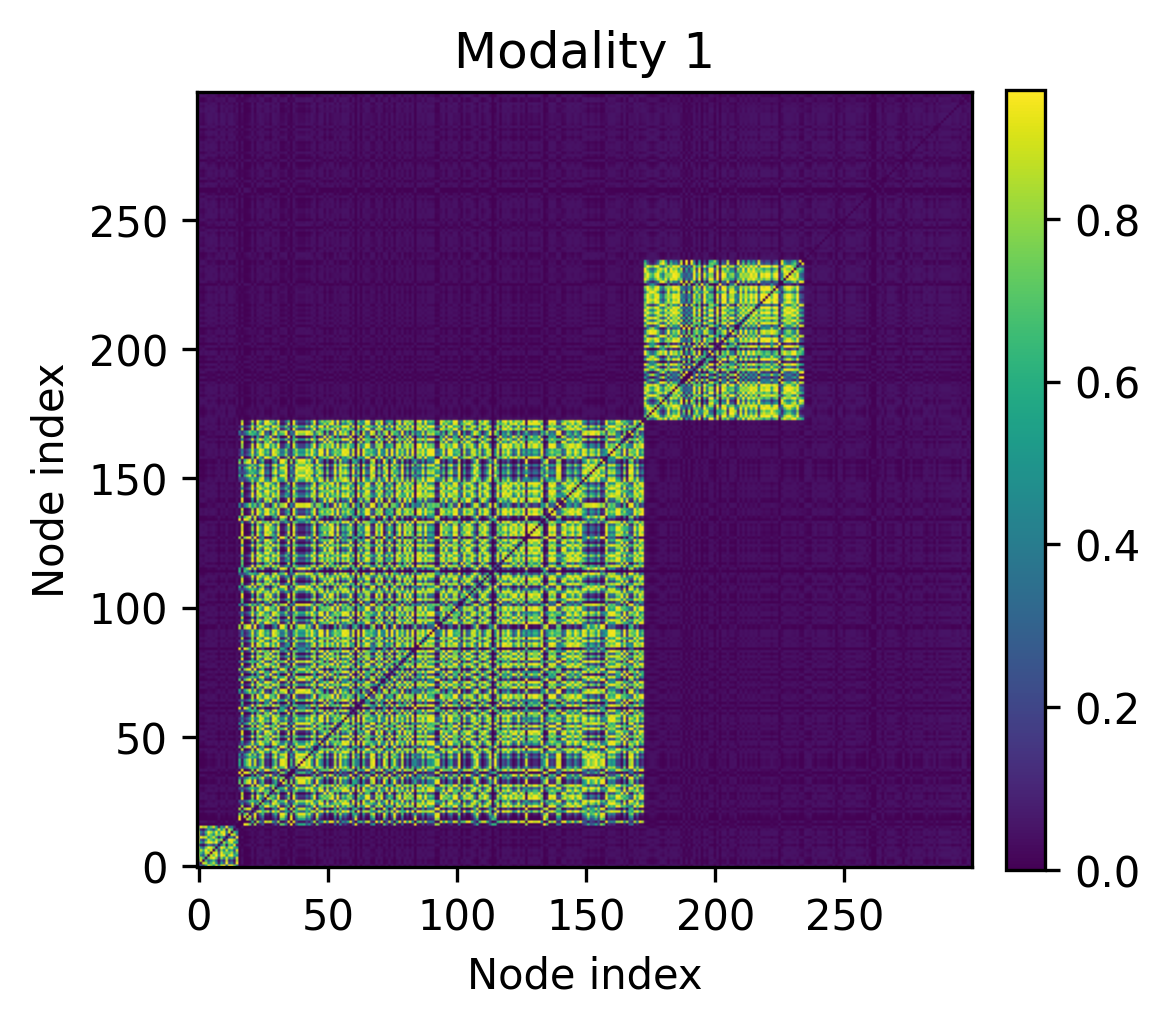}
        \vspace{2pt}
        {\footnotesize \\(a) Modality 1}
    \end{minipage}\hfill
    \begin{minipage}[c]{0.49\linewidth}
        \centering
        \includegraphics[width=0.8\linewidth]{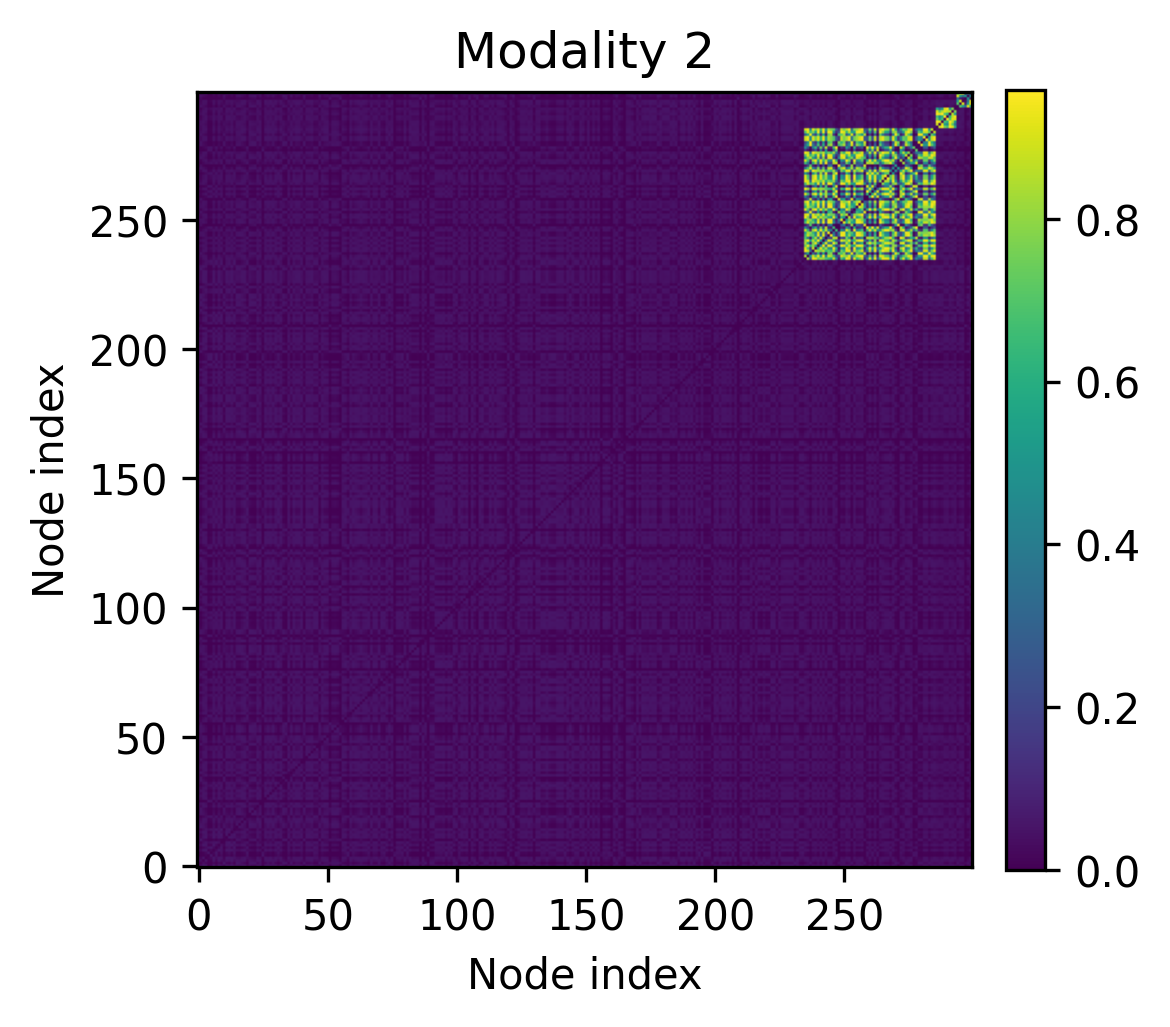}
        \vspace{2pt}
        {\footnotesize \\(b) Modality 2}
    \end{minipage}

    \vspace{0.5em}

    \begin{minipage}[c]{0.49\linewidth}
        \centering
        \includegraphics[width=0.8\linewidth]{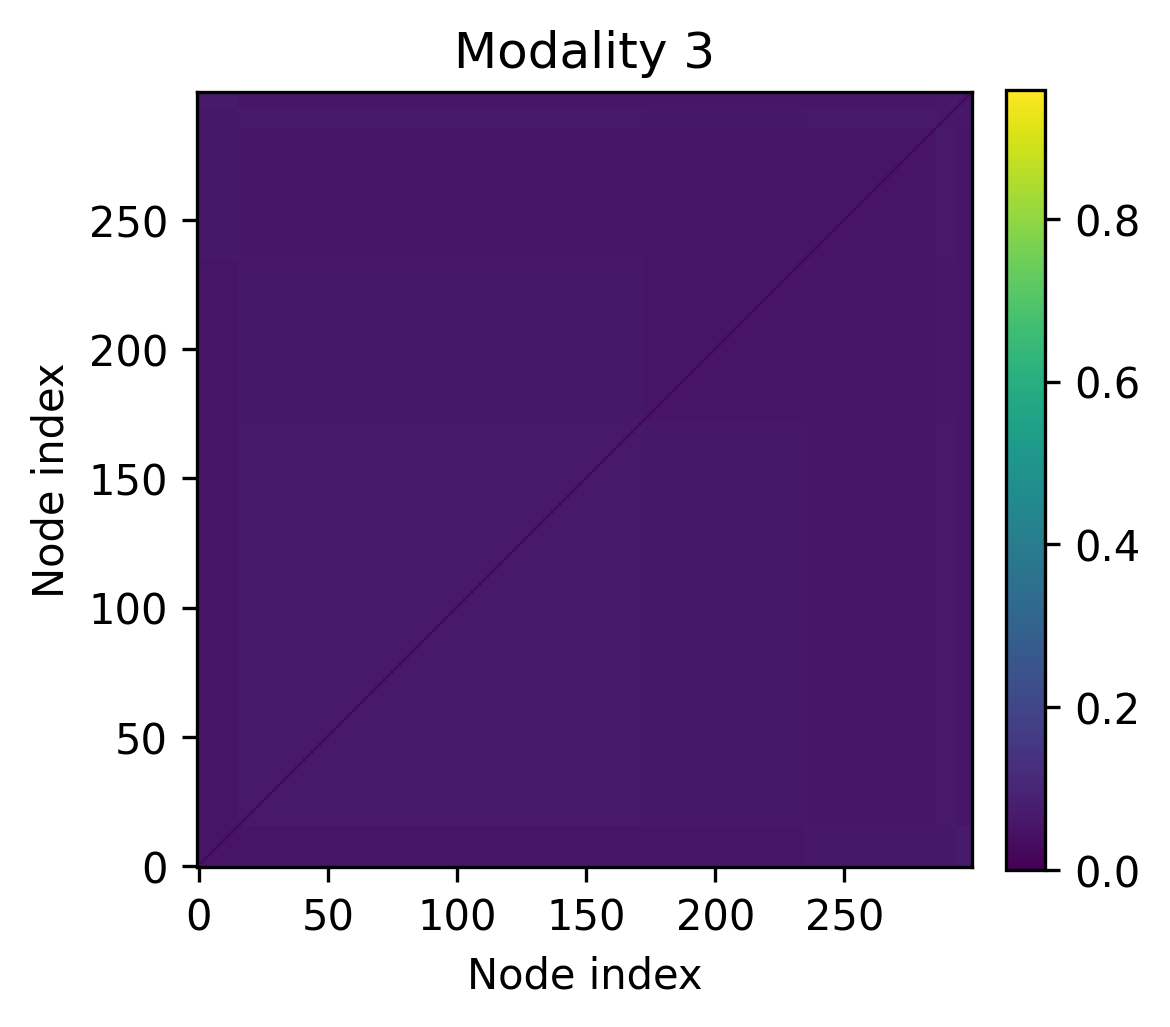}
        \vspace{2pt}
        {\footnotesize \\(c) Modality 3}
    \end{minipage}\hfill
    \begin{minipage}[c]{0.49\linewidth}
        \centering
        \includegraphics[width=0.8\linewidth]{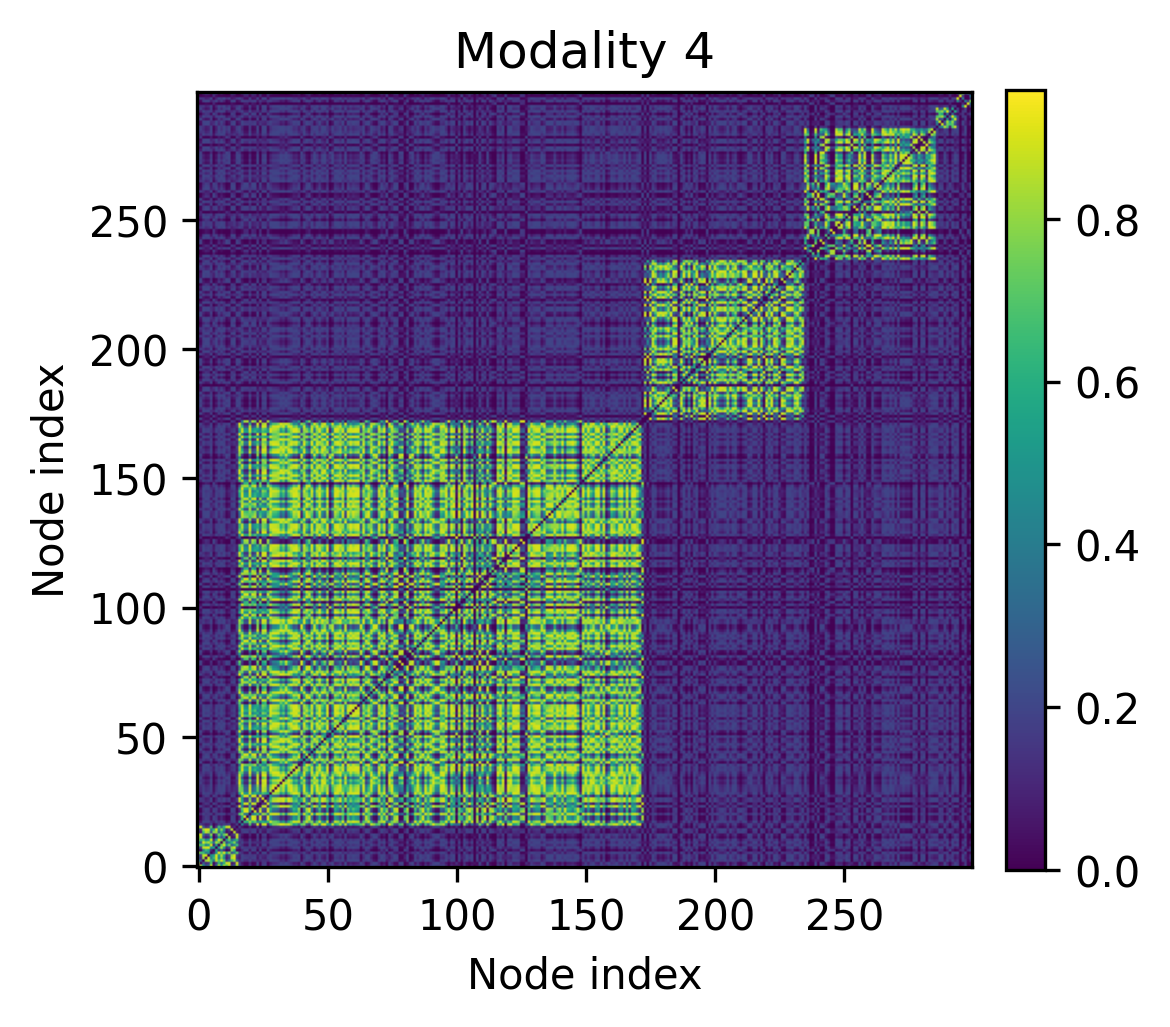}
        \vspace{2pt}
        {\footnotesize \\(d) Modality 4}
    \end{minipage}

    \caption{\footnotesize Adjacency matrix heatmaps for the $m=4$ modalities in the weighted SBM dataset. Nodes ordered by their ground-truth cluster assignments. Bright diagonal blocks correspond to strong intra-cluster connectivity, while darker off-diagonal regions indicate weaker inter-cluster connections.}
    \label{heatmaps}
\end{figure}


The subsequent datasets utilize multimodal data in the form of feature matrices $\mathbf{Z}^{(i)} \in \mathbb{R}^{N \times d_i}$, where $N$ is the number of samples and $d_i$ is the number of features in modality $i$.
We define the affinity matrix $\mathbf{W}^{(i)} \in \mathbb{R}^{N \times N}$ using a self-tuning Gaussian kernel \cite{zelnik2004self} in terms of its entries as in \cite{7053905, nataliani}:
\[
w^{(i)}_{pq} = \begin{cases}
\exp\left(-\dfrac{\|\mathbf{z}^{(i)}_p - \mathbf{z}^{(i)}_q\|^2}{\sigma_p \sigma_q}\right)~, & p \ne q \\
0~, & p = q.
\end{cases}
\]
where $\mathbf{z}^{(i)}_p$ is the $p$th column of $\mathbf{Z}^{(i)}$ and $\sigma_p$ is a local bandwidth parameter for each sample defined as the distance to its $k$-th nearest neighbor. This results in a fully connected, symmetric graph with adaptive Gaussian weights and zero diagonal. Given $\mathbf{W}^{(i)}$, we construct the corresponding symmetric normalized Laplacian as before.

\subsection{Caltech-7}

We consider a multimodal subset of the Caltech-101 image dataset, commonly referred to as \textit{Caltech-7} \cite{5995740,li_andreeto_ranzato_perona_2022, li2015large}. This benchmark consists of 1,474 images across 7 ground-truth categories.
Each image is represented in six distinct image feature modalities.

\subsection{Digits}

We consider a multimodal version of the UCI Optical Recognition of Handwritten Digits dataset \cite{liu2013multi}, commonly referred to as \textit{Digits}. This dataset consists of 3,823 grayscale images of 10 classes corresponding to handwritten digits 0 through 9. We manually extract two distinct feature representations for each image sample.

\subsection{Nutrimouse}

We consider the \textit{Nutrimouse} dataset \cite{nutrimouse}, which arises from a nutrigenomics study of mice and is commonly used as a benchmark for multiview learning \cite{parea}. It contains $N=40$ mouse samples, each with two modalities: gene expression data and fatty acid concentration profiles. Each mouse is labeled according to one of $k=5$ diet categories.

\subsection{MSRC}

We construct a multimodal image dataset derived from the MSRC Object Recognition Image Database (MSRC-ORID) \cite{msrc}. We use $k=7$ object categories and sample an equal number of images per category, resulting in a total of $N=210$ images. We represent each image with four common visual feature modalities.

\subsection{Nonlinear Gaussian Mixture (NGM)}

We generate a synthetic two-modality dataset using the \texttt{make\_gaussian\_mixture} generator from the \texttt{mvlearn} library \cite{perry2021mvlearn}. We follow the construction described in \cite{perry2021mvlearn} and set $N=5000$ samples and $k=2$ clusters.

}
\ifthenelse{\boolean{arxiv}}
  {\section{NUMERICAL EXPERIMENTS}
\label{exp}

Following standard practice, we evaluate clustering quality with normalized mutual information (NMI), which rescales mutual information by the label entropies so that scores lie in $[0,1]$ ($1$ being perfect agreement and $0$ being independence) \cite{scikit-learn}.

\subsection{Direct Optimization of Smoothness Objectives}
\label{direct_opt}

Our aim is to test whether optimizing for single-directional smoothness or BASE smoothness produces better clustering embeddings.
We apply projected gradient ascent over the standard simplex $\Delta^{m-1}$ (the space of valid weight vectors). That is, each update takes a step in the direction of the gradient to increase the objective, followed by a projection back onto the feasible set to maintain constraints \cite{manopt}. We initialize with uniform weights, at each iteration compute either $\lambda_1(\mathbf{L(\pmb{\mu})})$ (for the single-directional smoothness formulation) or $\sum_{j=1}^{k} \lambda_j(\mathbf{L(\pmb{\mu})})$ (for the BASE smoothness formulation) on the convex combination $\mathbf{L(\pmb{\mu})}$, take a step along the gradient of the objective with respect to $\pmb{\mu}$, and  project back onto the simplex via Euclidean projection.
We extract the bottom $k$ eigenvectors of $\mathbf{L(\pmb{\mu}^*)}$ and perform $k$-means clustering on the resulting embedding at each step, plotting the normalized mutual information (NMI) vs. the smoothness objective.

Figures~\ref{fig:vector_synth} -~\ref{fig:vector_mix} show the clustering performance (NMI) as a function of the single-directional smoothness objective $\lambda_1(\mathbf{L(\pmb{\mu})})$ on all six of our datasets. The final NMIs of the direct optimization of single-directional smoothness on each of the datasets can be seen in the figures of in Table~\ref{tab:all_methods_nmi}.

\begin{figure}[H]
    \centering
    \includegraphics[width=3.5in]{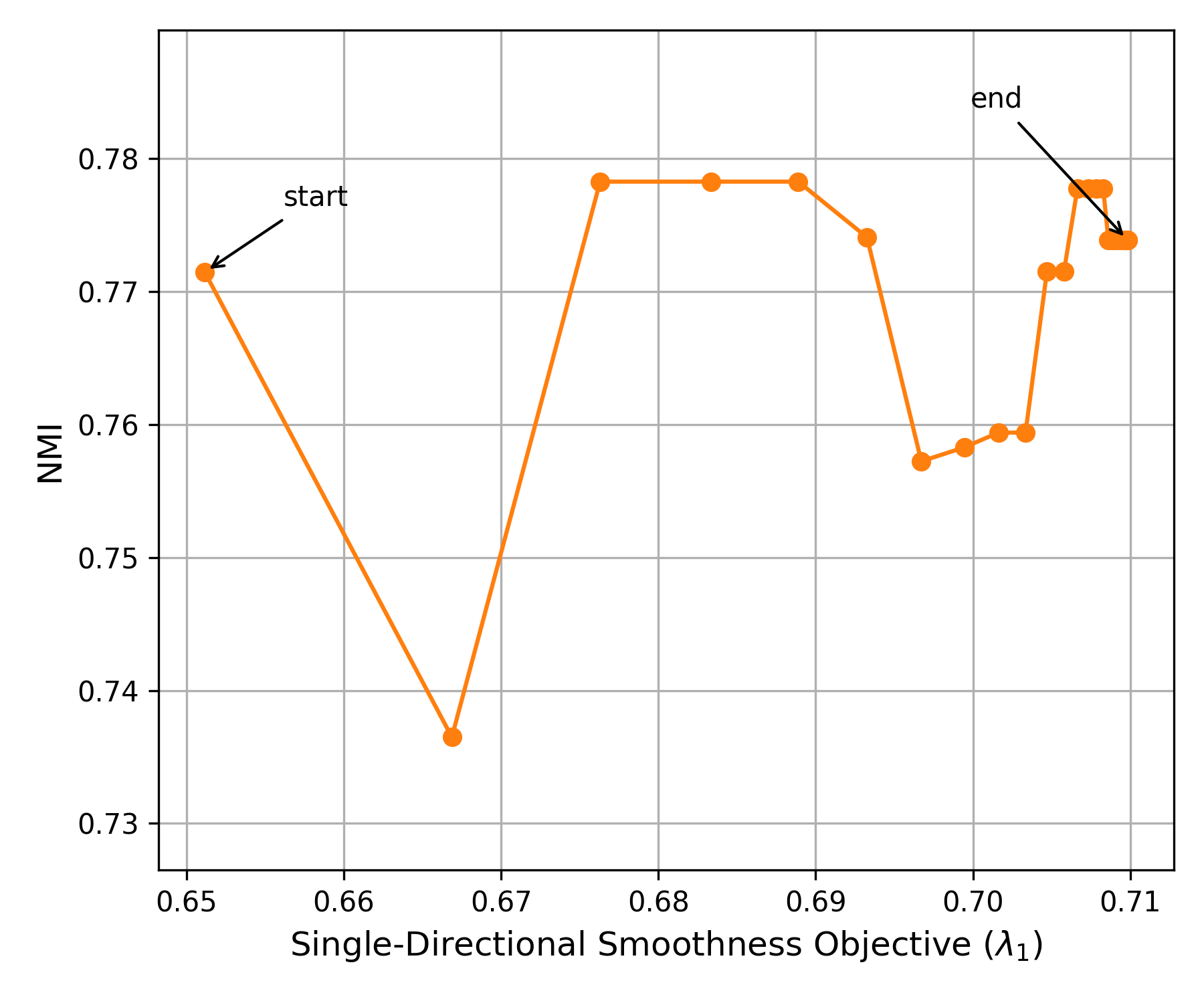}
    \caption{30-iteration direct optimization of single-directional smoothness objective on SBM dataset: NMI vs. second smallest eigenvalue $\lambda_1(\mathbf{L(\pmb{\mu})})$.}
    \label{fig:vector_synth}
\end{figure}

\begin{figure}[H]
    \centering
    \includegraphics[width=3.5in]{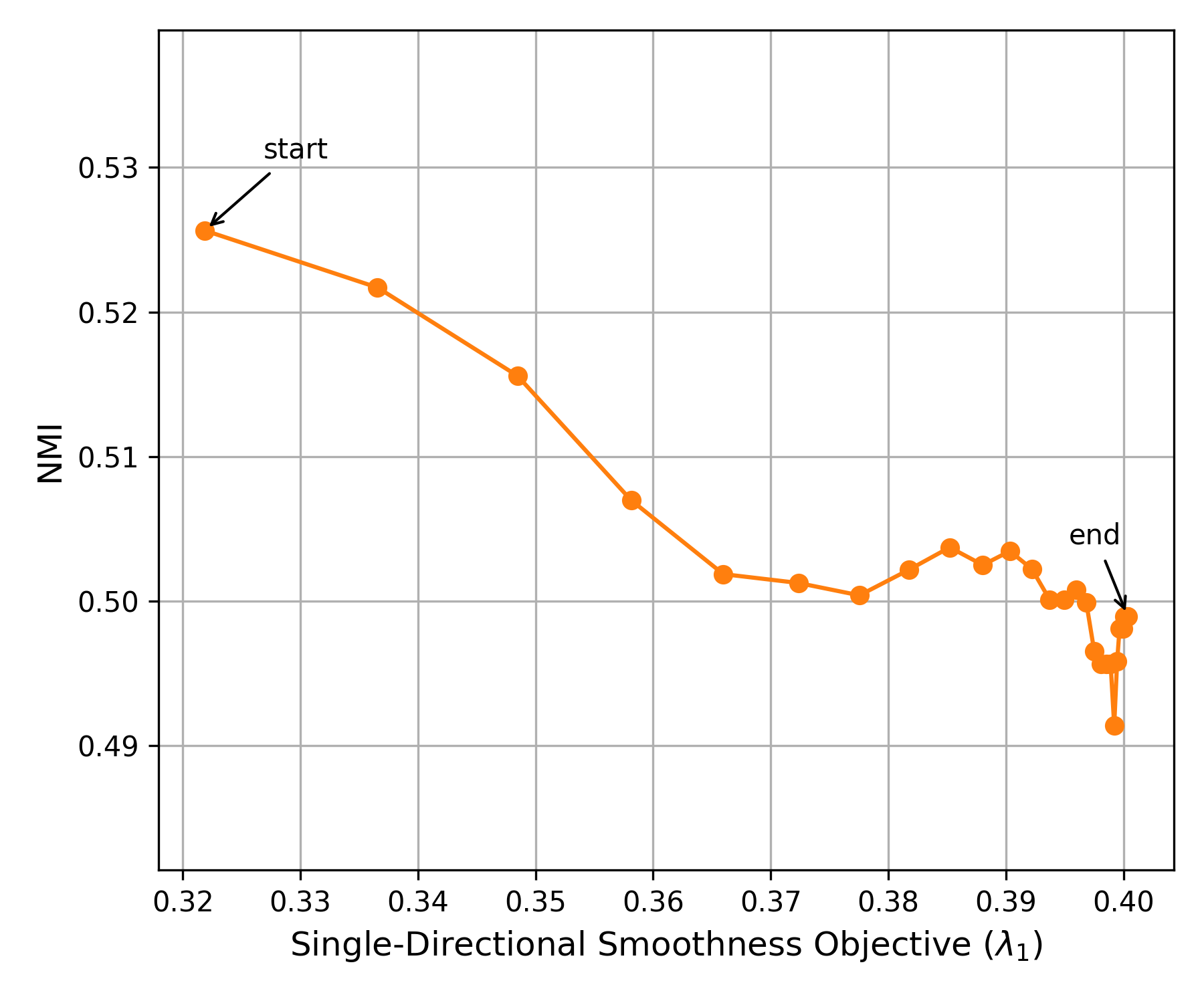}
    \caption{30-iteration direct optimization of single-directional smoothness objective on Caltech-7 dataset: NMI vs. second smallest eigenvalue $\lambda_1(\mathbf{L(\pmb{\mu})})$.}
    \label{fig:vector_cal}
\end{figure}

\begin{figure}[H]
    \centering
    \includegraphics[width=3.5in]{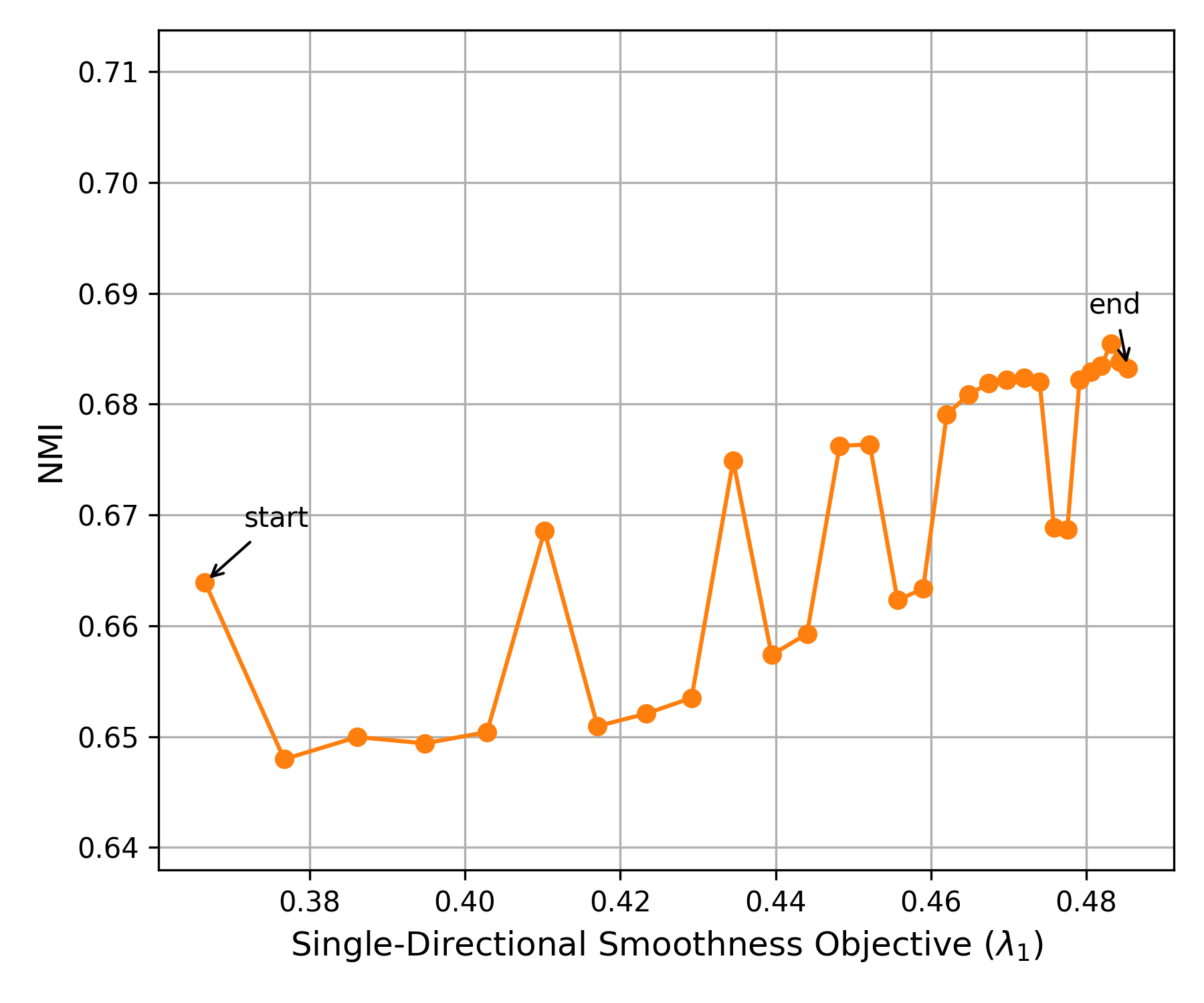}
    \caption{30-iteration direct optimization of single-directional smoothness objective on Digits dataset: NMI vs. second smallest eigenvalue $\lambda_1(\mathbf{L(\pmb{\mu})})$.}
    \label{fig:vector_digits}
\end{figure}

\begin{figure}[H]
    \centering
    \includegraphics[width=3.5in]{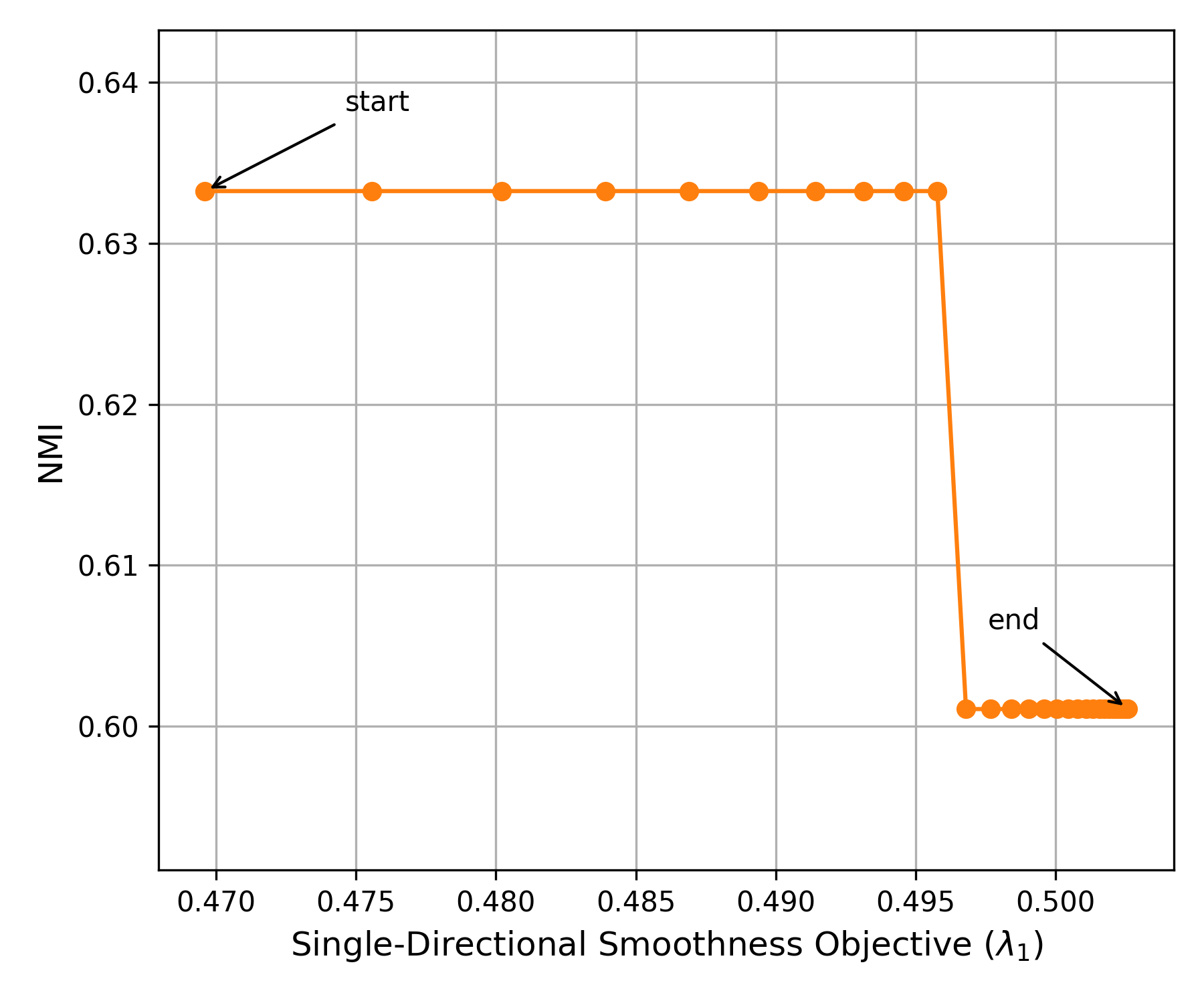}
    \caption{30-iteration direct optimization of single-directional smoothness objective on Nutrimouse dataset: NMI vs. second smallest eigenvalue $\lambda_1(\mathbf{L(\pmb{\mu})})$.}
    \label{fig:vector_mouse}
\end{figure}

\begin{figure}[H]
    \centering
    \includegraphics[width=3.5in]{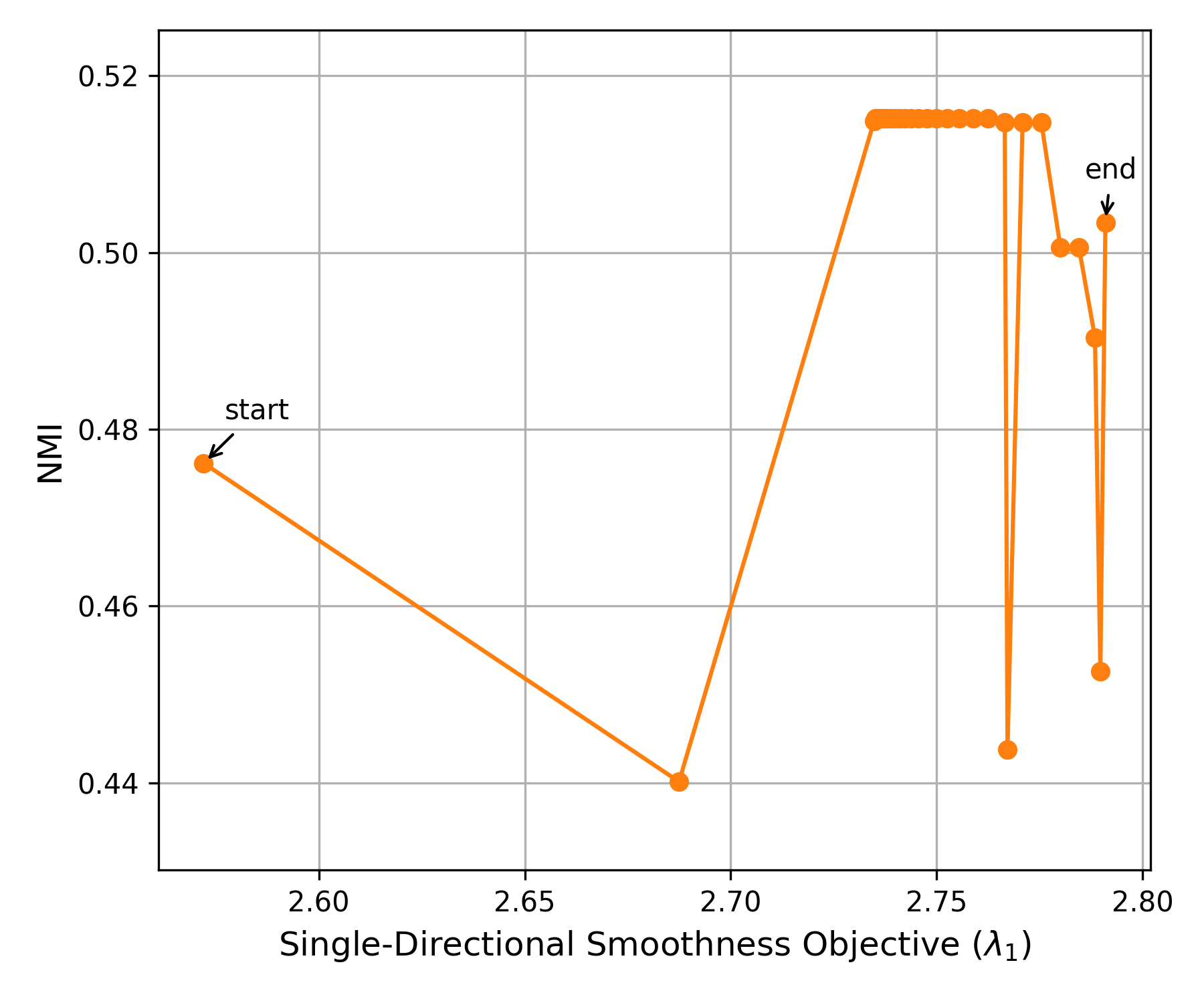}
    \caption{30-iteration direct optimization of single-directional smoothness objective on MSRC dataset: NMI vs. second smallest eigenvalue $\lambda_1(\mathbf{L(\pmb{\mu})})$.}
    \label{fig:vector_msrc}
\end{figure}

\begin{figure}[H]
    \centering
    \includegraphics[width=3.5in]{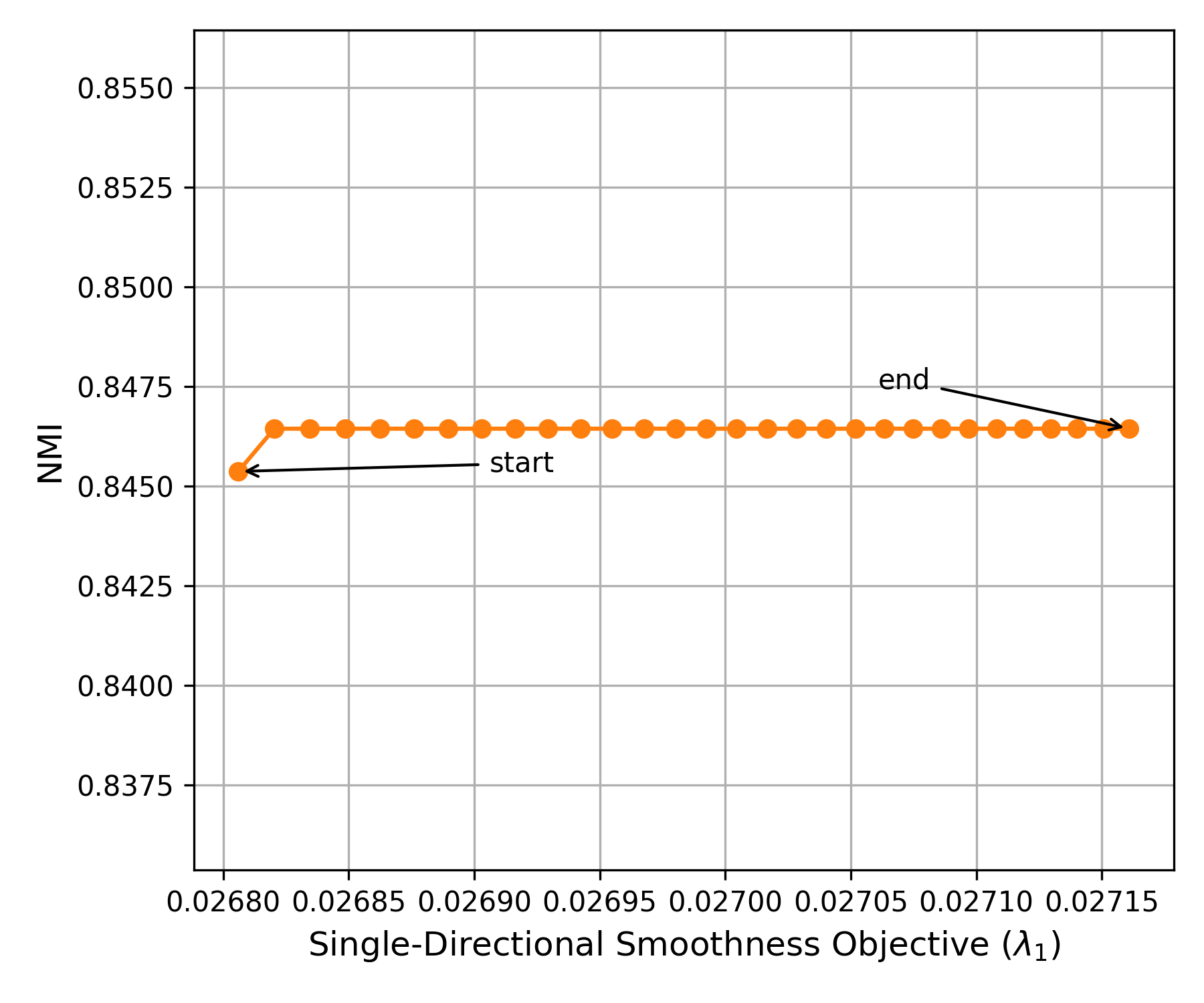}
    \caption{30-iteration direct optimization of single-directional smoothness objective on NGM dataset: NMI vs. second smallest eigenvalue $\lambda_1(\mathbf{L(\pmb{\mu})})$.}
    \label{fig:vector_mix}
\end{figure}

Figures~\ref{fig:matrix_synth} - ~\ref{fig:matrix_mix}report the NMI obtained during direct optimization of the BASE smoothness objective $\sum_{j=1}^{k} \lambda_j(\mathbf{L(\pmb{\mu}))}$. On the Digits, Nutrimouse, and NGM datasets, the final NMI matches that of the single-directional approach,  indicating that both formulations are equally effective in this case. On the MSRC dataset, the BASE smoothness objective optimization falls a bit behind the single-directional one, but on the SBM and Caltech-7 datasets, it surpasses the single-directional objective optimization, displaying a benefit in directly targeting the full bottom-$k$ subspace.

\begin{figure}[H]
    \centering
    \includegraphics[width=3.5in]{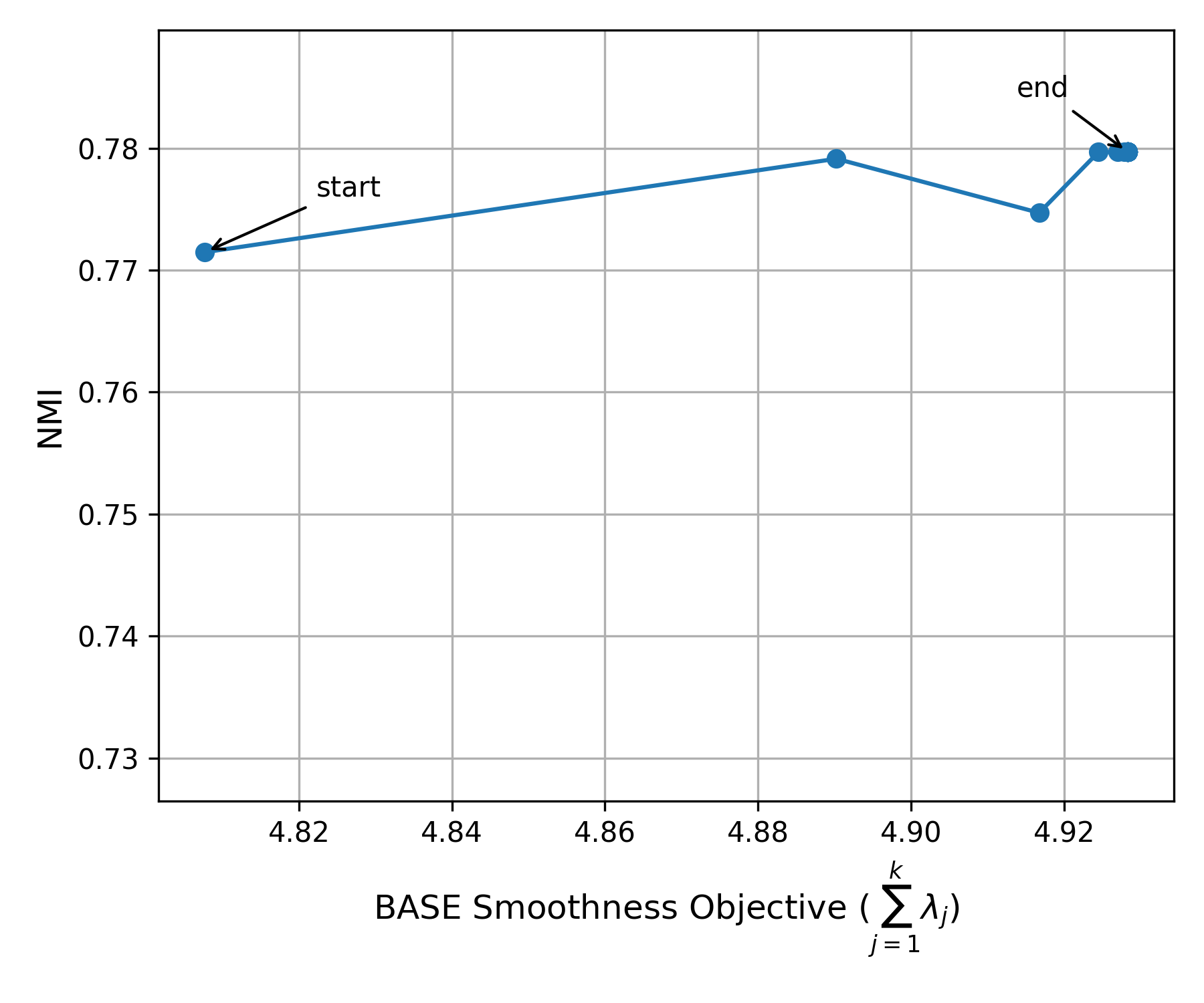}
    \caption{30-iteration direct optimization of BASE smoothness objective on SBM dataset: NMI vs. $\sum_{j=1}^{k} \lambda_j(\mathbf{L(\pmb{\mu})})$.}
    \label{fig:matrix_synth}
\end{figure}

\begin{figure}[H]
    \centering
    \includegraphics[width=3.5in]{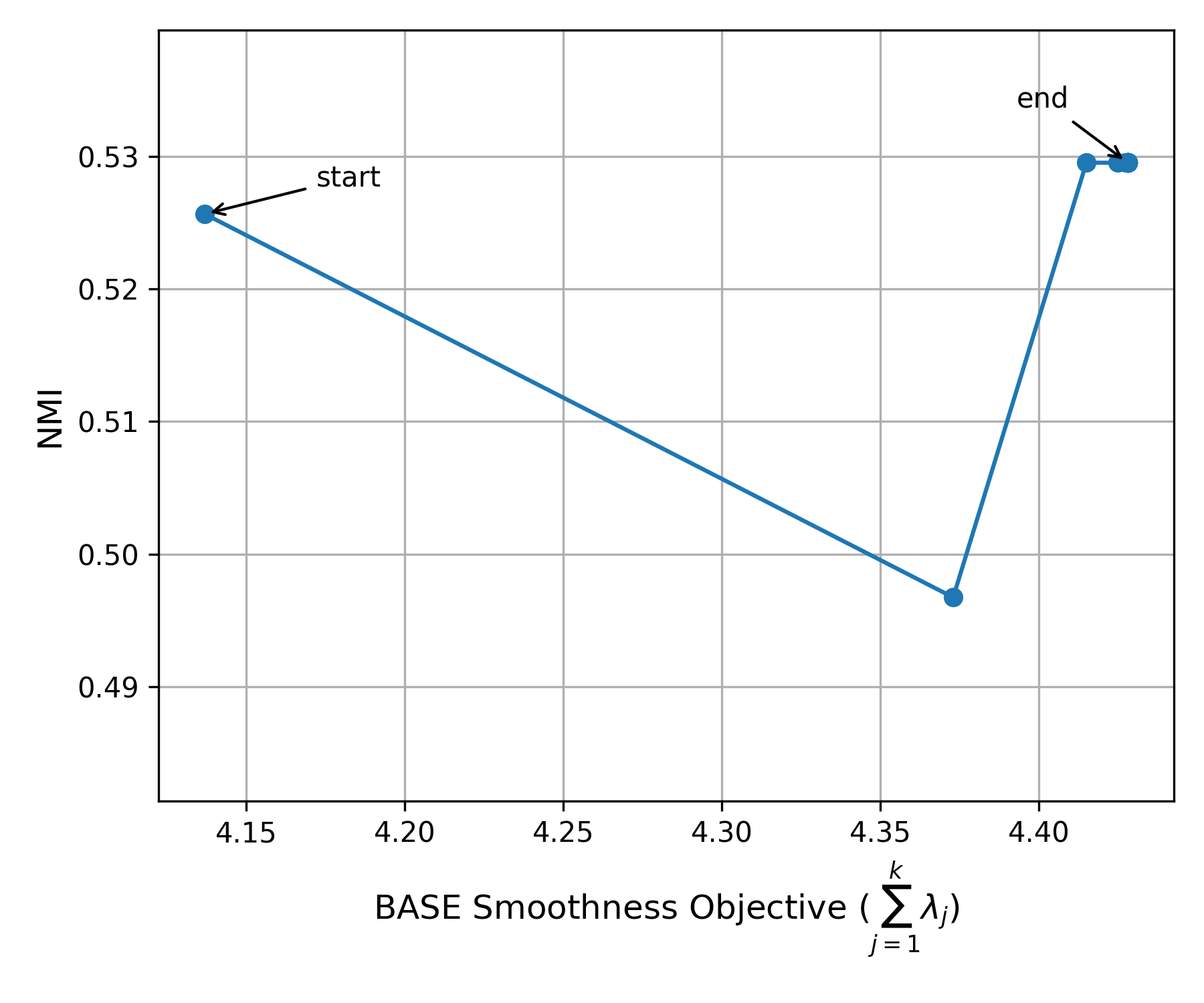}
    \caption{30-iteration direct optimization of BASE smoothness objective on Caltech-7 dataset: NMI vs. $\sum_{j=1}^{k} \lambda_j(\mathbf{L(\pmb{\mu})})$.}
    \label{fig:matrix_cal}
\end{figure}

\begin{figure}[H]
    \centering
    \includegraphics[width=3.5in]{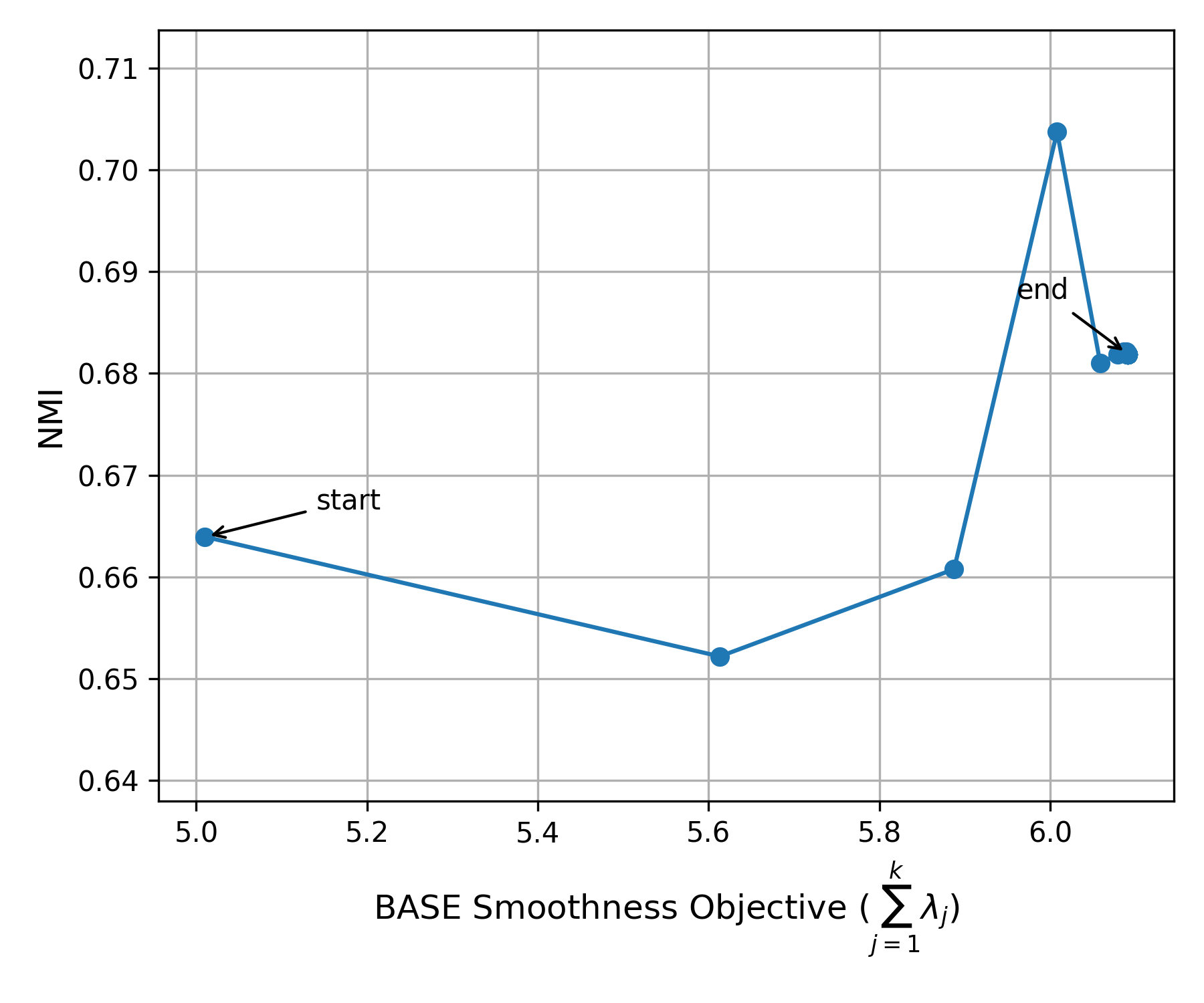}
    \caption{30-iteration direct optimization of BASE smoothness objective on Digits dataset: NMI vs. $\sum_{j=1}^{k} \lambda_j(\mathbf{L(\pmb{\mu})})$.}
    \label{fig:matrix_digits}
\end{figure}

\begin{figure}[H]
    \centering
    \includegraphics[width=3.5in]{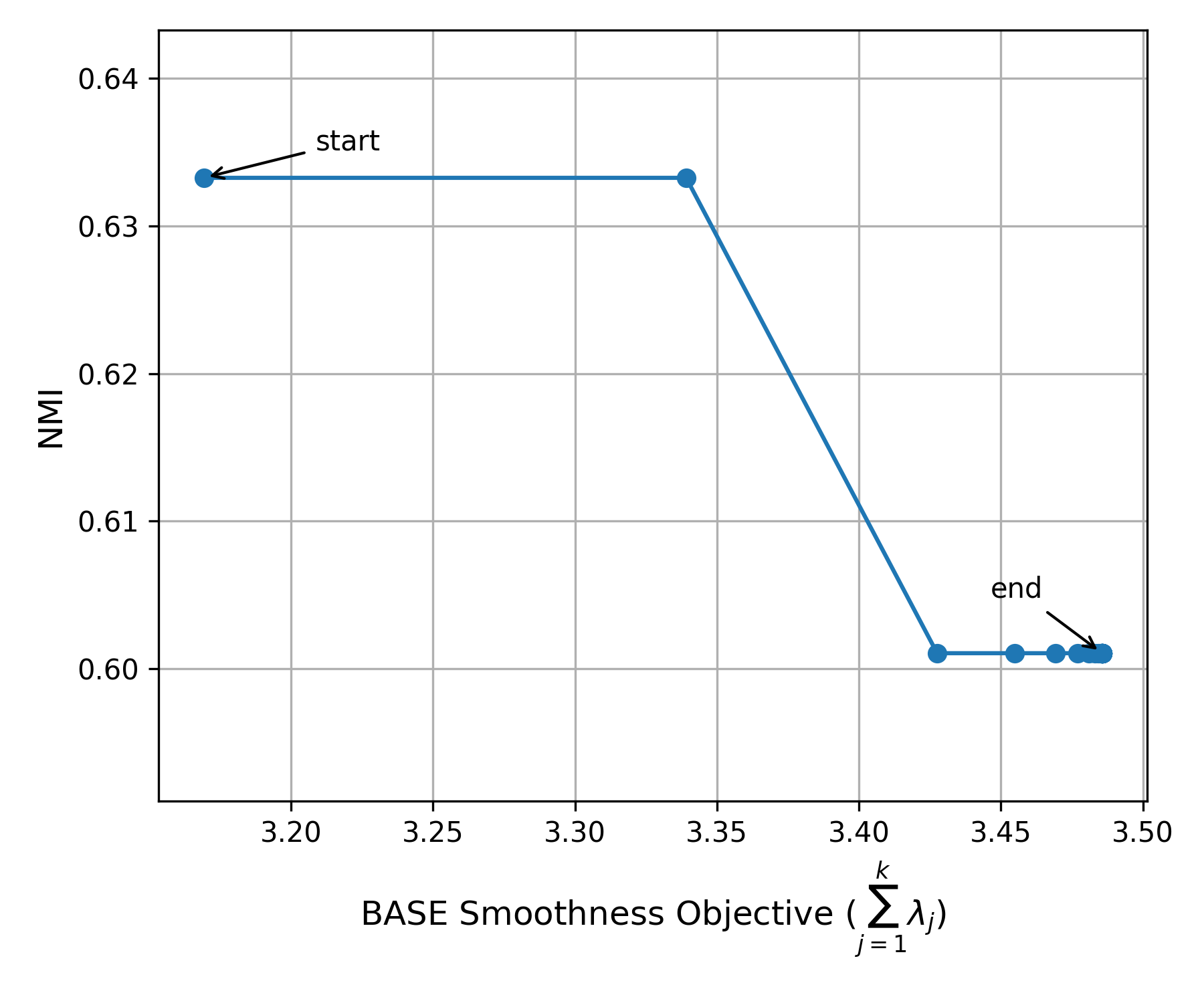}
    \caption{30-iteration direct optimization of BASE smoothness objective on Nutrimouse dataset: NMI vs. $\sum_{j=1}^{k} \lambda_j(\mathbf{L(\pmb{\mu})})$.}
    \label{fig:matrix_mouse}
\end{figure}

\begin{figure}[H]
    \centering
    \includegraphics[width=3.5in]{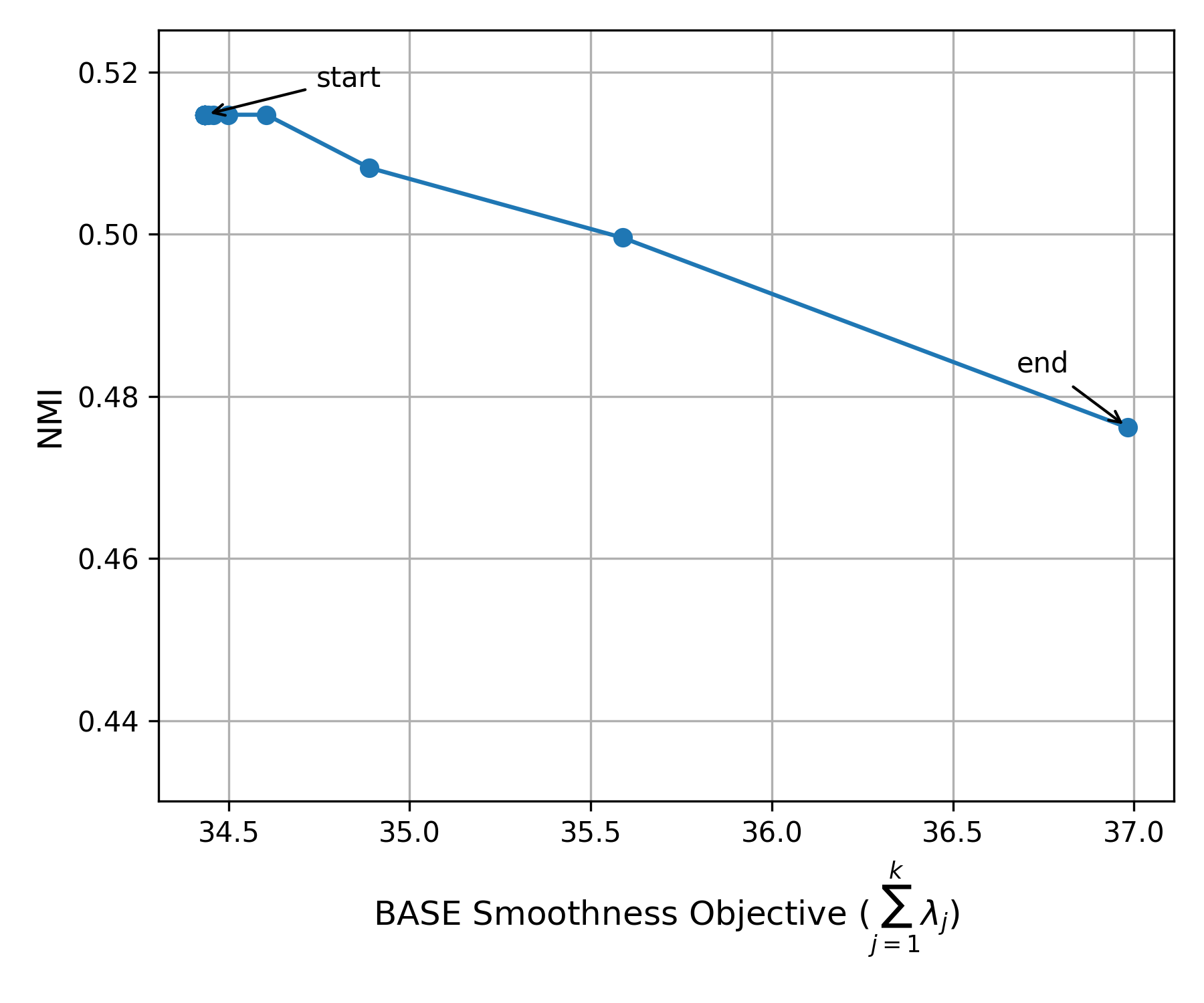}
    \caption{30-iteration direct optimization of BASE smoothness objective on MSRC dataset: NMI vs. $\sum_{j=1}^{k} \lambda_j(\mathbf{L(\pmb{\mu})})$.}
    \label{fig:matrix_msrc}
\end{figure}

\begin{figure}[H]
    \centering
    \includegraphics[width=3.5in]{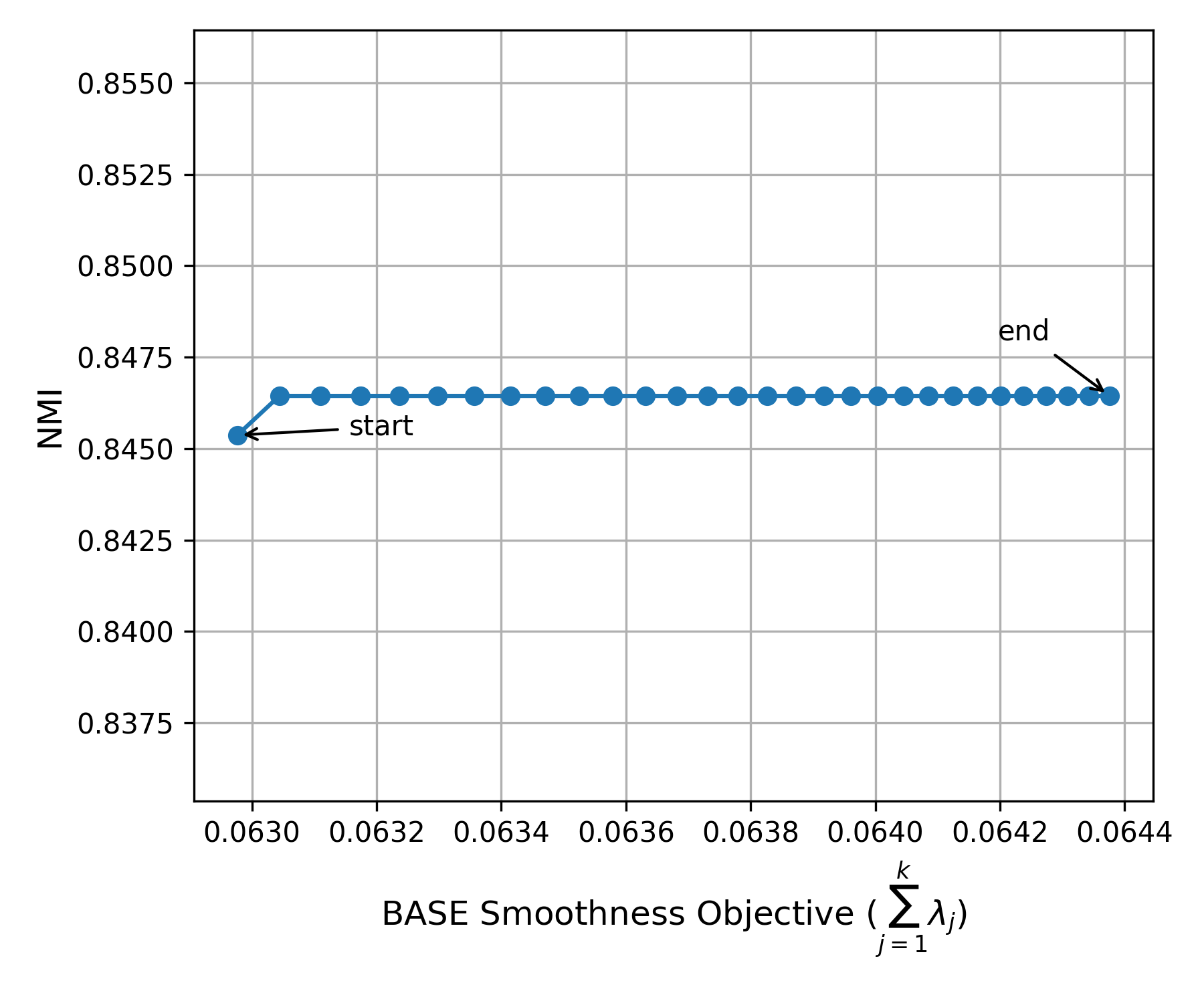}
    \caption{30-iteration direct optimization of BASE smoothness objective on NGM dataset: NMI vs. $\sum_{j=1}^{k} \lambda_j(\mathbf{L(\pmb{\mu})})$.}
    \label{fig:matrix_mix}
\end{figure}

\subsection{RJD-BASE: Trial Landscape and Selection}
\label{rjd_landscape}

Direct optimization of the smoothness objectives provides high-quality embeddings, especially in the BASE objective case. Here, in an effort to simplify and parallelize, we evaluate whether the BASE smoothness objective can serve as an effective selection criterion across many RJD trials.

We run RJD-BASE for T=3000 and plot NMI against the BASE smoothness objective $\sum_{j=1}^{k} \lambda_j(\mathbf{L(\pmb{\mu})})$ for each RJD instance. We also mark the mean NMI point and the point selected by RJD-BASE.

In Figures \ref{synth_scatter} - \ref{mix_scatter}, we see that over 3000 trials in all three datasets, RJD-BASE yields an above average RJD instance selection with respect to the end-goal NMI.

\begin{figure}[H]
    \centering
    \includegraphics[width=3.5in]{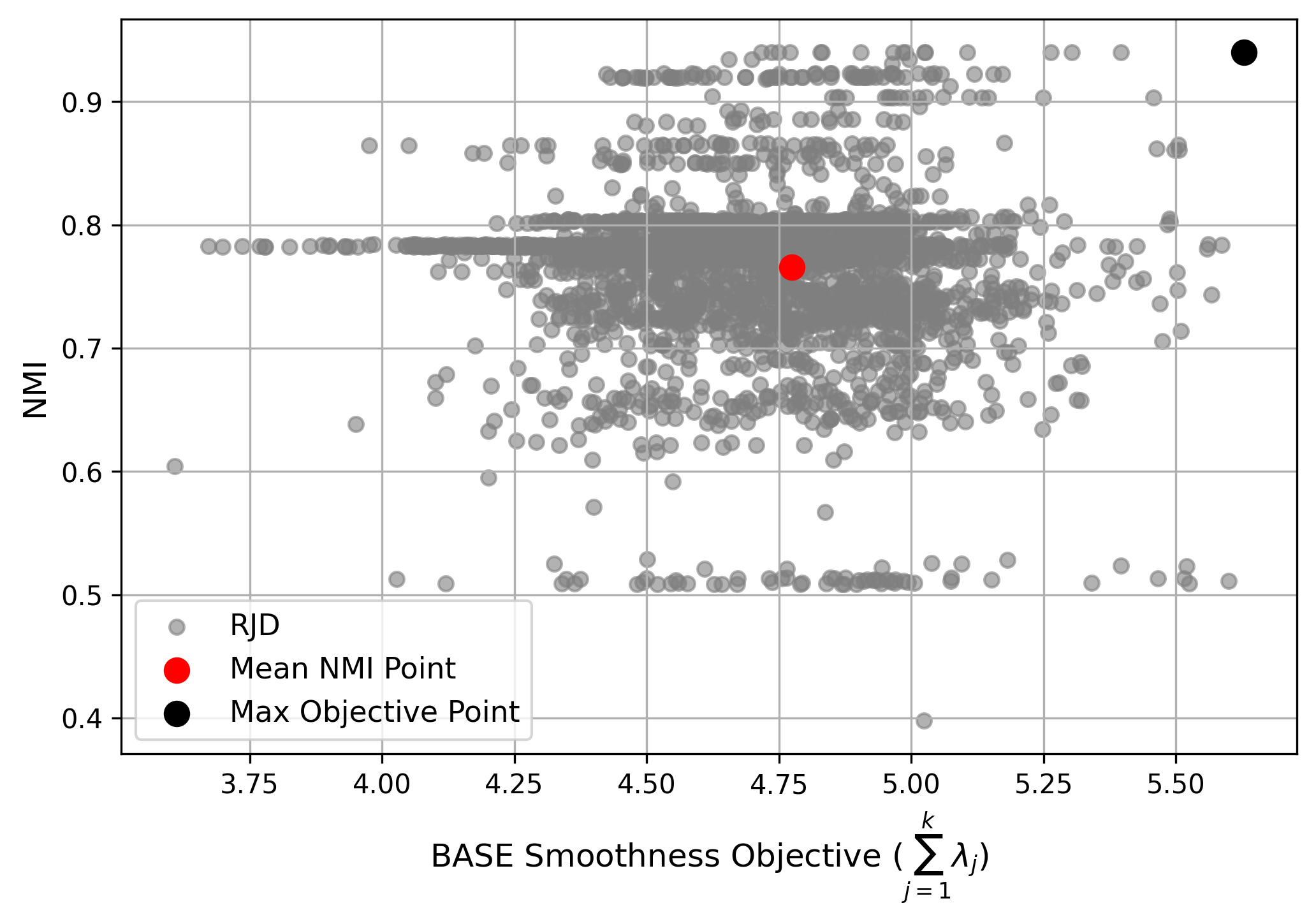}
    \caption{Scatter plot of NMI vs. BASE smoothness objective for 3000 independent RJD instance on the weighted SBM dataset. Mean NMI point indicated in red and point maximizing BASE smoothness objective (i.e. that which would be selected by RJD-BASE) in black.  }
    \label{synth_scatter}
\end{figure}

\begin{figure}[H]
    \centering
    \includegraphics[width=3.5in]{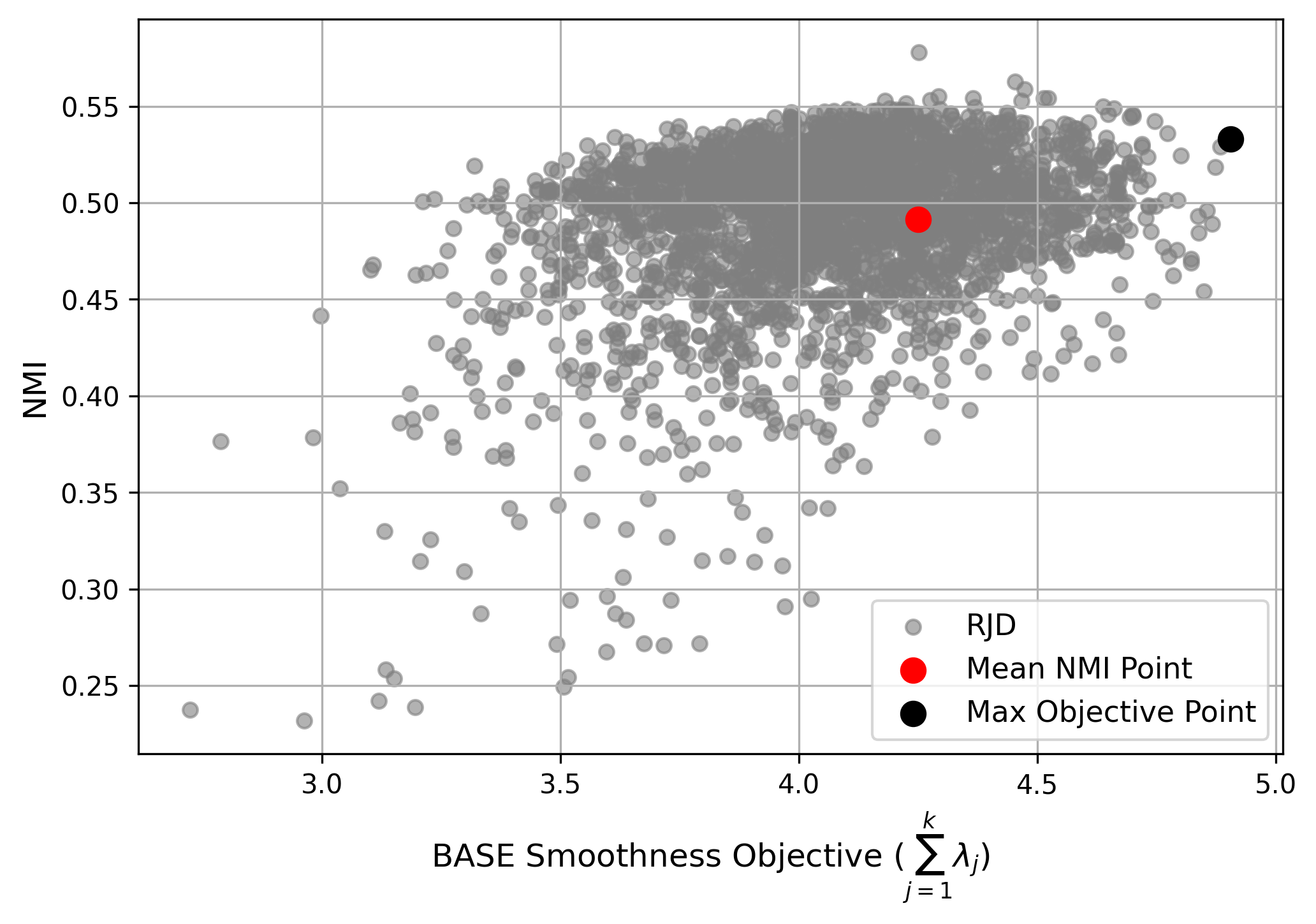}
    \caption{Scatter plot of NMI vs. BASE smoothness objective for 3000 independent RJD instance on the Caltech-7 dataset. Mean NMI point indicated in red and point maximizing BASE smoothness objective (i.e. that which would be selected by RJD-BASE) in black.}
    \label{cal_scatter}
\end{figure}

\begin{figure}[H]
    \centering
    \includegraphics[width=3.5in]{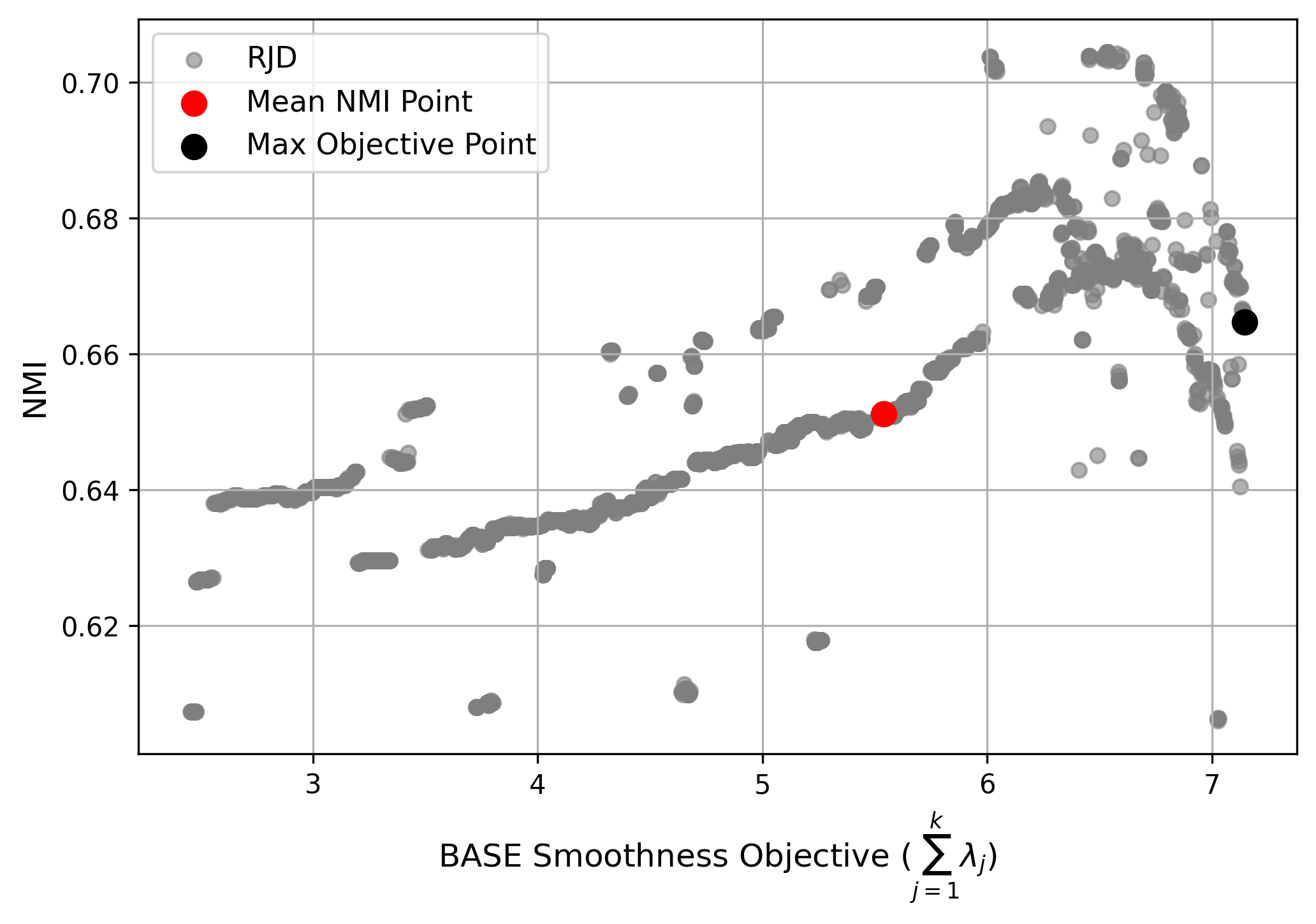}
    \caption{Scatter plot of NMI vs. BASE smoothness objective for 3000 independent RJD instance on the Digits dataset. Mean NMI point indicated in red and point maximizing BASE smoothness objective (i.e. that which would be selected by RJD-BASE) in black.}
    \label{digits_scatter}
\end{figure}

\begin{figure}[H]
    \centering
    \includegraphics[width=3.5in]{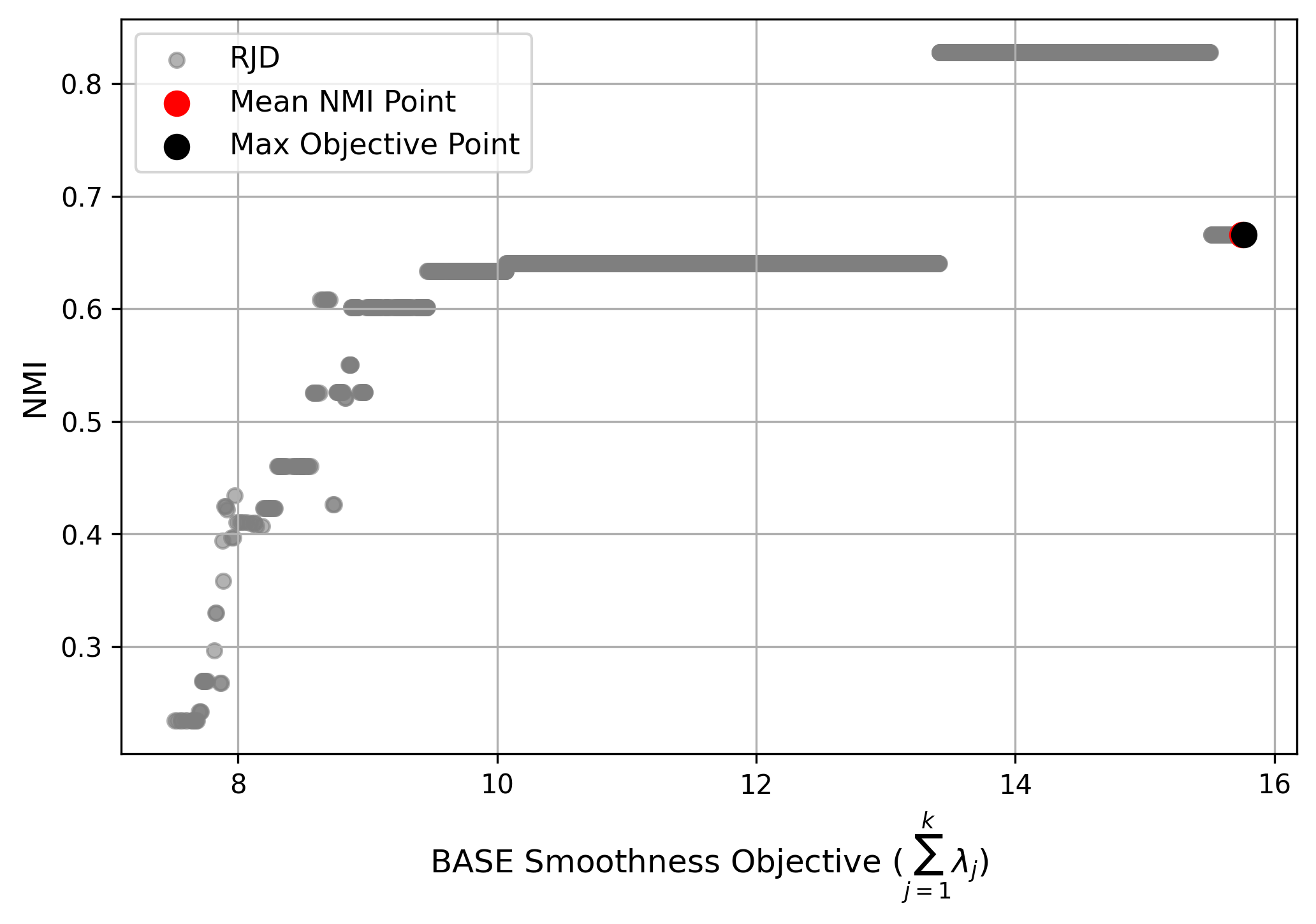}
    \caption{Scatter plot of NMI vs. BASE smoothness objective for 3000 independent RJD instance on the Nutrimouse dataset. Mean NMI point indicated in red and point maximizing BASE smoothness objective (i.e. that which would be selected by RJD-BASE) in black.}
    \label{mouse_scatter}
\end{figure}

\begin{figure}[H]
    \centering
    \includegraphics[width=3.5in]{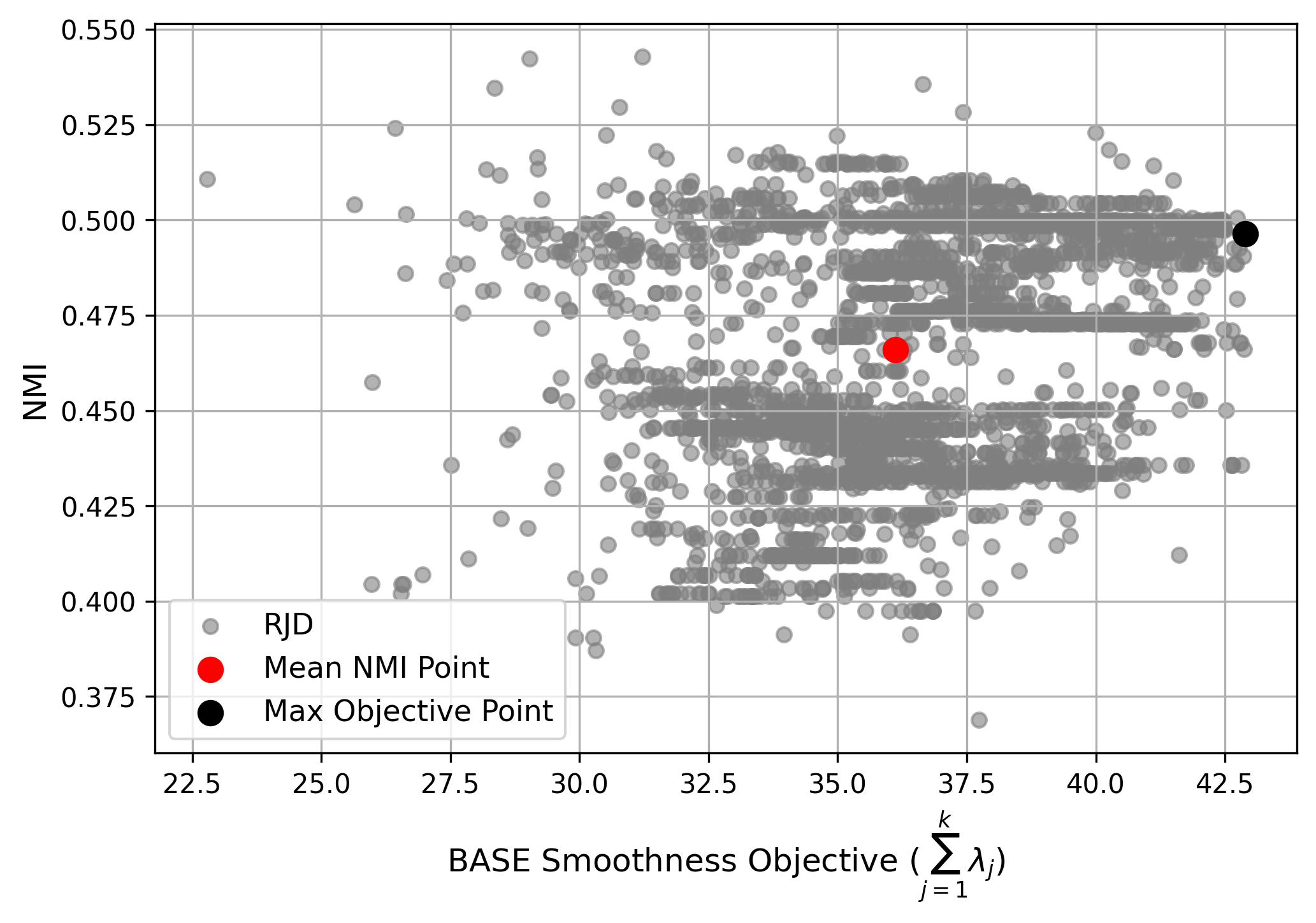}
    \caption{Scatter plot of NMI vs. BASE smoothness objective for 3000 independent RJD instance on the MSRC dataset. Mean NMI point indicated in red and point maximizing BASE smoothness objective (i.e. that which would be selected by RJD-BASE) in black.}
    \label{msrc_scatter}
\end{figure}

\begin{figure}[H]
    \centering
    \includegraphics[width=3.5in]{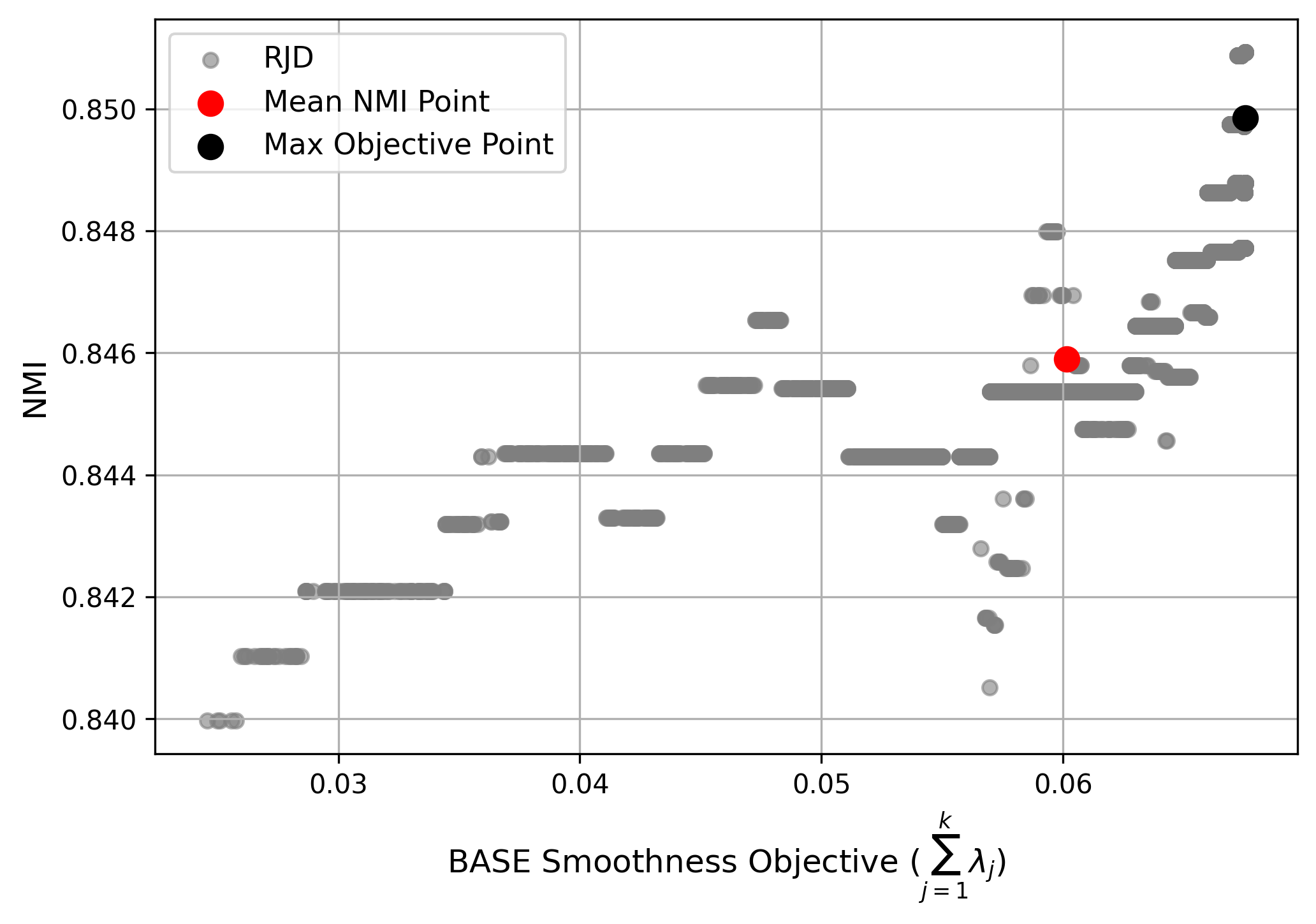}
    \caption{Scatter plot of NMI vs. BASE smoothness objective for 3000 independent RJD instance on the NGM dataset. Mean NMI point indicated in red and point maximizing BASE smoothness objective (i.e. that which would be selected by RJD-BASE) in black.}
    \label{mix_scatter}
\end{figure}

To more concretely quantify the effectiveness of this selection rule, we perform 1000 trials where in each trial we run RJD-BASE with $T=10$ and record whether the selected embedding’s NMI is above the global mean. This yields an empirical estimate of how often RJD-BASE beats a random draw in expectation even with very small $T$. The weighted SBM, Caltech-7, Digits, Nutrimouse, MSRC, and NGM datasets achieved 57\%, 66\%, 96\%, 76\%, 85\%, and 99\% above-average embeddings, respectively.


\subsection{RJD-BASE with QN-Diag and JADE}
\label{rjdbase_qn_diag_jade}

While approximate joint diagonalization algorithms such as QN-Diag and JADE are often employed in multimodal and blind source separation settings for downstream clustering or dimensionality reduction, they operate by optimizing global off-diagonal energy across the full spectrum of eigenvectors. In the context of spectral clustering, however, we observe that such refinement is often unnecessary and, in some cases, actively counterproductive. This is because clustering relies specifically on the structure of the bottom $k$ eigenvectors, and full-basis diagonalization may distort this subspace. In this section, we demonstrate that RJD-BASE, despite its simplicity and lack of iterative refinement, outperforms both QN-Diag and JADE.

We conduct the following experiment on our three datasets: We run RJD-BASE with $T=200$ and use the output embedding as the initialization to QN-Diag and JADE.  Note that although RJD-BASE directly produces only the $N \times k$ matrix of bottom-$k$ eigenvectors, iterative joint diagonalization methods such as QN-Diag and JADE operate on a full $N \times N$ orthogonal basis. To bridge this, we take the complete $N \times N$ eigenvector matrix from the selected RJD-BASE trial - the linear combination achieving the highest BASE smoothness objective - and use the full eigendecomposition of it to initialize the iterative method.
These algorithms then internally order the $N$ output vectors before extracting the bottom-$k$ subspace according to the the average of Rayleigh quotients over the modes. This reordering is part of the standard procedure to align the spectrum across modalities, but can change which $k$ directions are selected for clustering.

We track the NMI at each iteration of QN-Diag and JADE and plot the resulting learning curves. For reference, all 200 RJD embeddings are included at iteration index 0, enabling a direct comparison between the spread of randomized trials and the convergence behavior of the iterative methods. This setup serves to test whether QN-Diag and JADE can improve upon a reasonable, data-driven initialization and also lets us evaluate whether iterative algorithms such as QN-Diag and JADE can act as genuine refinement steps or simply alter the spectral subspace in ways that are misaligned with clustering objectives.


\begin{figure}[h]
    \centering
    \includegraphics[width=3.5in]{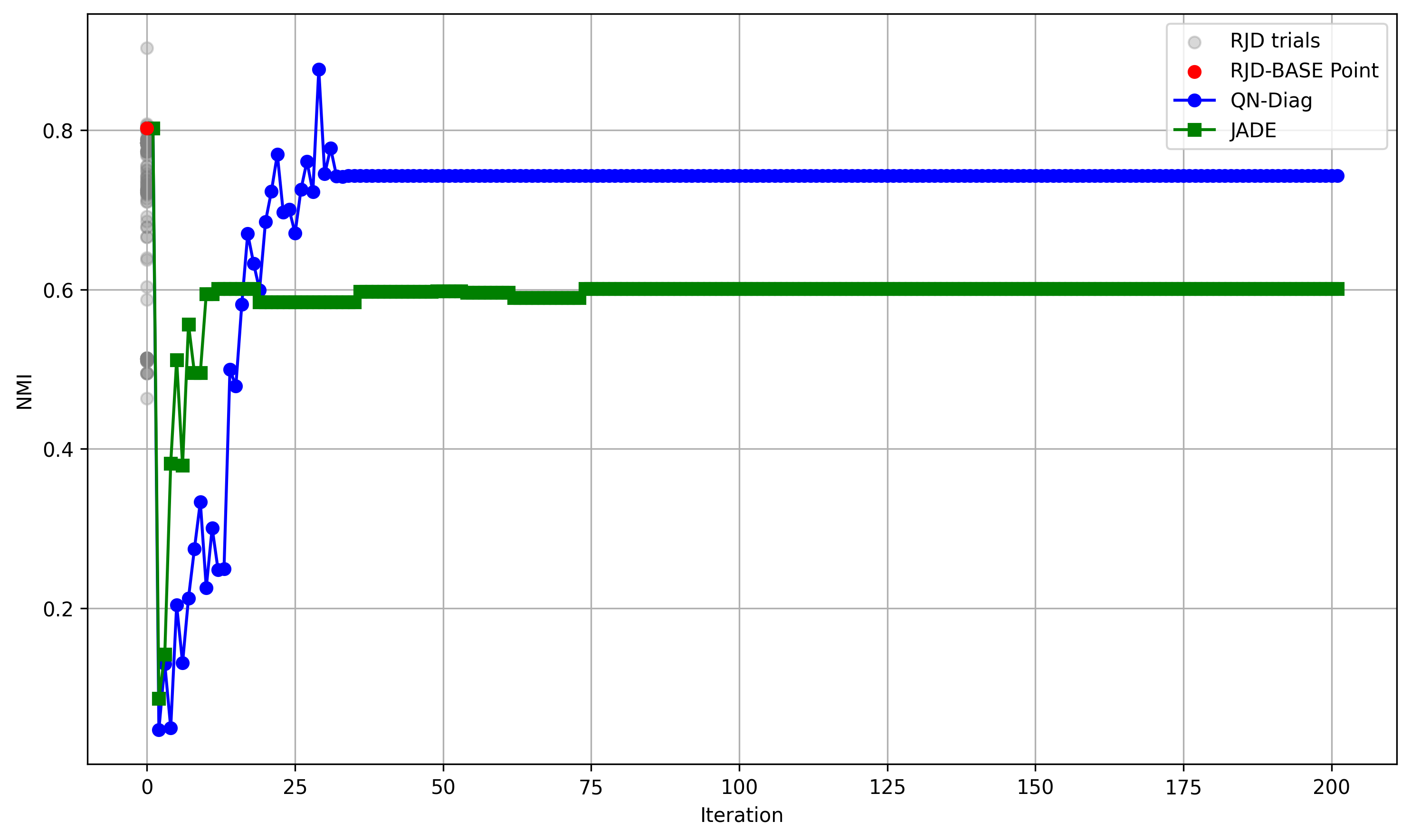}
    \caption{NMI convergence of QN-Diag and JADE when initialized with RJD-BASE on the weighted SBM dataset.}
    \label{synth_conv}
\end{figure}

\begin{figure}[h]
    \centering
    \includegraphics[width=3.5in]{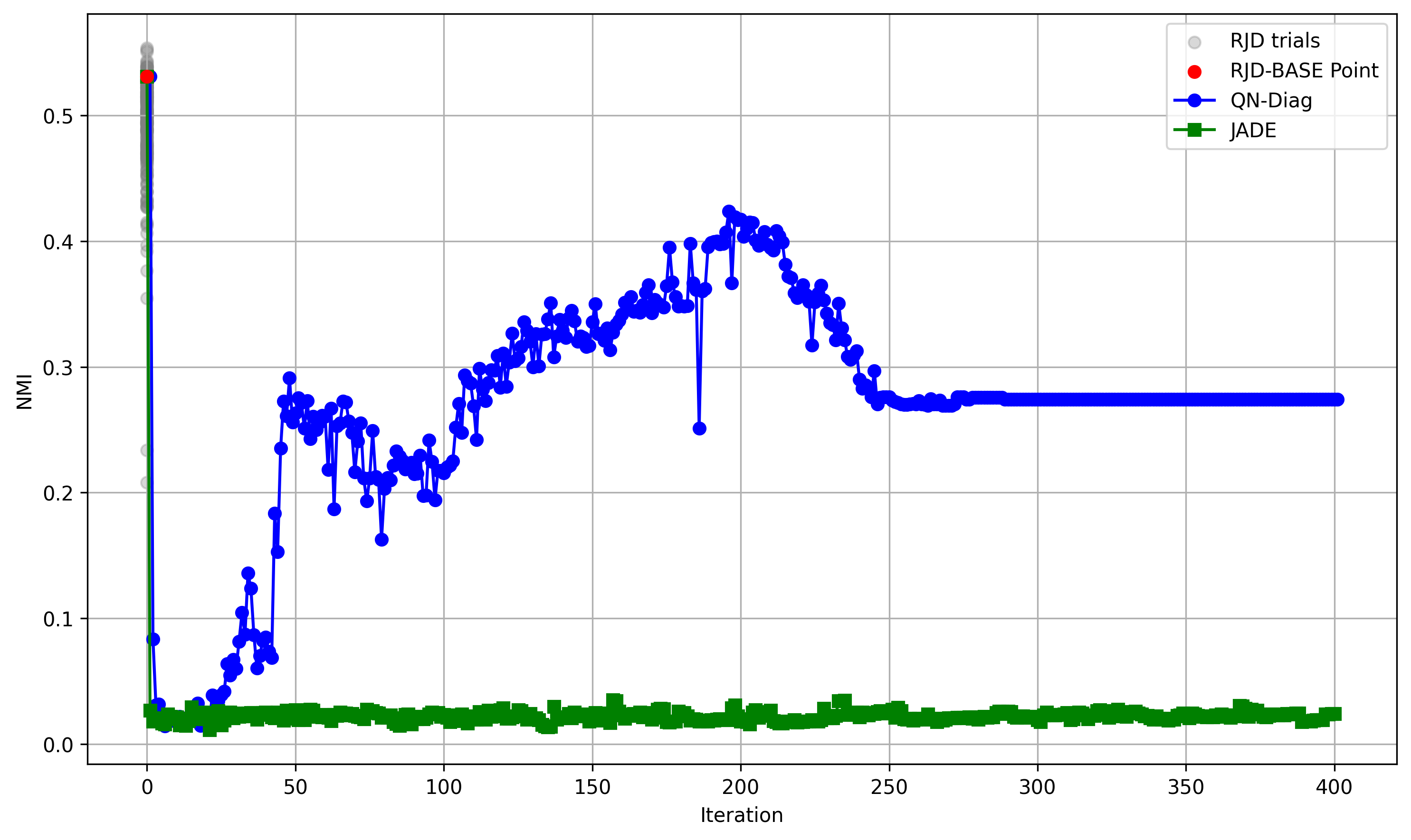}
    \caption{NMI convergence of QN-Diag and JADE when initialized with RJD-BASE on the Caltech-7 dataset.}
    \label{cal_conv}
\end{figure}

\begin{figure}[h]
    \centering
    \includegraphics[width=3.5in]{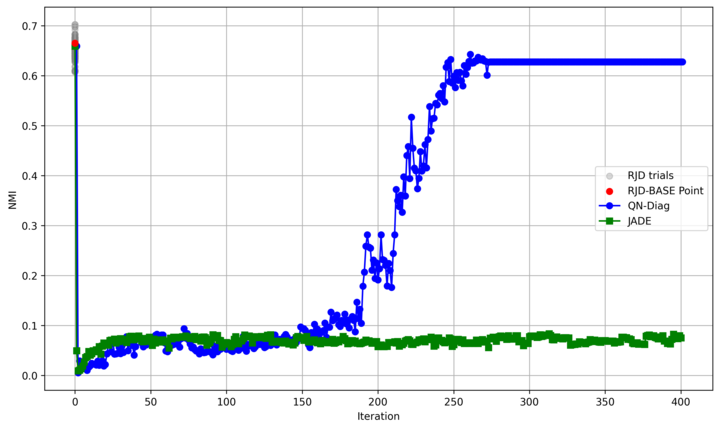}
    \caption{NMI convergence of QN-Diag and JADE when initialized with RJD-BASE on the Digits dataset.}
    \label{digits_conv}
\end{figure}

\begin{figure}[h]
    \centering
    \includegraphics[width=3.5in]{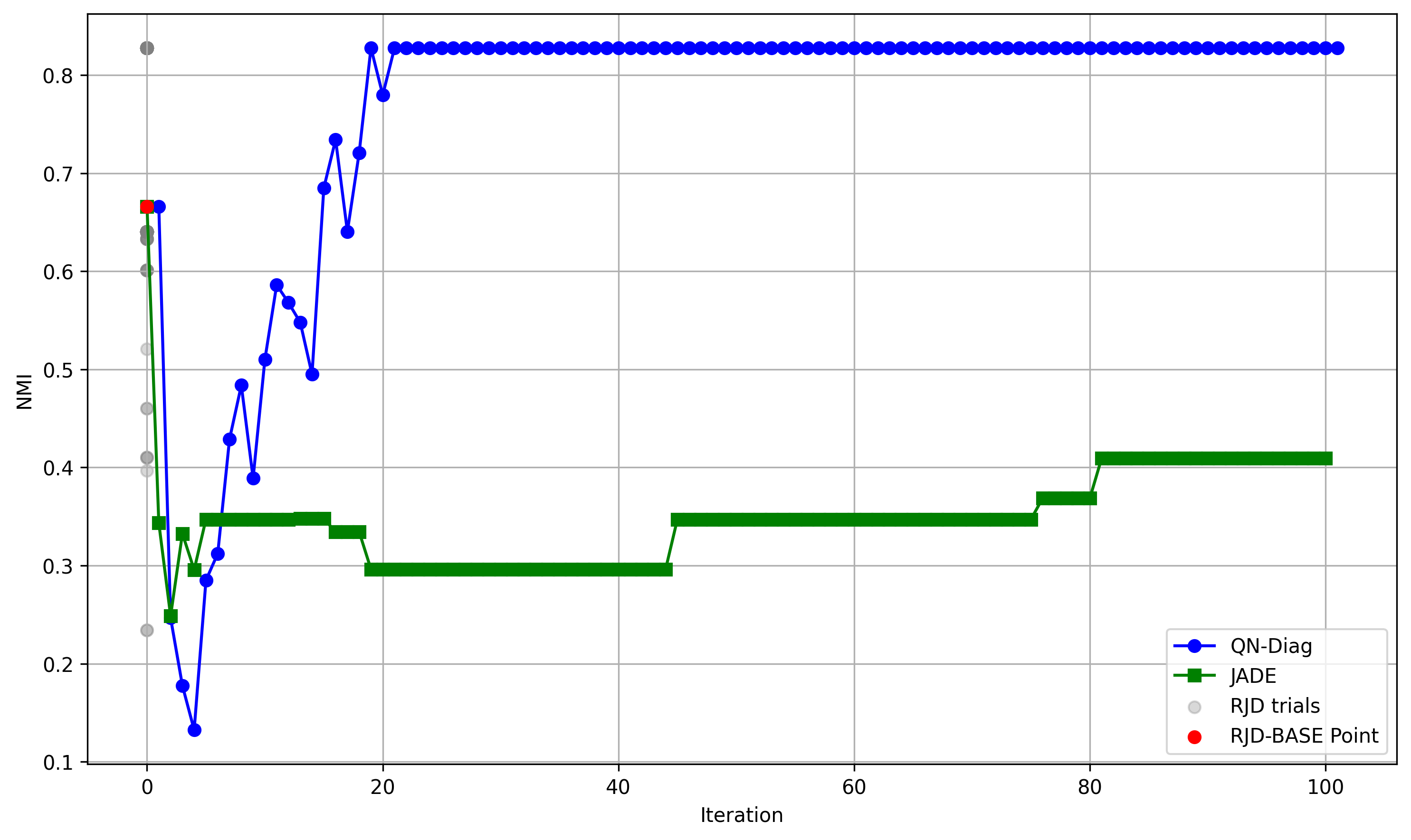}
    \caption{NMI convergence of QN-Diag and JADE when initialized with RJD-BASE on the Nutrimouse dataset.}
    \label{mouse_conv}
\end{figure}

\begin{figure}[h]
    \centering
    \includegraphics[width=3.5in]{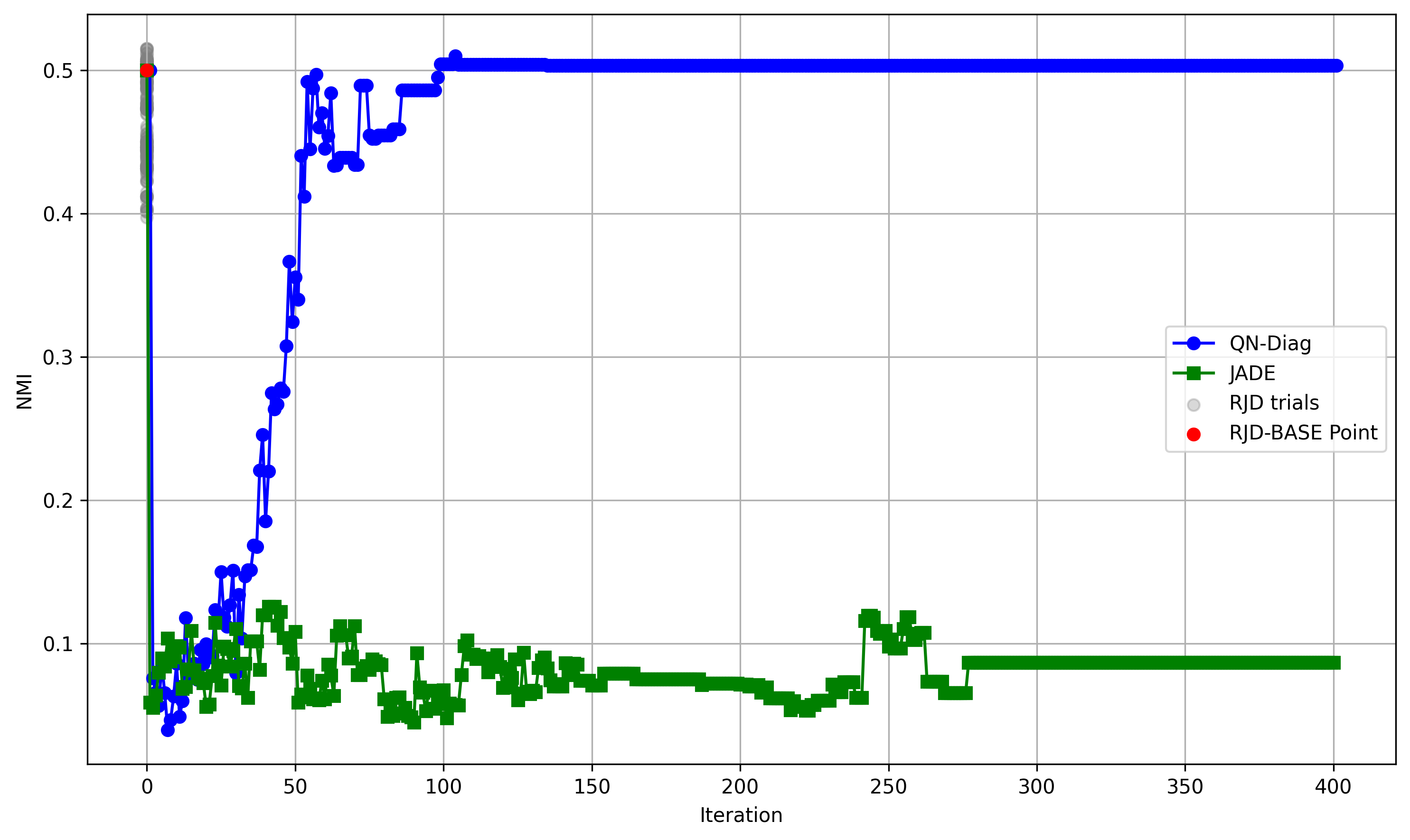}
    \caption{NMI convergence of QN-Diag and JADE when initialized with RJD-BASE on the MSRC dataset.}
    \label{msrc_conv}
\end{figure}


As illustrated in 
Figures \ref{synth_conv} - \ref{msrc_conv}, the first few JD iterations using either QN-Diag or JADE lead to  significant performance degradation for RJD-BASE on the SBM, Caltech-7, Digits Nutrimouse and MSRC datasets, suggesting that a full-basis diagonalization method such as QN-Diag or JADE could fails to improve RJD-BASE and potentially degrades the bottom-$k$ spectral subspace used for clustering. This supports the hypothesis that full-spectrum JD criteria are misaligned with the clustering objective, which depends only on the bottom-k subspace.

Regarding asymptotic behavior, we observe that JADE consistently fails to recover performance across all datasets. In contrast, QN-Diag eventually reaches an accuracy level similar to RJD-BASE on the MSRC dataset and even achieves improved accuracy on the Nutrimouse dataset. This suggests that the KL-divergence loss employed by QN-Diag may be better aligned with the underlying structure of these specific datasets, effectively mitigating the inherent misalignment between the full-spectrum objective and the clustering task. A formal investigation into how to systematically leverage this KL-divergence loss is left for future research.




\subsection{Impact of Modality Closeness}

We further investigate how the relative similarity between modalities affects the performance of RJD-BASE using the weighted SBM dataset. In this experiment, all modalities share the same ground-truth partition, but we explicitly control how close the graph Laplacians are to one another.

We first generate a reference Laplacian $\mathbf{L}_{\mathrm{ref}}$ from a weighted SBM. Additional modalities are then constructed as convex combinations of this reference Laplacian and independently generated SBM Laplacians. Specifically, for modality $i$, we define
\[
\mathbf{L}^{(i)}(\alpha)
=
(1-\alpha)\mathbf{L}_{\mathrm{ref}}
+
\alpha\,\mathbf{L}^{(i)}_{\mathrm{indep}},
\qquad \alpha \in [0,1].
\]
Here, $\alpha=0$ corresponds to identical modalities, while $\alpha=1$ produces fully independent views.

We quantify modality similarity using the average relative Frobenius distance to the reference:
\[
\frac{\|\mathbf{L}^{(i)} - \mathbf{L}_{\mathrm{ref}}\|_F}
{\|\mathbf{L}_{\mathrm{ref}}\|_F},
\]
averaged across modalities.

For each $\alpha \in \{0.0, 0.1, 0.2, 0.4, 0.6, 0.8, 1.0\}$, we run RJD-BASE with $T=200$ and report the best NMI and average over $20$ independent repetitions. Figure~\ref{effect_a} shows the resulting performance as a function of the average relative distance ($\alpha$ grows from left to right).

\begin{figure}[H]
    \centering
    \includegraphics[width=3.0in]{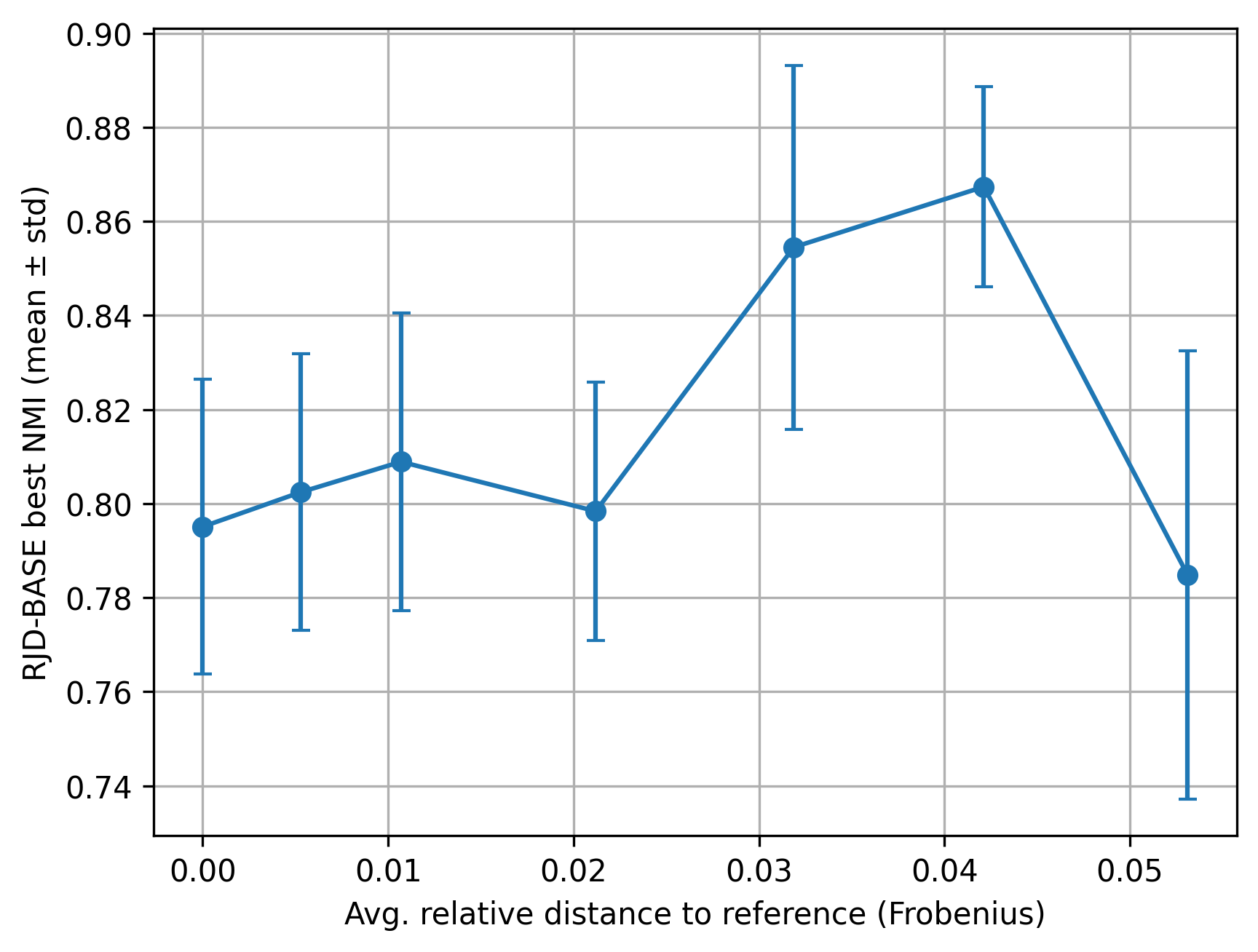}
    \caption{RJD-BASE performance (best NMI ± std) as a function of modality closeness in the weighted SBM. Modalities are generated as convex combinations of a reference Laplacian and independent SBM Laplacians, with closeness measured via average relative Frobenius distance. Peak performance occurs for moderately distinct but still correlated views.}
    \label{effect_a}
\end{figure}

As expected, we observe that when $\alpha=0$, all modalities are identical and RJD-BASE reduces to single-view spectral clustering and as $\alpha$ increases, performance improves, peaking for moderately distinct but still correlated modalities ($\alpha \approx 0.6$–$0.8$). This regime corresponds to complementary views that share structure while providing additional information. When $\alpha$ approaches $1$, modalities become too dissimilar and performance degrades, reflecting the diminishing benefit of aggregating unrelated views.

\subsection{Distribution Comparison}\label{subsec:dist_comp}

We compare the influence of the sampling distribution of $\mu$ on Algorithm~\ref{rjdbase}. The current sampling distribution is not uniform over the standard simplex, but rather concentrated toward the center \cite{willms2021uniform, smith2004sampling}. To sample uniformly over the standard simplex, we use a normalized exponential distribution \cite[Chapter XI.4]{distribution_book}. Specifically, we normalize a vector of i.i.d. exponential random variables with mean $1$:
\[ \mu = \frac{(\mu_1, \dots, \mu_n)}{\sum_{i=1}^n \mu_i}, \quad \mu_i \sim \text{Exp}(1) \text{ i.i.d.} \]
Empirically, the choice between the two distributions has a negligible impact on the final accuracy. To verify this, we run Algorithm~\ref{rjdbase} $50$ times with $T=200$ for both the centered and uniform sampling methods on the Digit dataset (Section~IV.\ref{subsec:digit}). For the concentrated distribution, we obtained an NMI mean of $0.6701$ (std $0.0019$), while the uniform distribution yielded an NMI mean of $0.6688$ (std $0.0012$). These results indicate that the choice of distribution has little effect on the clustering accuracy.

\subsection{Impact of $T$}
\label{subsec:T_sensitivity}

We further investigate the sensitivity of Algorithm~\ref{rjdbase} to the number of samples $T$. Since RJD-BASE selects the best embedding among $T$ random convex combinations of Laplacians, increasing $T$ improves the probability of sampling a high-quality combination. We quantify this by considering the Digits dataset and, for each $T \in \{5, 10, 20, 50, 100, 200, 500\}$, performing $20$ independent RJD-BASE runs and recording the best NMI per run. Figure~\ref{effect_t} shows the mean and standard deviation across runs. We observe that performance stabilizes quickly as $T$ increases. Values of $T \geq 100$ yield nearly identical mean performance with substantially reduced variability. This indicates that moderate values of $T$ are sufficient in practice and that RJD-BASE does not require excessively large sampling budgets to achieve stable accuracy.

\begin{figure}[H]
    \centering
    \includegraphics[width=3.0in]{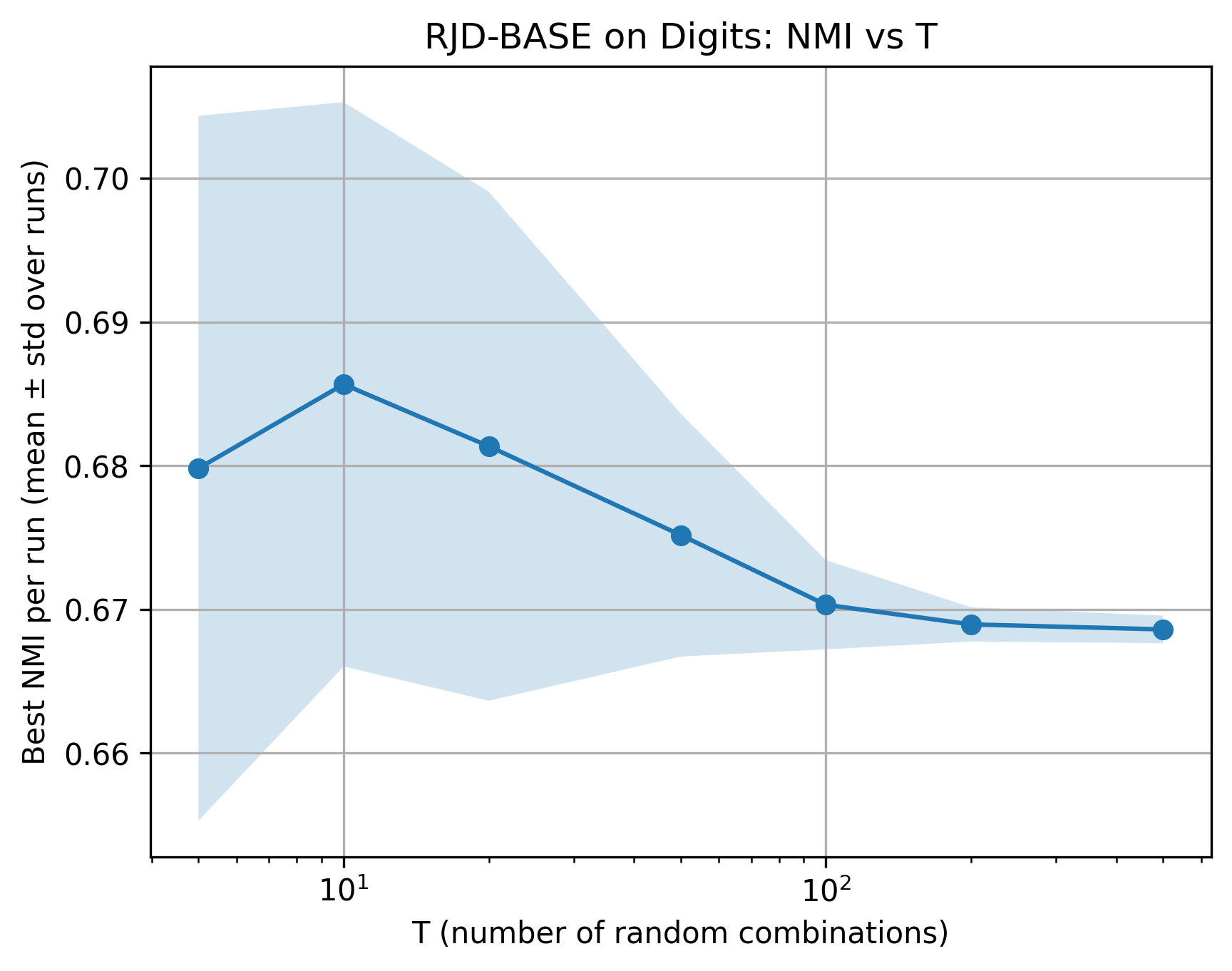}
    \caption{Mean $\pm$ std NMI after running RJD-BASE on Digits dataset with varying $T$.}
    \label{effect_t}
\end{figure}

}
  {\section{NUMERICAL EXPERIMENTS}
\label{exp}

Following standard practice, we evaluate clustering quality with normalized mutual information (NMI), which rescales mutual information by the label entropies so that scores lie in $[0,1]$ ($1$ being perfect agreement and $0$ being independence) \cite{scikit-learn}.

Because of space restrictions, we only show a few representative plots that are representative across methods and datasets. A full account - with expanded methodology and results, including plots and numerical metrics for every method on every dataset - can be found in~\cite{arxiv_version}.

\subsection{Direct Optimization of Smoothness Objectives}
\label{direct_opt}

We compare single-directional smoothness and BASE smoothness by directly optimizing the convex weights \(\pmb{\mu}\in\Delta^{m-1}\) on the combined operator \(\mathbf{L}(\pmb{\mu})=\sum_{i=1}^m \mu_i \mathbf{L}_i\). Using projected gradient ascent with Euclidean projection onto the simplex \cite{manopt}, we maximize either $\lambda_1(\mathbf{L}(\pmb{\mu}))$ (single-directional) or $\sum_{j=1}^{k}\lambda_j(\mathbf{L}(\pmb{\mu}))$ (BASE), starting from uniform \(\pmb{\mu}\). At each iterate we take the bottom-\(k\) eigenvectors of \(\mathbf{L}(\pmb{\mu})\), run \(k\)-means on the rows, and record NMI versus the objective value.

On Caltech-7 (Fig.~\ref{fig:direct_opt_caltech}), the single-directional run plateaus around 0.499 NMI, whereas the BASE objective reaches 0.530 NMI, indicating a benefit from targeting the full bottom-\(k\) spectrum. The results for the other datasets follow a similar pattern; see Section \ref{summary_table} as well as \cite{arxiv_version}. 

\begin{figure}[h]
    \centering
    \begin{minipage}[c]{0.48\linewidth}
        \includegraphics[width=\linewidth]{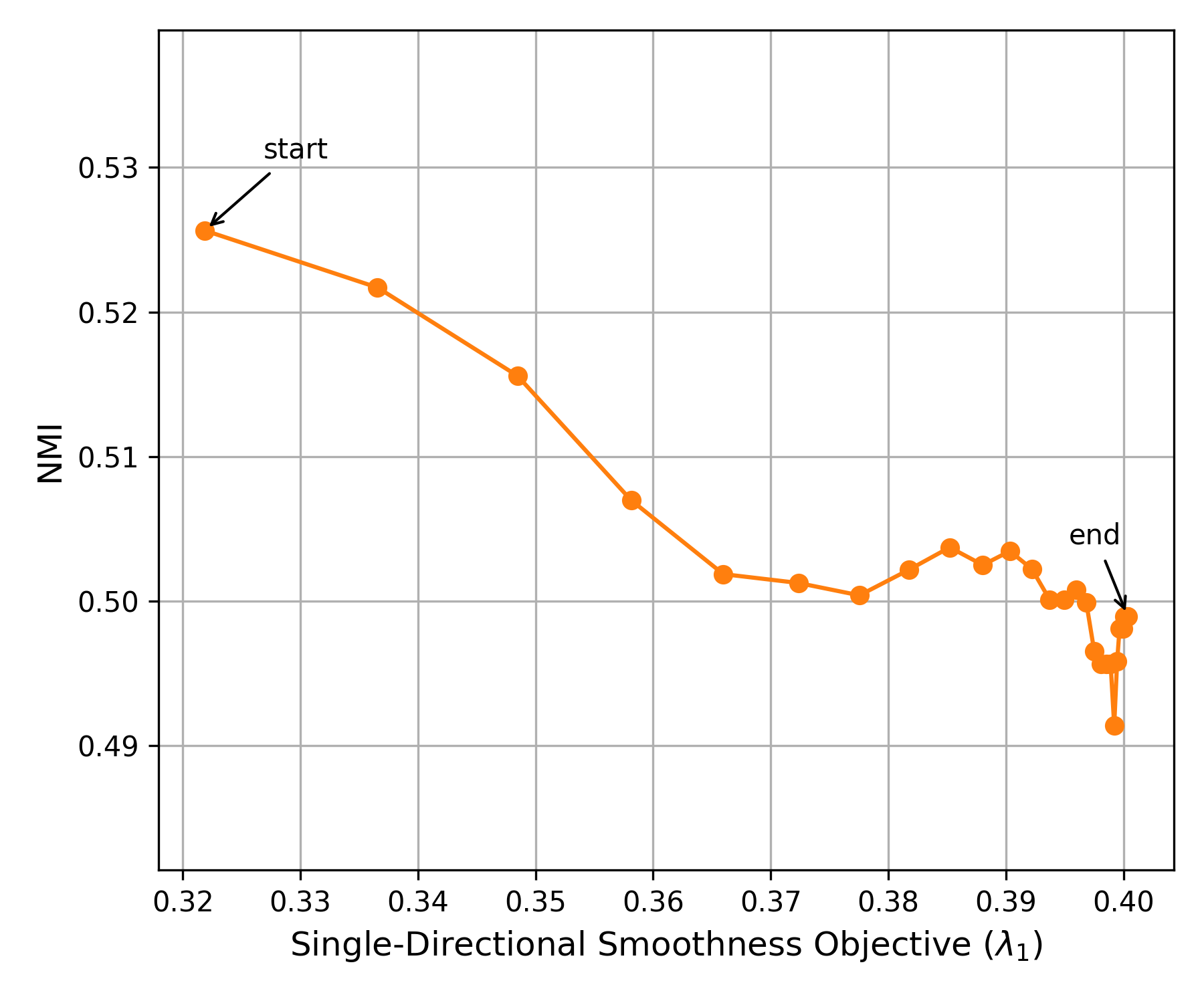}
        {\footnotesize (a) Single-directional: NMI vs.\ \(\lambda_1(\mathbf{L}(\pmb{\mu}))\).}
    \end{minipage}\hfill
    \begin{minipage}[c]{0.48\linewidth}
        \includegraphics[width=\linewidth]{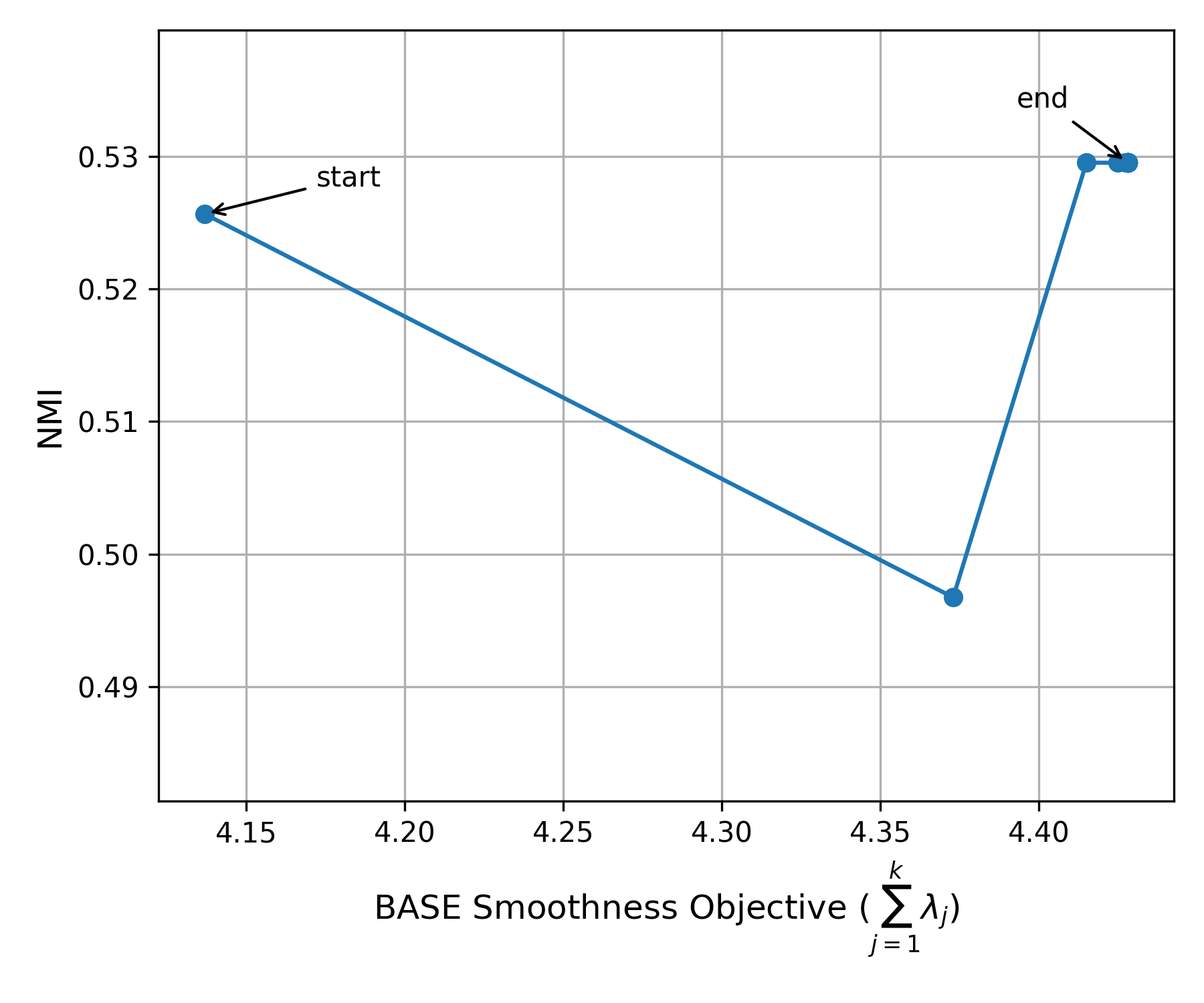}
        {\footnotesize (b) BASE: NMI vs.\ \(\sum_{j=1}^{k}\lambda_j(\mathbf{L}(\pmb{\mu}))\).}
    \end{minipage}
    \caption{Direct optimization on Caltech-7 with 30 iterations: NMI versus smoothness objective.}
    \label{fig:direct_opt_caltech}
\end{figure}

\subsection{RJD-BASE: Trial Landscape and Selection}
\label{rjd_landscape}

Direct optimization of the smoothness objectives provides high-quality embeddings, especially in the BASE objective case. Here, in an effort to simplify and parallelize, we evaluate whether the BASE smoothness objective can serve as an effective selection criterion across many RJD trials.

In Figure \ref{fig:rjd_scatter_all}, we run RJD-BASE for T=3000 and plot NMI against the BASE smoothness objective $\sum_{j=1}^{k} \lambda_j(\mathbf{L(\pmb{\mu})})$ for each RJD instance. We display four representative datasets due to space constraints and observe that RJD-BASE yields an above average RJD instance selection with respect to the end-goal NMI.

\begin{figure}[H]
    \centering
    \begin{minipage}[c]{0.48\linewidth}
        \centering
        \includegraphics[width=\linewidth]{pados3a.png}
        {\footnotesize (a) Weighted SBM}
    \end{minipage}\hfill
    \begin{minipage}[c]{0.48\linewidth}
        \centering
        \includegraphics[width=\linewidth]{pados3b.png}
        {\footnotesize (b) Caltech-7}
    \end{minipage}

    \vspace{0.5em}

    \begin{minipage}[c]{0.48\linewidth}
        \centering
        \includegraphics[width=\linewidth]{pados3c.png}
        {\footnotesize (c) Digits}
    \end{minipage}\hfill
    \begin{minipage}[c]{0.48\linewidth}
        \centering
        \includegraphics[width=\linewidth]{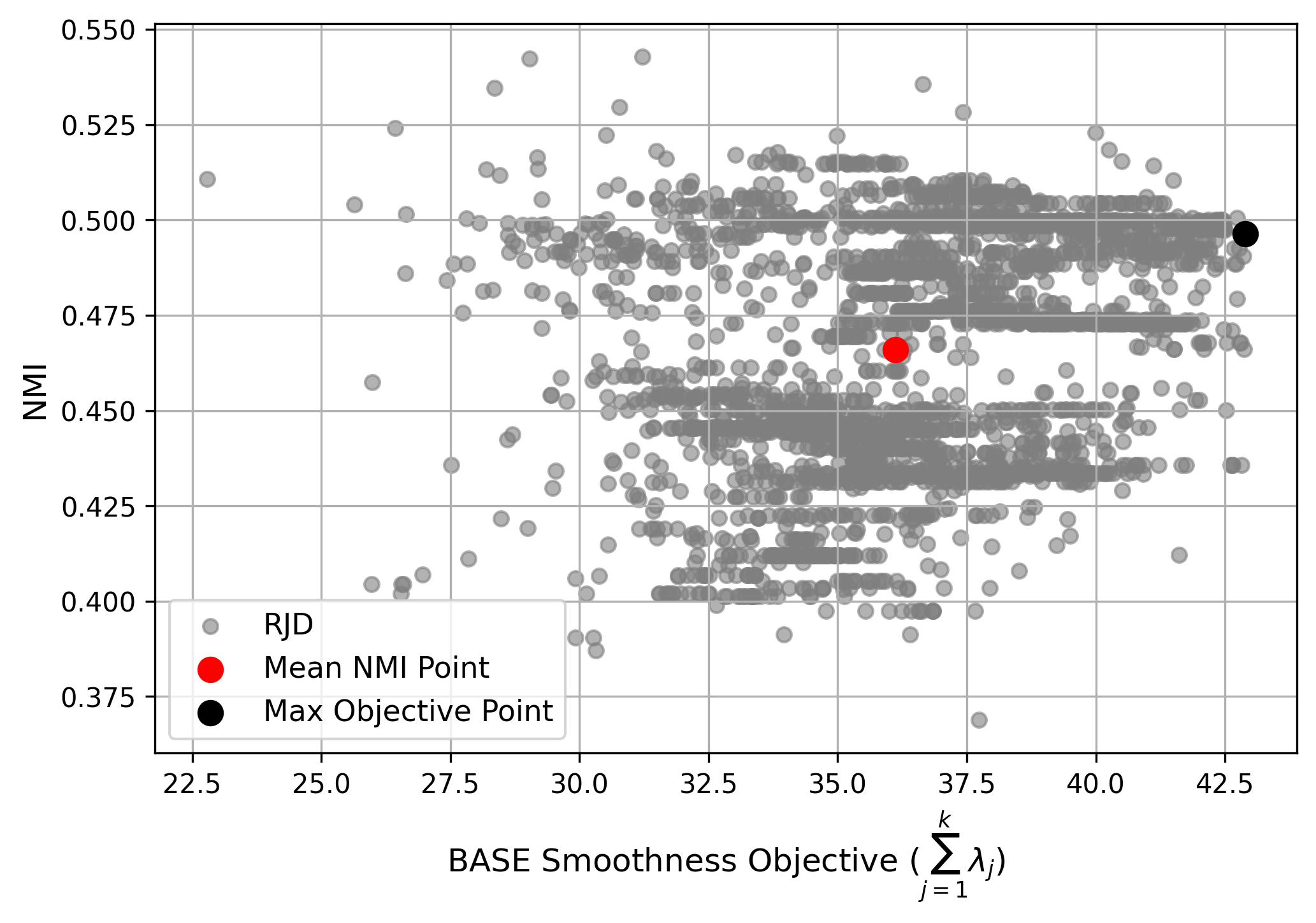}
        {\footnotesize (d) MSRC}
    \end{minipage}

    \caption{Scatter plot of NMI vs. BASE smoothness objective for 3000 independent RJD instance. Mean NMI point indicated in red and point maximizing BASE smoothness objective (i.e. that which would be selected by RJD-BASE) in black.}
    \label{fig:rjd_scatter_all}
\end{figure}

To more concretely quantify the effectiveness of this selection rule, we perform 1000 trials where in each trial we run RJD-BASE with $T=10$ and record whether the selected embedding’s NMI is above the global mean. This yields an empirical estimate of how often RJD-BASE beats a random draw in expectation even with very small $T$. The weighted SBM, Caltech-7, Digits, Nutrimouse, MSRC, and NGM datasets achieved 57\%, 66\%, 96\%, 76\%, 85\%, and 99\% above-average embeddings, respectively.

\subsection{RJD-BASE with QN-Diag and JADE}
\label{rjdbase_qn_diag_jade}

Approximate joint diagonalization methods (QN-Diag, JADE) optimize full-spectrum off-diagonal energy, whereas spectral clustering only focusses on the bottom-\(k\) subspace. It is natural to ask whether iteratively refining a good bottom-\(k\) embedding from \textbf{RJD-BASE} results in any advantage.

In Figure~\ref{fig:jd_refine_digits}, on the Digits dataset, we run RJD-BASE with \(T=200\), take the selected trial and use a full eigendecomposition as initialization to QN-Diag and JADE, tracking NMI in every iteration. It turns out that neither method improves on the RJD-BASE initialization. In fact, the JD methods generally degrade NMI and this holds as well when probing QN-Diag on the set of all above-average RJD trials. 

Similar observations are made for all other datasets; see Section~\ref{summary_table} as well as~\cite{arxiv_version} for numerical results on these datasets.

\begin{figure}[H]
    \centering
    \includegraphics[width=3.0in]{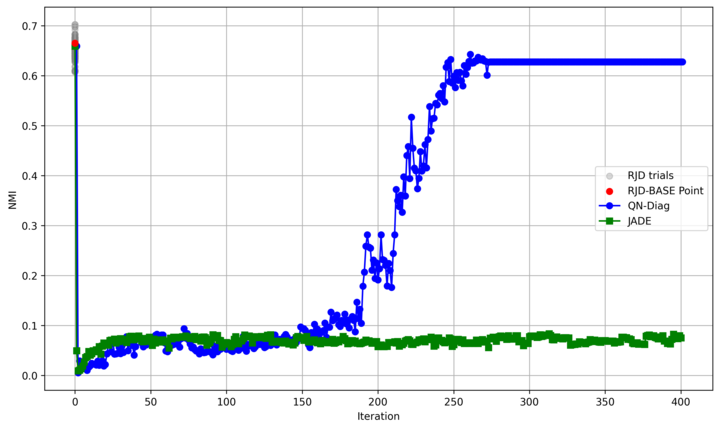}
    \caption{NMI vs.\ iteration learning curves for QN-Diag/JADE initialized with RJD-BASE on Digits dataset.}
    \label{fig:jd_refine_digits}
\end{figure}


\subsection{Impact of Modality Closeness}

In Section~V.D of \cite{arxiv_version}, we investigate how the relative similarity between modalities affects the performance of RJD-BASE using the weighted SBM dataset. In the experiment, all modalities share the same ground-truth partition, but we explicitly control how close the graph Laplacians are to one another. As expected, we find that when all modalities are identical, RJD-BASE reduces to single-view spectral clustering and as separation increases, performance improves, peaking for moderately distinct but still correlated modalities. This regime corresponds to complementary views that share structure while providing additional information. When modalities become too dissimilar, performance degrades, reflecting the diminishing benefit of aggregating unrelated views.}
\ifthenelse{\boolean{arxiv}}
  {\section{CLUSTERING EVALUATION}
\label{summary_table}

In this section, we consolidate results of all our clustering experiments and compare the performance of RJD-BASE to competing methods and baselines on our datasets. In each case, the goal is to compute a spectral embedding matrix $\mathbf{X} \in \mathbb{R}^{N \times k}$, and then apply $k$-means clustering to its rows. Each method is evaluated in terms of its final NMI after clustering.

The methods we compare are as follows:

\begin{itemize}
    \item \textbf{Single Laplacian (per modality):} Computed embedding from each individual Laplacian $\mathbf{L_i}$ diagonalization.
    
    \item \textbf{RJD Average:} The average NMI across 200 RJD trials, as in Section~\ref{exp}-\ref{rjd_landscape}.
    
    \item \textbf{RJD-BASE:} RJD with BASE smoothness objective selection, as in Section~\ref{exp}-\ref{rjd_landscape} ($T=200$).
    
    \item \textbf{QN-Diag:} Standard QN-Diag.
    
    \item \textbf{QN-Diag (RJD-BASE init.):} QN-Diag initialized with RJD-BASE, as in Section~\ref{exp}-\ref{rjdbase_qn_diag_jade}.

    \item \textbf{JADE:} Standard JADE.
    
    \item \textbf{JADE (RJD-BASE init.):} JADE initialized with RJD-BASE as in Section~\ref{exp}-\ref{rjdbase_qn_diag_jade}.
    
    \item \textbf{MVSC:} Standard MVSC.

    \item \textbf{CoReg-MVSC:} Standard CoReg-MVSC.
    
    \item \textbf{MV-KMeans:} MV-KMeans with k-means++ centroid initialization.

    \item \textbf{MV-SphKMeans:} Standard MV-SphKMeans.

    \item \textbf{Single-Directional Smoothness Objective:} The direct maximization of $\lambda_1(\mathbf{L(\pmb{\mu})})$ using projected gradient ascent, as presented in Section~\ref{exp}-\ref{direct_opt}.
    
    \item \textbf{BASE Smoothness Objective:} The direct maximization of $\sum_{j=1}^{k} \lambda_j(\mathbf{L(\pmb{\mu})})$ using projected gradient ascent, as presented in Section~\ref{exp}-\ref{direct_opt}.
\end{itemize}

See \cite{arxiv_version} for implementation details.

Table~\ref{tab:all_methods_nmi} summarizes clustering performance across all datasets and methods. 
RJD-BASE consistently outperforms the average RJD trial. Both JD algorithms show degradation relative to RJD-BASE, reinforcing the argument presented in Section~\ref{background}.\ref{jd_background}; full-spectrum JD criteria could be misaligned with the clustering
objective, which depends only on the bottom-k subspace. A notable exception occurs with JD using QN-Diag for Nutrimouse dataset, suggesting that  the KL-divergence loss employed by QN-Diag is potentially better aligned with  this data. Leveraging this KL-divergence  loss remains an objective for future research.

Classical multiview clustering methods yield mixed, but generally inferior results as compared to RJD-BASE. 
Direct optimization of our BASE smoothness objective achieves high NMI scores as expected but at a higher, non-parallelizable cost. The dashed lines indicate that the two-modality method was not applicable to the dataset.

\revise{To complement the quantitative comparison, Figure~\ref{fig:sbm_qual} visualizes the cluster assignments of all methods on the weighted SBM dataset. All panels share a fixed two-dimensional t-SNE layout computed from the input modalities, so that differences between panels reflect only the cluster assignments. RJD-BASE recovers most of the cluster structure, with remaining errors concentrated in the largest cluster, whereas most baselines mix several clusters within this region. Compared with the average RJD trial, the selected embedding also separates the smaller clusters more cleanly, illustrating the benefit of BASE selection.}

\begin{figure}[H]
  \centering
  \includegraphics[width=\columnwidth]{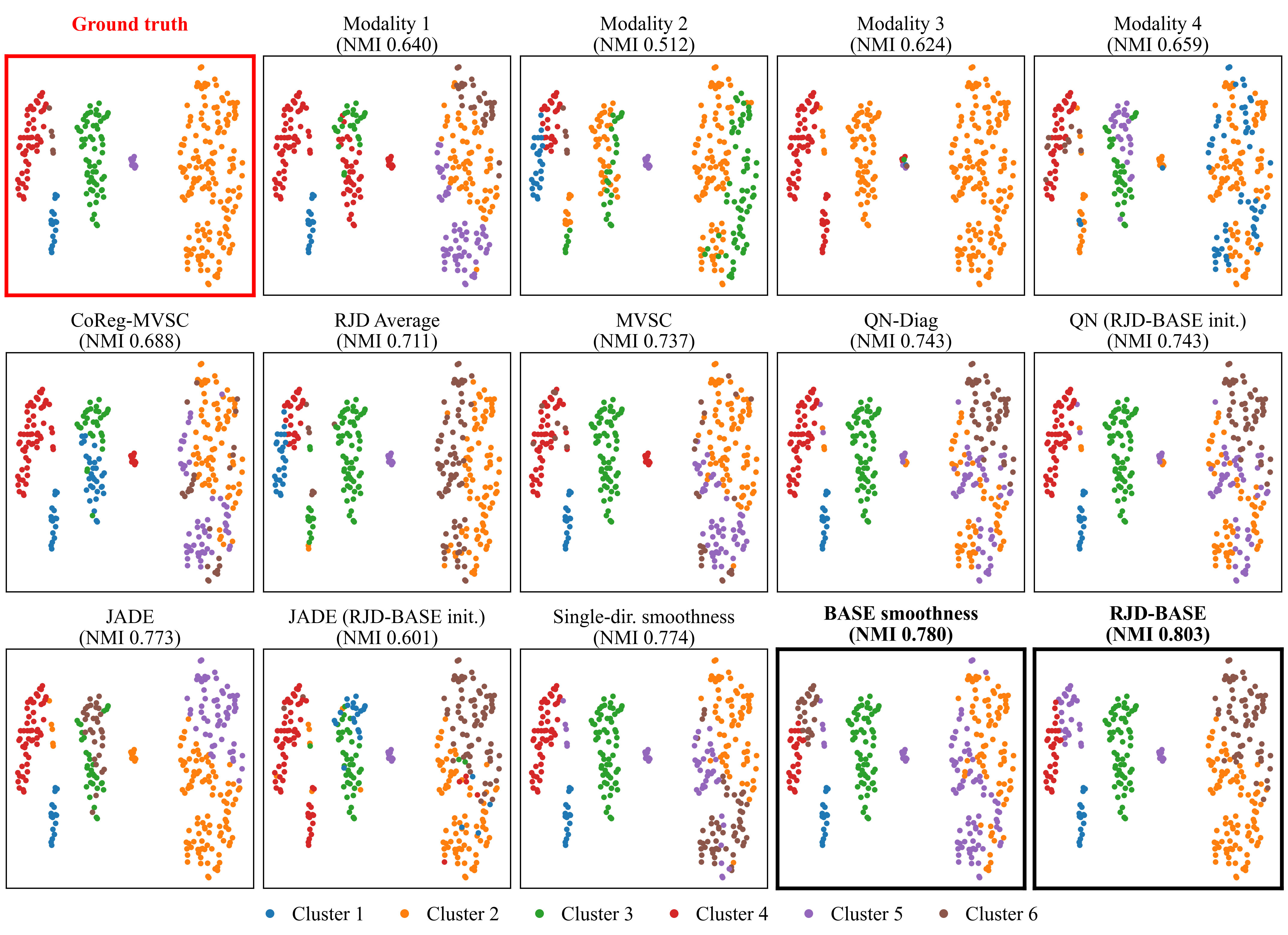}
  \caption{\revise{Qualitative clustering results on the weighted SBM dataset. All panels use the same t-SNE layout of the concatenated input modalities. Points are colored by the cluster assignments of each method such that a color change relative to the ground truth indicates a misassigned node.}}
  \label{fig:sbm_qual}
\end{figure}

\begin{table*}[t]
\centering
\caption{Clustering performance (NMI) across methods and datasets.`-' indicates method not applicable to dataset and `X' indicates too high runtime.} 
\label{tab:all_methods_nmi}
\setlength{\tabcolsep}{3pt}
\renewcommand{\arraystretch}{1.2}
\scriptsize
\begin{tabular}{|p{80pt}|p{80pt}|p{45pt}|p{45pt}|p{45pt}|p{45pt}|p{90pt}|}
\hline
\textbf{Method} & \textbf{Weighted SBM} & \textbf{Caltech-7} & \textbf{Digits} & \textbf{Nutrimouse} & \textbf{MSRC} & \textbf{Nonlinear Gaussian Mixture}\\
\hline
Single Laplacians & 
\begin{tabular}[c]{@{}c@{}}0.640 (1)\\0.512 (2)\\0.624 (3)\\0.659 (4)\end{tabular} & 
\begin{tabular}[c]{@{}c@{}}0.158 (Gabor)\\0.322 (Wavelet)\\0.355 (Centrist)\\0.421 (HOG)\\0.341 (GIST)\\0.507 (LBP)\end{tabular} &
\begin{tabular}[c]{@{}c@{}}0.665 (DCT)\\0.607 (Patch)\end{tabular} &
\begin{tabular}[c]{@{}c@{}}0.234 (Gene)\\0.666 (Lipid)\end{tabular} &
\begin{tabular}[c]{@{}c@{}}0.399 (LBP)\\0.332 (Gray)\\ 0.437 (Gabor)\\0.437 (Edge)\end{tabular} &
\begin{tabular}[c]{@{}c@{}}0.843 (1)\\0.839 (2)\end{tabular}\\
\hline
RJD Average & 0.711 ($\pm$0.012) & 0.491 ($\pm$0.002) & 0.650 ($\pm$0.001) & 0.666 ($\pm$0.011) & 0.469 ($\pm$0.001) & 0.846\\
\hline
\textbf{RJD-BASE} & \textbf{0.803} & \textbf{0.531} & \textbf{0.665} & \textbf{0.667} & \textbf{0.500} & \textbf{0.850}\\
\hline
QN-Diag & 0.743 & 0.285 & 0.627 & 0.827 & 0.503 & X\\
\hline
QN-Diag (RJD-BASE init.) & 0.743 & 0.274 & 0.627 & 0.827 & 0.500 & X\\
\hline
JADE & 0.773 & 0.415 & 0.650 & 0.640 & 0.440 & X\\
\hline
JADE (RJD-BASE init.) & 0.601 & 0.024 & 0.075 & 0.325 & 0.087 & X\\
\hline
MVSC & 0.737 & 0.476 & 0.661 & 0.601 & 0.520 & 0.841\\
\hline
CoReg-MVSC & 0.688 & 0.431 & 0.679 & 0.372 & 0.479 & 0.817\\
\hline
MV-KMeans & -- & -- & 0.489 & 0.613 & -- & 0.841\\
\hline
MV-SphKMeans & -- & -- & 0.528 & 0.727 & -- & 0.841\\
\hline
Single-Dir. Smoothness & 0.774 & 0.499 & 0.682 & 0.601 & 0.499 & 0.846\\
\hline
\textbf{BASE Smoothness} & \textbf{0.780} & \textbf{0.530} & \textbf{0.682} & \textbf{0.601} & \textbf{0.476} & \textbf{0.846} \\
\hline
\end{tabular}
\end{table*}

Table~\ref{tab:wall_clock_times} reports approximate wall-clock times real elapsed time for all methods on each dataset. 
All runs were executed uniformly under an identical software environment, using a single process with default library threading; no GPU or distributed computation was used. 
RJD-BASE was run with $T{=}200$ without parallelization. The reported runtimes measure the core joint diagonalization or eigenvalue 
optimization step, the k-means clustering, and exclude graph construction. The results highlight the practical efficiency of RJD-BASE relative to full-spectrum diagonalization.

\begin{table*}[t]
\centering
\caption{Approximate wall-clock runtime for each method and dataset. RJD-BASE uses $T=200$ without parallelization. `-' indicates method not applicable to dataset and `X' indicates too high runtime.}
\label{tab:wall_clock_times}
\setlength{\tabcolsep}{3pt}
\renewcommand{\arraystretch}{1.2}
\scriptsize
\centering
\begin{tabular}{|p{80pt}|p{80pt}|p{45pt}|p{45pt}|p{45pt}|p{45pt}|p{90pt}|}
\hline
\textbf{Method} & \textbf{Weighted SBM} & \textbf{Caltech-7} & \textbf{Digits} & \textbf{Nutrimouse} & \textbf{MSRC} & \textbf{Nonlinear Gaussian Mixture}\\
\hline
QN-Diag (200 iters) & $\sim$2 min & $\sim$15 min & $\sim$35 min & $\sim$10 sec & $\sim$5 min & X\\
\hline
JADE (200 iters) & $\sim$4 min & $\sim$900 min & $\sim$7000 min & $\sim$10 sec & $\sim$5 min & X\\
\hline
\textbf{RJD-BASE} ($T=200$) & \textbf{\boldmath$\sim$1 s}
 & \textbf{\boldmath$\sim$20 s}
 & \textbf{\boldmath$\sim$90 s} & \textbf{$\sim$1 s} & \textbf{$\sim$3 s} & \textbf{$\sim$3 min}
 \\
\hline
MVSC & $\sim$1 s & $\sim$2 min & $\sim$2 min & $\sim$1 s & $\sim$4 min & $\sim$4 min\\
\hline
CoReg-MVSC & $\sim$1 s & $\sim$30 s & $\sim$30 s & $\sim$1 s & $\sim$1 min & $\sim$30 s\\
\hline
MV-KMeans & -- & -- & $\sim$1 s & $\sim$1 s & -- & $\sim$1 s\\
\hline
MV-SphKMeans & -- & -- & $\sim$1 s & $\sim$1 s & -- & $\sim$1 s\\
\hline
Single-Dir. Smoothness (30 iters) & $\sim$5 s & $\sim$30 s & $\sim$10 min & $\sim$2 s & $\sim$30 sec & $\sim$10 min\\
\hline
\textbf{BASE Smoothness} (30 iters) & \textbf{\boldmath$\sim$5 s}
 & \textbf{\boldmath$\sim$30 s}
 & \textbf{\boldmath$\sim$10 min} & \textbf{$\sim$2 s} & \textbf{$\sim$30 sec} & \textbf{$\sim$10 min}
 \\
\hline
\end{tabular}
\end{table*}
}
  {\section{CLUSTERING EVALUATION}
\label{summary_table}

In this section, we consolidate results of all our clustering experiments and compare the performance of RJD-BASE to competing methods and baselines on our datasets. In each case, the goal is to compute a spectral embedding matrix $\mathbf{X} \in \mathbb{R}^{N \times k}$, and then apply $k$-means clustering to its rows. Each method is evaluated in terms of its final NMI after clustering.

The methods we compare are as follows:

\begin{itemize}
    \item \textbf{Single Laplacian (per modality):} Computed embedding from each individual Laplacian $\mathbf{L_i}$ diagonalization.
    
    \item \textbf{RJD Average:} The average NMI across 200 RJD trials, as in Section~\ref{exp}-\ref{rjd_landscape}.
    
    \item \textbf{RJD-BASE:} RJD with BASE smoothness objective selection, as in Section~\ref{exp}-\ref{rjd_landscape} ($T=200$).
    
    \item \textbf{QN-Diag:} Standard QN-Diag.
    
    \item \textbf{QN-Diag (RJD-BASE init.):} QN-Diag initialized with RJD-BASE, as in Section~\ref{exp}-\ref{rjdbase_qn_diag_jade}.

    \item \textbf{JADE:} Standard JADE.
    
    \item \textbf{JADE (RJD-BASE init.):} JADE initialized with RJD-BASE as in Section~\ref{exp}-\ref{rjdbase_qn_diag_jade}.
    
    \item \textbf{MVSC:} Standard MVSC.

    \item \textbf{CoReg-MVSC:} Standard CoReg-MVSC.
    
    \item \textbf{MV-KMeans:} MV-KMeans with k-means++ centroid initialization.

    \item \textbf{MV-SphKMeans:} Standard MV-SphKMeans.

    \item \textbf{Single-Directional Smoothness Objective:} The direct maximization of $\lambda_1(\mathbf{L(\pmb{\mu})})$ using projected gradient ascent, as presented in Section~\ref{exp}-\ref{direct_opt}.
    
    \item \textbf{BASE Smoothness Objective:} The direct maximization of $\sum_{j=1}^{k} \lambda_j(\mathbf{L(\pmb{\mu})})$ using projected gradient ascent, as presented in Section~\ref{exp}-\ref{direct_opt}.
\end{itemize}

See \cite{arxiv_version} for implementation details.

Table~\ref{tab:all_methods_nmi} summarizes clustering performance across all datasets and methods. 
RJD-BASE consistently outperforms the average RJD trial. Both JD algorithms show degradation relative to RJD-BASE, reinforcing the argument presented in Section~\ref{background}.\ref{jd_background}; full-spectrum JD criteria could be misaligned with the clustering
objective, which depends only on the bottom-k subspace. A notable exception occurs with JD using QN-Diag for Nutrimouse dataset, suggesting that  the KL-divergence loss employed by QN-Diag is potentially better aligned with  this data. Leveraging this KL-divergence  loss remains an objective for future research.

Classical multiview clustering methods yield mixed, but generally inferior results as compared to RJD-BASE. 
Direct optimization of our BASE smoothness objective achieves high NMI scores as expected but at a higher, non-parallelizable cost. The dashed lines indicate that the two-modality method was not applicable to the dataset.

\begin{table*}[t]
\centering
\caption{Clustering performance (NMI) across methods and datasets.`-' indicates method not applicable to dataset and `X' indicates too high runtime.} 
\label{tab:all_methods_nmi}
\setlength{\tabcolsep}{3pt}
\renewcommand{\arraystretch}{1.2}
\scriptsize
\begin{tabular}{|p{80pt}|p{80pt}|p{45pt}|p{45pt}|p{45pt}|p{45pt}|p{90pt}|}
\hline
\textbf{Method} & \textbf{Weighted SBM} & \textbf{Caltech-7} & \textbf{Digits} & \textbf{Nutrimouse} & \textbf{MSRC} & \textbf{Nonlinear Gaussian Mixture}\\
\hline
Single Laplacians & 
\begin{tabular}[c]{@{}c@{}}0.640 (1)\\0.512 (2)\\0.624 (3)\\0.659 (4)\end{tabular} & 
\begin{tabular}[c]{@{}c@{}}0.158 (Gabor)\\0.322 (Wavelet)\\0.355 (Centrist)\\0.421 (HOG)\\0.341 (GIST)\\0.507 (LBP)\end{tabular} &
\begin{tabular}[c]{@{}c@{}}0.665 (DCT)\\0.607 (Patch)\end{tabular} &
\begin{tabular}[c]{@{}c@{}}0.234 (Gene)\\0.666 (Lipid)\end{tabular} &
\begin{tabular}[c]{@{}c@{}}\0.399 (LBP)\\0.332 (Gray)\\ 0.437 (Gabor)\\0.437 (Edge)\end{tabular} &
\begin{tabular}[c]{@{}c@{}}0.843 (1)\\0.839 (2)\end{tabular}\\
\hline
RJD Average & 0.711 ($\pm$0.012) & 0.491 ($\pm$0.002) & 0.650 ($\pm$0.001) & 0.666 ($\pm$0.011) & 0.469 ($\pm$0.001) & 0.846\\
\hline
\textbf{RJD-BASE} & \textbf{0.803} & \textbf{0.531} & \textbf{0.665} & \textbf{0.667} & \textbf{0.500} & \textbf{0.850}\\
\hline
QN-Diag & 0.743 & 0.285 & 0.627 & 0.827 & 0.503 & X\\
\hline
QN-Diag (RJD-BASE init.) & 0.743 & 0.274 & 0.627 & 0.827 & 0.500 & X\\
\hline
JADE & 0.773 & 0.415 & 0.650 & 0.640 & 0.440 & X\\
\hline
JADE (RJD-BASE init.) & 0.601 & 0.024 & 0.075 & 0.325 & 0.087 & X\\
\hline
MVSC & 0.737 & 0.476 & 0.661 & 0.601 & 0.520 & 0.841\\
\hline
CoReg-MVSC & 0.688 & 0.431 & 0.679 & 0.372 & 0.479 & 0.817\\
\hline
MV-KMeans & -- & -- & 0.489 & 0.613 & -- & 0.841\\
\hline
MV-SphKMeans & -- & -- & 0.528 & 0.727 & -- & 0.841\\
\hline
Single-Dir. Smoothness & 0.774 & 0.499 & 0.682 & 0.601 & 0.499 & 0.846\\
\hline
\textbf{BASE Smoothness} & \textbf{0.780} & \textbf{0.530} & \textbf{0.682} & \textbf{0.601} & \textbf{0.476} & \textbf{0.846} \\
\hline
\end{tabular}
\end{table*}

Table~\ref{tab:wall_clock_times} reports approximate wall-clock times real elapsed time for all methods on each dataset. 
All runs were executed uniformly under an identical software environment, using a single process with default library threading; no GPU or distributed computation was used. 
RJD-BASE was run with $T{=}200$ without parallelization. The reported runtimes measure the core joint diagonalization or eigenvalue 
optimization step, the k-means clustering, and exclude graph construction. The results highlight the practical efficiency of RJD-BASE relative to full-spectrum diagonalization.

\begin{table*}[t]
\centering
\caption{Approximate wall-clock runtime for each method and dataset. RJD-BASE uses $T=200$ without parallelization. `-' indicates method not applicable to dataset and `X' indicates too high runtime.}
\label{tab:wall_clock_times}
\setlength{\tabcolsep}{3pt}
\renewcommand{\arraystretch}{1.2}
\scriptsize
\centering
\begin{tabular}{|p{80pt}|p{80pt}|p{45pt}|p{45pt}|p{45pt}|p{45pt}|p{90pt}|}
\hline
\textbf{Method} & \textbf{Weighted SBM} & \textbf{Caltech-7} & \textbf{Digits} & \textbf{Nutrimouse} & \textbf{MSRC} & \textbf{Nonlinear Gaussian Mixture}\\
\hline
QN-Diag (200 iters) & $\sim$2 min & $\sim$15 min & $\sim$35 min & $\sim$10 sec & $\sim$5 min & X\\
\hline
JADE (200 iters) & $\sim$4 min & $\sim$900 min & $\sim$7000 min & $\sim$10 sec & $\sim$5 min & X\\
\hline
\textbf{RJD-BASE} ($T=200$) & \textbf{\boldmath$\sim$1 s}
 & \textbf{\boldmath$\sim$20 s}
 & \textbf{\boldmath$\sim$90 s} & \textbf{$\sim$1 s} & \textbf{$\sim$3 s} & \textbf{$\sim$3 min}
 \\
\hline
MVSC & $\sim$1 s & $\sim$2 min & $\sim$2 min & $\sim$1 s & $\sim$4 min & $\sim$4 min\\
\hline
CoReg-MVSC & $\sim$1 s & $\sim$30 s & $\sim$30 s & $\sim$1 s & $\sim$1 min & $\sim$30 s\\
\hline
MV-KMeans & -- & -- & $\sim$1 s & $\sim$1 s & -- & $\sim$1 s\\
\hline
MV-SphKMeans & -- & -- & $\sim$1 s & $\sim$1 s & -- & $\sim$1 s\\
\hline
Single-Dir. Smoothness (30 iters) & $\sim$5 s & $\sim$30 s & $\sim$10 min & $\sim$2 s & $\sim$30 sec & $\sim$10 min\\
\hline
\textbf{BASE Smoothness} (30 iters) & \textbf{\boldmath$\sim$5 s}
 & \textbf{\boldmath$\sim$30 s}
 & \textbf{\boldmath$\sim$10 min} & \textbf{$\sim$2 s} & \textbf{$\sim$30 sec} & \textbf{$\sim$10 min}
 \\
\hline
\end{tabular}
\end{table*}
}
\section{CONCLUSION}
\label{conc}

We proposed a new framework for multimodal spectral clustering that introduces randomization as a core component of the embedding generation process and pairs it with a principled, task-aligned selection rule. By sampling random convex combinations of modality-specific Laplacians and evaluating them using a novel $k$-dimensional smoothness criterion - \textbf{Bottom-$k$ Aggregated Spectral Energy (BASE)} - our method efficiently explores the space of spectral embeddings without requiring optimization, initialization, or iterative refinement.

Our experiments demonstrate that the proposed algorithm,  \textbf{RJD-BASE}, reliably selects high-quality embeddings across synthetic and real-world datasets and effectively integrates information from different modalities to improve clustering. It outperforms classical techniques while operating at a low computational cost.

We believe these findings suggest a broader potential for randomized, selection-based strategies in spectral learning, possibly sparking future exploration of principled selection criteria, hybrid randomized schemes, and applications beyond clustering.


\section*{ACKNOWLEDGMENT}
The second author gratefully acknowledges support from the MIT International Science and Technology Initiatives.

\bibliographystyle{IEEEtran}
\bibliography{references}




\vfill\pagebreak

\end{document}